\documentclass[12pt, a4paper]{amsart}

\calclayout

\usepackage[english]{babel}
\usepackage[utf8]{inputenc}
\usepackage{amsmath}
\usepackage{graphicx}
\usepackage[colorinlistoftodos]{todonotes}
\usepackage{amssymb,amscd,latexsym,epsfig,xy}
\usepackage[mathscr]{euscript}
\usepackage{amsthm}
\usepackage{tikz}
\usepackage{MnSymbol}
\usetikzlibrary{matrix,arrows,decorations.pathmorphing}
\usepackage{hyperref}
\usepackage{enumerate}
\usepackage{float}
\usepackage{tikz-cd}
\usepackage{relsize}
\usepackage{mathtools}

\newtheorem{defn}{Definition}[section]
\newtheorem{thm}[defn]{Theorem}
\newtheorem{lem}[defn]{Lemma}
\newtheorem{cor}[defn]{Corollary}
\newtheorem{conj}[defn]{Conjecture}
\newtheorem{prop}[defn]{Proposition}

\newcommand\C{\mathbb C}
\newcommand\E{\mathbb E}

\newcommand\NN{\mathbb N}

\newcommand\PP{\mathbb P}
\newcommand\R{\mathbb R}
\newcommand\Q{\mathbb Q}
\newcommand\V{\mathbb V}
\newcommand\Z{\mathbb Z}
\newcommand\cA{\mathcal A}

\newcommand\cC{\mathcal C}

\newcommand\cH{\mathcal H}

\newcommand\cM{\mathcal M}
\newcommand\cN{\mathcal N}
\newcommand\cO{\mathcal O}

\newcommand\cY{\mathcal Y}

\newcommand\fd{\mathfrak d}
\newcommand\fD{\mathfrak D}
\newcommand\fg{\mathfrak g}
\newcommand\fm{\mathfrak m}
\newcommand\fr{\mathfrak r}

\newcommand\Aut{\operatorname{Aut}}

\newcommand\Spf{\operatorname{Spf}\,}

\newcommand\Hom{\operatorname{Hom}}

\newcommand\id{\operatorname{id}\,}

\newcommand\Tr{\operatorname{Tr}}

\newcommand\Ad{\operatorname{Ad}}

\newcommand\ev{\operatorname{ev}}
\newcommand\lcm{\operatorname{lcm}}
\newcommand\rk{\operatorname{rk}}
\newcommand\ord{\operatorname{ord}}
\newcommand\trop{\mathrm{trop}}
\newcommand\virt{\mathrm{virt}}
\newcommand\pt{\mathrm{pt}}
\newcommand\spl{\mathrm{split}}
\newcommand\tw{\mathrm{tw}}

\title[The quantum tropical vertex]{The quantum tropical vertex}

\author{Pierrick Bousseau}

\date{}

\begin{document}

\begin{abstract}
Gross, Pandharipande, and Siebert have shown that the 
2-dimensional Kontsevich-Soibelman scattering diagrams
compute certain genus zero log Gromov-Witten invariants of 
log Calabi-Yau surfaces. We show that the $q$-refined 2-dimensional Kontsevich-Soibelman scattering diagrams compute, 
after the change of variables $q=e^{i \hbar}$, 
generating series of certain higher genus log Gromov-Witten invariants of log Calabi-Yau surfaces.

This result provides a mathematically rigorous realization of the physical derivation of the 
refined wall-crossing formula from topological string theory proposed by Cecotti-Vafa, 
and in particular can be viewed as a non-trivial mathematical check of the connection suggested by Witten 
between higher genus open A-model and Chern-Simons theory.

We also prove some new BPS integrality results and propose some other BPS integrality conjectures.

\end{abstract}

\maketitle

\setcounter{tocdepth}{1}

\tableofcontents

\thispagestyle{empty}

\section*{Introduction}

\subsection*{Statements of results}
We start by giving slightly imprecise versions of the main results of the present paper. 
For us, a log Calabi-Yau surface is a pair $(Y,D)$, where $Y$ is a smooth complex projective surface 
and $D$ is a reduced effective normal crossing anticanonical divisor on $Y$. 
A log Calabi-Yau surface $(Y,D)$ has maximal boundary if $D$ is singular.

\begin{thm} \label{main_thm_ch2}
The functions attached to the rays of the 
$q$-deformed 2-dimensional Kontsevich-Soibelman scattering diagrams are, 
after the change of variables $q=e^{i \hbar}$, generating series of higher genus log Gromov-Witten invariants---with
maximal tangency condition and insertion of the top lambda class---of log Calabi-Yau surfaces with maximal boundary.
\end{thm}  

A precise version of Theorem
\ref{main_thm_ch2} is given by 
Theorems \ref{precise_main_thm_ch2}
and \ref{precise_main_thm_orbifold}
in \mbox{Section
\ref{section_main_result}}. 

\begin{thm} \label{main_thm_integ}
Higher genus log Gromov-Witten invariants--with maximal tangency condition and insertion 
of the top lambda class--of log Calabi-Yau surfaces with maximal boundary satisfy 
an Ooguri-Vafa/ open BPS integrality property.
\end{thm}

A precise version of Theorem 
\ref{main_thm_integ} is given by 
Theorem 
\ref{thm_integ} in Section 
\ref{section_integrality}.

We also formulate a new conjecture.

\begin{conj} \label{conj_main}
Higher genus relative Gromov-Witten invariants-with maximal tangency condition and insertion of the 
top lambda class--of a del Pezzo surface $S$ relative to a smooth anticanonical divisor 
are related to refined counts of dimension one stable sheaves on the local Calabi-Yau 3-fold $\mathrm{Tot} K_S$, 
the
total space of the canonical line bundle of
$S$.
\end{conj}

A precise version of Conjecture
\ref{conj_main} is given by 
Conjecture  \ref{conj_del_pezzo} in
Section
\ref{section_del_pezzo}.

\subsection*{Context and motivations}

\subsubsection*{\textbf{SYZ}} The Strominger-Yau-Zaslow 
\cite{MR1429831}
picture of
mirror symmetry suggests a two-step construction of the mirror of a 
Calabi-Yau variety admitting a
Lagrangian torus fibration: 
first, construct the ``semi-flat''
mirror by dualizing the non-singular 
torus fibers; second, correct the complex structure of the ``semi-flat'' mirror
such that it extends across the locus of singular fibers. It is expected--see
\cite{MR1429831} and Fukaya \cite{MR2131017}--
that 
the corrections involved in the second step are determined by some counts of
holomorphic discs in the original variety with boundary on torus fibers.

\subsubsection*{\textbf{KS}}

In dimension two and with at most
nodal singular fibers in the torus fibration, Kontsevich and Soibelman
\cite{MR2181810} had the insight that algebraic self-consistency constraints on the corrections 
were strong enough to determine these corrections uniquely. More precisely,
they reduced the problem to an algebraic 
computation of commutators in a group of formal families of symplectomorphisms of the $2$-dimensional algebraic torus.

This algebraic formalism, graphically encoded
under the form of scattering diagrams, was
generalized and extended to higher dimensions by Gross-Siebert
\cite{MR2846484} and plays an essential role in the Gross-Siebert algebraic approach to mirror symmetry.

\subsubsection*{\textbf{GPS}}

In 
\cite{MR2667135}, Gross, Pandharipande and Siebert 
made some progress in connecting 
the original enumerative expectation and the
algebraic recipe of scattering diagrams.
They showed that the 2-dimensional 
Kontsevich-Soibelman scattering diagrams 
indeed have an enumerative meaning: they compute some genus $0$ log Gromov-Witten invariants 
of some log Calabi-Yau surfaces with maximal boundary, ie complements of a singular normal crossing anticanonical divisor in a smooth projective surface.

This agrees with the original expectation because these geometries admit Lagrangian torus fibrations 
and these genus $0$ log Gromov-Witten invariants should be thought of as algebraic definitions of some counts of holomorphic discs with boundary on Lagrangian torus fibers.
For some symplectic approaches, relating
counts of holomorphic discs in hyperk\"ahler manifolds of real dimension 4 
and the Konstevich-Soibelman wall-crossing formula, we refer to the works of 
Lin
\cite{lin2017correspondence} and Iacovino
\cite{iacovino2017ks}. 

The combination of
2-dimensional scattering diagrams with 
their enumerative interpretation given 
by Gross, Pandharipande and Siebert
\cite{MR2667135}
is the main tool in the Gross-Hacking-Keel
\cite{MR3415066} construction of mirrors for log Calabi-Yau surfaces with maximal boundary.

\subsubsection*{\textbf{Higher genus GPS = refined KS}}

In  \cite[Section 11.8]{MR2181810}
(see also 
\cite{MR2596639}), Kontsevich and Soibelman remarked that the 2-dimensional scattering diagram formalism has a natural $q$-deformation, 
with the group of formal families of symplectomorphisms of the 2-dimensional algebraic torus replaced 
by a group a formal families of automorphisms of the 2-dimensional quantum torus, a natural non-commutative deformation of the 2-dimensional algebraic torus. 
The enumerative meaning of this $q$-deformed scattering diagram was \emph{a priori} 
unclear. 

In Section 5.8 of 
\cite{MR2667135}, Gross, Pandharipande and Siebert remarked that the genus $0$ log Gromov-Witten invariants 
they consider have a natural extension to higher genus,
by integration of the top lambda class, and they asked 
if there is an interpretation 
of these higher genus invariants in terms of scattering diagrams.

The main result of the present paper,
Theorem \ref{main_thm_ch2}, 
is that the two previous questions, the 
enumerative meaning of the algebraic $q$-deformation and the algebraic meaning of
the higher genus deformation, are answers to each other.

\subsubsection*{\textbf{OV}}
The higher genus log Gromov-Witten invariants of log Calabi-Yau surfaces 
that we are considering--with insertion of the top lambda class--should be thought of as an 
algebrogeometric definition of some counts of higher genus Riemann surfaces with boundary on a Lagrangian torus fiber in a Calabi-Yau 3-fold geometry, 
essentially the product of the log Calabi-Yau surface by a third trivial direction;
see Section \ref{section_3d}.
For such counts of higher genus open curves in a Calabi-Yau 3-fold geometry, 
Ooguri-Vafa \cite{MR1765411} have conjectured an open BPS integrality structure. Theorem \ref{main_thm_integ}, 
which is a consequence of Theorem \ref{main_thm_ch2} and
of non-trivial algebraic properties of 
$q$-deformed scattering diagrams, can be viewed as a check of this BPS integrality structure.

\subsubsection*{\textbf{DT}}
The non-trivial integrality properties of $q$-deformed scattering diagrams are well-known to be related to 
integrality properties of refined Donaldson-Thomas (DT) invariants; see Kontsevich and Soibelman
 \cite{kontsevich2008stability}.
Indeed, $q$-deformed scattering diagrams control the wall-crossing behavior of refined 
DT invariants. 

The fact that the integrality structure of 
DT invariants coincides with the Ooguri-Vafa integrality structure of higher genus open Gromov-Witten invariants of Calabi-Yau 3-folds, 
essentially involving the quantum dilogarithm in both cases, can be viewed as an early indication
that something like Theorem
\ref{main_thm_ch2} should be true.

As consequence of \mbox{Theorem
\ref{main_thm_ch2}}, we obtain explicit relations between refined DT invariants of some quivers 
and higher genus log Gromov-Witten invariants of log Calabi-Yau surfaces --see Section
\ref{section_dt}-- generalizing the unrefined/genus $0$ relation of \cite{MR2662867}, \cite{MR3004575}.

\subsubsection*{\textbf{CV}}
In fact, Cecotti and Vafa
\cite{cecotti2009bps} have given a 
physical derivation of the wall-crossing formula in DT theory going through the
higher genus open Gromov-Witten theory of some Calabi-Yau 3-fold. 
We will explain in Section
\ref{section_cecotti_vafa} 
that Theorem \ref{main_thm_ch2} and \ref{main_thm_integ} are indeed fully compatible
with the Cecotti-Vafa argument. In particular, 
\mbox{Theorem
\ref{main_thm_ch2}} can be viewed as a
highly non-trivial mathematical check 
of the connection predicted by Witten
\cite{MR1362846} between higher genus open A-model and quantum Chern-Simons theory.

\subsubsection*{\textbf{del Pezzo}}
Theorem \ref{main_thm_ch2} and \ref{main_thm_integ} are about 
log Calabi-Yau surfaces with maximal boundary, ie with a singular normal crossing anticanonical divisor. 
Similar questions can be asked for log Calabi-Yau surfaces with respect to a smooth anticanonical divisor. Conjecture
\ref{conj_main} gives a non-trivial correspondence in such a case, 
suggested by the similarities between refined DT theory and open higher genus Gromov-Witten invariants discussed above. 

\subsection*{Applications}

In a way parallel to how
\cite{MR2667135} is used by Gross-Hacking-Keel
\cite{MR3415066}
to construct Poisson varieties as
mirrors of log Calabi-Yau surfaces with maximal boundary, we will use
Theorem \ref{main_thm_ch2} in a coming work
\cite{bousseau2018quantum} to construct
deformation quantizations of these Poisson varieties.

\subsection*{Comments on the proof of Theorem \ref{main_thm_ch2}}
The curve counting invariants appearing in 
Theorem \ref{main_thm_ch2} are log Gromov-Witten invariants, as defined by Gross and Siebert \cite{MR3011419},
and Abramovich and Chen
\cite{MR3224717}, \cite{MR3257836}.
The proof of Theorem \ref{main_thm_ch2}
relies on recently developed general properties of log Gromov-Witten invariants, such as the
decomposition formula of Abramovich, Chen, Gross and Siebert
\cite{abramovich2017decomposition}.

The main tool of
\cite{MR2667135} is a reduction to a tropical setting using the correspondence theorem of 
Mikhalkin 
\cite{MR2137980}
and Nishinou-Siebert
\cite{MR2259922} 
between counts of curves in complex toric surfaces and counts of tropical curves in $\R^2$. 
Similarly, the main tool of the present paper is a reduction to a tropical setting 
using a correspondence theorem previously proved by the author
\cite{bousseau2017tropical} between generating series of higher genus curves in toric
surfaces and Block-Göttsche 
\mbox{$q$-deformed}
tropical invariants. 

Given that the
relation between $q$-deformed tropical invariants and $q$-deformed scattering diagrams has already been worked out by 
Filippini and Stoppa
\cite{MR3383167}, Theorem \ref{main_thm_ch2} should really be viewed as a combination of 
\cite{bousseau2017tropical} and  
\cite{MR3383167}.
The new results required for the proof of 
\mbox{Theorem 
\ref{main_thm_ch2}} are: the check that the degeneration step used in
\cite{MR2667135} to go from a log Calabi-Yau setting 
to a toric setting extends to the higher genus case and 
the check that the correspondence 
proved in 
\cite{bousseau2017tropical} has exactly the correct form to be used as input in
\cite{MR3383167}. To make this paper more self-contained, we will in fact review the content of \cite{MR3383167}.

The most technical part is the higher genus version of the degeneration step. 
As the general version of the degeneration formula in log Gromov-Witten theory is not yet known, we combine the general 
decomposition formula of 
\cite{abramovich2017decomposition} with some situation specific vanishing statements, 
which, as in \cite{bousseau2017tropical}, reduce the gluing operations to some 
torically transverse locus where they are under control, thanks to, for example
Kim, Lho and Ruddat
\cite{kim2018degeneration}.

\subsection*{Comments on the proof of Theorem \ref{main_thm_integ}.}
The proof of \mbox{Theorem \mbox{\ref{main_thm_integ}}} is a combination of Theorem \ref{main_thm_ch2}
and of the non-trivial integrality results
about $q$-deformed scattering diagrams 
proved by Kontsevich and Soibelman in 
Section 6 of \cite{MR2851153}.
In fact, to get the most general form of
Theorem \ref{main_thm_integ}, the results contained in \cite{MR2851153} do not seem to be enough. 
We use an induction argument on scattering diagrams, parallel to the one used by Gross, Hacking, Keel and Kontsevich
\cite[Appendix C3]{MR3758151}, to reduce the most general case to a case which can be treated by \cite{MR2851153}. 

A small technical point is to keep track of signs,
because of the difference between quantum tori and twisted quantum tori; see \mbox{Section 
\ref{section_quad_ref}} on the quadratic refinement for details.

\subsection*{Plan of the paper}
In Section \ref{section_scattering}, we review the
notion of 2-dimensional scattering diagrams, both classical
and quantum, with an emphasis on the symplectic/Hamiltonian aspects. 
In Section \ref{section_log_calabi_yau}, 
we introduce a class of 
log Calabi-Yau surfaces and their
log Gromov-Witten invariants. 

In \mbox{Section \ref{section_main_result}}, we state our main result,
Theorem \ref{precise_main_thm_ch2},
a precise version of 
\mbox{Theorem \ref{main_thm_ch2}}, relating 2-dimensional quantum
scattering diagrams and generating series of higher genus log Gromov-Witten invariants of log Calabi-Yau surfaces.
We also state a generalization of Theorem
\ref{precise_main_thm_ch2}, Theorem 
\ref{precise_main_thm_orbifold}, phrased in terms of orbifold log Gromov-Witten invariants.

\mbox{ Sections
\ref{section_gw_toric},
\ref{section_degeneration},
\ref{section_tropical} and
\ref{section_end_proof}}
are dedicated to the proof of 
Theorems \ref{precise_main_thm_ch2} and 
\ref{precise_main_thm_orbifold}.
The general structure of the proof is
parallel to \cite{MR2667135}.
In \mbox{Section \ref{section_gw_toric}}, we introduce higher genus log Gromov-Witten invariants
of toric surfaces.
In Section \ref{section_degeneration}, the most technical part of this paper, 
we prove a degeneration formula relating log Gromov-Witten invariants of log Calabi-Yau surfaces defined in \mbox{Section 
\ref{section_log_calabi_yau}}
and appearing in 
\mbox{Theorem 
\ref{precise_main_thm_ch2}}, with log Gromov-Witten invariants of toric surfaces
defined in 
\mbox{Section
\ref{section_log_gw_toric}}.
In \mbox{Section \ref{section_tropical}},
following Filippini-Stoppa
\cite{MR3383167}, we review the connection between quantum scattering diagrams and refined counts of tropical curves. 
We finish the proof of 
Theorem \ref{precise_main_thm_ch2}
in Section \ref{section_end_proof}, 
combining the results of \mbox{Sections
\ref{section_degeneration} and
\ref{section_tropical}}
with the correspondence theorem proved in 
\cite{bousseau2017tropical} between refined counts of tropical curves and log Gromov-Witten invariants of toric surfaces.
The orbifold Gromov-Witten computation needed to finish the proof of Theorem 
\ref{precise_main_thm_orbifold} is done in 
\mbox{Section \ref{section_BP_orbifold}}.

In Section \ref{section_integ_conj}, we formulate a BPS integrality 
conjecture for higher genus 
log Gromov-Witten invariants 
of log Calabi-Yau surfaces.
In Section \ref{section_integ_result}, we state Theorem \ref{thm_integ}, precise form of Theorem \ref{main_thm_integ}. 
The proof of Theorem \ref{thm_integ} takes Sections
\ref{section_quad_ref}, \ref{section_proof_integ}.
In Section \ref{section_dt}, Theorem
\ref{thm_dt} gives an explicit connection with refined DT invariants of quivers.
Finally, in Section
\ref{section_del_pezzo}, we state 
Conjecture \ref{conj_del_pezzo}, precise version of Conjecture
\ref{conj_main}.

In Section \ref{section_cecotti_vafa}, we explain how Theorem \ref{main_thm_ch2} 
can be viewed as a mathematical check of the physics work of Cecotti and Vafa 
\cite{cecotti2009bps}
and how \mbox{Theorem \ref{main_thm_integ}} is compatible with the Ooguri-Vafa
integrality conjecture \cite{MR1765411}.

\subsection*{Acknowledgements.}
I would like to thank my supervisor Richard Thomas for continuous support and innumerous discussions, suggestions and corrections. 

I thank Rahul Pandharipande, Tom Bridgeland, Liana Heuberger, Michel van Garrel and Bernd Siebert 
for various invitations to conferences and seminars where this work (or rather its application \cite{bousseau2018quantum}) has been presented. 
I thank Luca Battistella and Michel van Garrel for corrections and comments on a first draft of this paper.
I thank the referee for useful corrections and suggestions.

This work is supported by the EPSRC award 1513338, Counting curves in algebraic geometry, Imperial College London, 
and has benefited from the EPRSC [EP/L015234/1], EPSRC Centre for Doctoral Training in Geometry and Number Theory 
(London School of Geometry and Number Theory), University College London.

\section{Scattering}
\label{section_scattering}

In this Section, we first fix our notation for
the basic objects considered in this paper:
tori, quantum tori and automorphisms of formal families of them. 
We then introduce scattering diagrams, both classical and quantum, following 
\cite{MR2181810}, \cite{MR2846484}, \cite{MR2667135} and \cite{MR3383167}.

\subsection{Torus}
We fix $T$ a $2$-dimensional
complex algebraic torus.
Let 
$M \coloneqq \Hom (T, \C^{*})$
be the $2$-dimensional lattice of 
characters of $T$. Characters form a linear basis of the algebra of functions 
on $T$,
\[ \Gamma (\cO_T) = \bigoplus_{m \in M} \C z^m \,,\]
with the product given by 
$z^m \cdot z^{m'}
=z^{m + m'}$. In other words,
the algebra of functions on $T$
is the algebra of the lattice
$M$: 
$\Gamma (\cO_T)=\C [M]$.

We fix 
\[\langle -,- \rangle \colon 
\bigwedge\nolimits^2 M \xrightarrow{\sim} \Z \]
an orientation of $M$,
ie an integral unimodular skew-symmetric 
bilinear form on $M$. This defines a Poisson bracket 
on $\Gamma(\cO_T)$, given by 
\[ \{z^m, z^{m'}\}=\langle m ,m' \rangle z^{m+m'} \,, \]
and a corresponding algebraic symplectic form 
$\Omega$ on $T$.

If we choose a basis $(m_1,m_2)$ of $M$ such that 
$\langle m_1,m_2 \rangle =1$, then, writing
$z_1\coloneqq z^{m_1}$ and $z_2\coloneqq z^{m_2}$, we have identifications 
$T=(\C^{*})^2$, $M=\Z^2$,
$\Gamma(\cO_T)=\C[z_1^\pm, z_2^\pm]$ 
and 
\[ \Omega = \frac{dz_1}{z_1} \wedge \frac{dz_2}{z_2}\,.\]

\subsection{Quantum torus}
Given the symplectic torus 
$(T, \Omega)$, or equivalently 
the Poisson algebra 
$(\Gamma(\cO_T), \{-,-\})$,
it is natural to look for a ``quantization''.
The quantum torus $\hat{T}^q$ is the non-commutative 
``space'' whose algebra of functions is the 
non-commutative $\C[q^{\pm \frac{1}{2}}]$-algebra 
$\Gamma(\cO_{\hat{T}^q})$, with linear basis 
indexed by the lattice $M$,
\[\Gamma(\cO_{\hat{T}^q}) = \bigoplus_{m \in M}
 \C[q^{\pm \frac{1}{2}}] \hat{z}^m \,,\]
and with product defined by 
\[ \hat{z}^m \cdot \hat{z}^{m'}=q^{\frac{1}{2} \langle m,m'
\rangle} \hat{z}^{m+m'}\,.\]
The quantum torus $\hat{T}^q$
is a quantization of the torus $T$
in the sense that writing $q=e^{i \hbar}$ and taking 
the limit $\hbar \rightarrow 0$,
$q \rightarrow 1$, the linear term
in $\hbar$
of the commutator 
$[\hat{z}^m, \hat{z}^{m'}]$
is determined by the 
Poisson bracket 
$\{z^m, z^{m'}\}$. Indeed, we have 
\[ [\hat{z}^m, \hat{z}^{m'}]
= (q^{\frac{1}{2} \langle m, m' \rangle}
- 
q^{-\frac{1}{2} \langle m, m' \rangle})
\hat{z}^{m + m'}\,,\]
and so 
\[ \lim_{\hbar \rightarrow 0} \frac{1}{i \hbar}[\hat{z}^m, \hat{z}^{m'}]
= \langle m,m' \rangle
\hat{z}^{m+m'} \,. \]

We denote by $\hat{T}^{\hbar}$ 
the non-commutative ``space''
whose algebra of 
functions is the 
$\C (\!( \hbar )\!)$-algebra 
$\Gamma (\cO_{\hat{T}^{\hbar}})\coloneqq
\Gamma (\cO_{\hat{T}^{\hbar}})
\otimes_{\C[q^{\pm \frac{1}{2}}]} \C(\!(
\hbar )\!)$.

\subsection{Automorphisms of formal families of tori}
\label{section_autom}
Let $R$ be a complete local $\C$-algebra 
and let 
$\fm_R$ be the maximal ideal of $R$.
By definition of completeness, we have
\[ R = \varprojlim_\ell R/\fm_R^\ell \,.\]
We denote $S \coloneqq \Spf R$
the corresponding formal scheme and $s_0$ the closed 
point of $S$ defined by $\fm_R$. 
Let $T_S$ be the trivial family of \mbox{$2$-dimensional} 
complex algebraic tori parametrized by $S$, ie
$T_S \coloneqq S \times T$. The corresponding algebra 
of functions is given 
by 
\[\Gamma (\cO_{T_S})= \varprojlim_\ell (R/\fm_R^\ell \otimes \Gamma(\cO_T))
 = \varprojlim_\ell (R/\fm_R^\ell \otimes \C[M])\,.\]
Let $\hat{T}^{\hbar}_S$
be the trivial family of non-commutative 
$2$-dimensional tori parametrized by $S$,
ie $\hat{T}^{\hbar}_S \coloneqq S \times \hat{T}^{\hbar}$. The corresponding 
algebra of functions is simply given by 
\[ \Gamma (\cO_{\hat{T}^{\hbar}_S})=
\varprojlim_\ell (R/\fm_R^\ell \otimes \Gamma (\cO_{\hat{T}^{\hbar}}))\,.\]

The family $T_S$ of tori has a natural Poisson structure, whose symplectic leaves are the torus fibers, 
and whose Poisson center is $R$. Explicitly, we have 
\[ \{H_m z^m, H_{m'} z^{m'}\}= H_m H_{m'} \{z^m, z^{m'} \} \,,\]
for every $H_m , H_{m'} \in R$ and $m, m' \in M$.
The family $\hat{T}_S^{\hbar}$ of non-commutative tori is a quantization of the Poisson variety $T_S$.

Let 
\[ H = \sum_{m \in M} H_m z^m \] 
be a function on 
$T_S$ whose restriction to the fiber over the closed point $s_0 \in S$ vanishes, ie such that 
$H = 0 \mod \fm_R$. 
Then $\{H,-\}$ defines a derivation of the algebra of functions on $T_S$ and so a vector field on $T_S$, 
the Hamiltonian vector field defined by $H$, whose restriction to the fiber over the closed 
point $s_0 \in S$ vanishes.

The time one flow of this vector field defines an automorphism
\[ \Phi_H \coloneqq \exp \left( 
\{H,-\} \right)\]
of $T_S$, whose restriction to the fiber over the closed point 
$s_0 \in S$ is the identity.
Note that $\Phi_H$ is well-defined 
because of the assumptions that 
$H=0 \mod \fm_R$ and $R$ is a complete local algebra,
ie $\exp$ makes sense formally.

Let $\V_R$ be the subset of automorphisms of $T_S$ which are of the form $\Phi_H$ for $H$ as above.
By the Baker-Campbell-Hausdorff 
formula, 
$\V_R$ is a subgroup of the group  of automorphisms of $T_S$. 
In \cite{MR2667135}, $\V_R$ is called the tropical vertex group.

Let 
\[ \hat{H} = \sum_{m \in M}
\hat{H}_m \hat{z}^m \]
be a function on $\hat{T}_S^{\hbar}$ whose restriction to the fiber over the closed point $s_0 \in S$ vanishes, 
ie such that 
$\hat{H}=0 \mod \fm_R$.
Conjugation by $\exp \left(
\hat{H} \right)$
defines an automorphism 
\[\hat{\Phi}_{\hat{H}} \coloneqq \Ad_{\exp \left(\hat{H} \right)}
=\exp \left(
\hat{H}\right) (-)
\exp \left(-\hat{H} \right)\]
of $\hat{T}_S^{\hbar}$ whose restriction to the fiber over the closed point $s_0
\in S$ is the identity.
Note that $\hat{\Phi}_{\hat{H}}$ is well-defined because of the assumption 
that $\hat{H}=0 \mod \fm_R$
and $R$ is a compete local algebra, ie
everything makes sense formally.
Let $\hat{\V}^{\hbar}_R$ be the subset of automorphisms of $\hat{T}_S^{\hbar}$ which are of the form
$\hat{\Phi}_{\hat{H}}$ for
$\hat{H}$ as above.
By the Baker-Campbell-Hausdorff formula, $\hat{\V}^{\hbar}_R$ is a subgroup 
from the group of automorphisms 
of $\hat{T}_S^{\hbar}$. 
We call $\hat{\V}^{\hbar}_R$ the quantum tropical vertex group.
This group is much bigger that the ``quantum tropical vertex group" of 
\cite{MR2851153}. We will meet the 
group of \cite{MR2851153} in Section 
\ref{section_integrality}, under the name
``BPS quantum tropical vertex group".

If the limit 
\[H \coloneqq \lim_{\hbar \rightarrow 0}
(i\hbar \hat{H})\]
exists, then, replacing $\hat{z}^m$ by $z^m$, 
$H$ can be naturally viewed as a function on
$T_S$ and is the classical limit of 
$\hat{H}$. It is easy to check that 
$\Phi_H$ is the classical limit of 
$\hat{\Phi}_{\hat{H}}$.

\subsection{Scattering diagrams}
\label{section_scattering_diag}
In this section, we work in the $2$-dimensional real plane 
$M_\R \coloneqq M \otimes_{\Z} \R$. 
A \emph{ray} is a half-line $\fd$ in $M_\R$ of rational slope with initial point the origin $0\in M_\R$, 
and we denote $m_{\fd} \in M-\{0\}$ its primitive integral direction, pointing away from the origin.

\begin{defn}
A \emph{scattering diagram} $\fD$
over $R$ is a set of rays
$\fd$ in $M_\R$, 
equipped with functions
$H_{\fd}$ such that
either 
\[ H_{\fd}
\in \varprojlim_\ell ( R/\fm_R^\ell \otimes \C[z^{m_{\fd}}]) \,,\]
or 
\[ H_{\fd}
\in \varprojlim_\ell ( R/\fm_R^\ell \otimes \C[z^{-m_{\fd}}]) \,,\]
and such that
$H_{\fd} =0 \mod \fm_R$,
and
for every $\ell \geqslant 1$, only finitely many rays
$\fd$ have 
$H_{\fd} \neq 0 \mod \fm_R^\ell$.

A ray $(\fd, H_{\fd})$ such that 
\[ H_{\fd}
\in \varprojlim_\ell ( R/\fm_R^\ell \otimes \C[z^{m_{\fd}}]) \,,\]
is called \emph{outgoing} and 
a ray 
$(\fd, H_{\fd})$ such that 
\[ H_{\fd}
\in \varprojlim_\ell ( R/\fm_R^\ell \otimes \C[z^{-m_{\fd}}]) \,,\]
is called \emph{ingoing}.

Given a ray $(\fd,H_\fd)$, we denote
$m(H_\fd) \coloneqq m_\fd$ if $(\fd,H_\fd)$ is outgoing, and $m(H_\fd)
\coloneqq -m_\fd$ if 
$(\fd,H_\fd)$ is ingoing. In both cases, we have \[ H_{\fd}
\in \varprojlim_\ell ( R/\fm_R^\ell \otimes \C[z^{m(H_\fd)}]) \,,\]
\end{defn}

We will always consider scattering diagrams up to the following simplifying operations:
\begin{itemize}
\item A ray $(\fd, H_{\fd})$ with $H_{\fd}=0$ is considered to be trivial and can be safely removed from the scattering diagram.
\item If two rays $(\fd_1, H_{\fd_1})$ and $(\fd_2, H_{\fd_2})$
are such that $\fd_1=\fd_2$ and are both ingoing or outgoing, 
then they can be replaced by a single ray $(\fd, H_{\fd})$, where $\fd=\fd_1=\fd_2$ and 
$H_{\fd}=H_{\fd_1}+H_{\fd_2}$. Note that, because 
$\{ H_{\fd_1}, H_{\fd_2}\}=0$, we have that 
$\Phi_{H_{\fd}}=\Phi_{H_{\fd_1}}\Phi_{H_{\fd_2}}
=\Phi_{H_{\fd_2}}\Phi_{H_{\fd_1}}$.
\end{itemize}




Let $\fD$ be a scattering diagram.
The \emph{singular locus} of $\fD$ is 
the union of the set of initial points of rays and of the set of non-trivial intersection points of rays.
Let $\gamma\colon [0,1] \rightarrow M_\R$ be a smooth path. We say that $\gamma$ is admissible 
if $\gamma$ does not intersect the singular locus of $\fD$, if the endpoints of $\gamma$ are not on rays of $\fD$, and if $\gamma$
intersects transversely all the rays of $\fD$.

Let $\gamma$ be an admissible smooth path in $M_\R$.
Let $\ell \geqslant 1$ be a positive integer. By definition, 
$\fD$ contains only finitely 
many rays $(\fd, H_\fd)$ with 
$H_{\fd} \neq 0 \mod \fm_R^\ell$.
We denote 
$0<t_1 \leqslant \dots 
\leqslant t_s <1$ the times of intersection of 
$\gamma$ with rays $(\fd_1, H_{\fd_1}),
\dots, (\fd_s, H_{\fd_s})$ of $\fD$ such that 
$H_{\fd_r} \neq 0 \mod \fm_R^\ell$. For every
$1 \leqslant r \leqslant s$, we define $\epsilon_r \in \{\pm 1\}$ to be the sign of $\langle m(H_{\fd_i}), \gamma'(t_r) \rangle$. 
We then define 
\[\theta_{\gamma, \fD,\ell} \coloneqq \Phi_{H_{\fd_s}}^{\epsilon_s} \cdots \Phi_{H_{\fd_1}}^{\epsilon_1} \,.\]
Taking the limit $\ell \rightarrow +\infty$, we define
\[\theta_{\gamma,\fD} \coloneqq
\lim_{\ell \rightarrow +\infty}
\theta_{\gamma,\fD,\ell} \,.\]

\begin{defn}
A scattering diagram $\fD$
over $R$ is \emph{consistent} if, for every closed admissible smooth path $\gamma \colon [0,1] \rightarrow M_\R$, 
we have $\theta_{\gamma, \fD}=\id$.
\end{defn}

The following result is due to Kontsevich and Soibelman
\cite[Theorem 6]{MR2181810} (see also \cite[Theorem 1.4]{MR2667135}).

\begin{prop} \label{prop_consistency}
Any scattering diagram $\fD$ can be canonically completed by adding
only outgoing rays to form a consistent scattering diagram 
$S(\fD)$.
\end{prop}

\begin{proof}
It is enough to show that for every 
nonnegative integer $\ell$, it is possible to add outgoing 
rays to $\fD$ to get a scattering diagram $\fD_\ell$ consistent at the order $\ell$, 
ie such that 
$\theta_{\gamma, \fD_\ell}= \id \mod 
\fm_R^{\ell+1}$. The construction is done by induction on $\ell$, starting with 
$\fD_0=\fD$. Let us assume we have constructed $\fD_{\ell-1}$, consistent at the order $\ell-1$. 
Let $p$ be a point in the singular locus of 
$\fD_{\ell-1}$ and let $\gamma$ be a
small anticlockwise closed loop around $p$. As $\fD_{\ell-1}$ is consistent at the order $\ell-1$, 
we can write $\theta_{\gamma, \fD_{\ell-1}}
=\Phi_H$ for some $H$ with 
$H = 0 \mod \fm_R^\ell$.
There are finitely many primitive
$m_j \in M-\{0\}$ such that we can write
\[ H=\sum_j H_j \mod \fm_R^{\ell+1}\] with 
$H_j \in \fm_R^\ell R \otimes \C[z^m_j]$.
We construct $\fD_{\ell}$ by adding to 
$\fD_{\ell-1}$ the outgoing rays 
$(p+\R_{\geqslant 0}m_j, \Phi_{-H_j})$.
\end{proof}

Adding hats everywhere, we obtain the definition of a 
quantum scattering diagram 
$\hat{\fD}$, with functions
\[ \hat{H}_{\fd} \in 
\varprojlim_\ell (R/\fm_R^\ell \otimes \C (\!(\hbar)\!) [\hat{z}^{m_\fd}])\,,\]
for outgoing rays and 
\[ \hat{H}_{\fd} \in 
\varprojlim_\ell (R/\fm_R^\ell \otimes \C (\!(\hbar)\!) [\hat{z}^{-m_\fd}])\,,\]
for ingoing rays,
the notion of consistent quantum scattering diagram, and the fact that 
every quantum scattering diagram 
$\hat{\fD}$ can be canonically completed by adding only
outgoing rays to form a 
consistent quantum scattering diagram 
$S(\hat{\fD})$.
We will often call $\hat{H}_\fd$ the Hamiltonian attached to the ray $\fd$.

A general notion of scattering diagram, as in Section 2 of \cite{kontsevich2013wall}, takes as input
a lattice $M$ and an $M$-graded Lie algebra
$\fg$.
What we call a (classical) scattering diagram is the special case 
where $M$ is the lattice of characters of a 2-dimensional symplectic torus $T$ and where $\fg=(\Gamma(\cO_{T_S}),\{-,-\})$. 
What we call a quantum scattering diagram is the special case where $M$ is the lattice of characters of a 2-dimensional symplectic torus $T$ 
and where 
$\fg=(\Gamma(\cO_{\hat{T}^{\hbar}_S}), [-,-])$.

In our definition of a scattering diagram, we attach to each ray 
$\fd$ a function 
\[ H_{\fd}=\sum_{\ell \geqslant 0} H_\ell z^{\ell m(H_{\fd})} \,, \]
such that $H_{\fd}=0 \mod \fm_R$, 
which can be interpreted as Hamiltonian generating an automorphism 
\[\Phi_{H_{\fd}}= \exp \left( 
\{H_{\fd},-\} \right)\,.\]
In \cite{MR2667135}, \cite{MR2846484}
or \cite{MR3383167}, the terminology is slightly different. To a ray $\fd$,
they attach a function 
\[f_{\fd} =\sum_{\ell \geqslant 0} c_\ell z^{\ell m(H_{\fd})} \,,\]
such that $f_{\fd}=1 \mod \fm_R$, and,
to a path $\gamma(t)$ intersecting transversely $\fd$ at time $t_0$,
 an automorphism
\[\theta_{f_{\fd},\gamma} \colon z^m \mapsto z^m f_{\fd}^{\langle n_\fd , m\rangle}\,,\]
where $n_\fd$ is the primitive generator 
of $\fd$ such that $\langle n_\fd , \gamma'(t_0)\rangle >0$.
These two choices are equivalent.
Indeed, if $\epsilon$ is the sign of 
$\langle m(H_\fd),\gamma'(t_0) \rangle$,
we have 
\[\Phi_{H_\fd}^{\epsilon}=\theta_{f_\fd,\gamma} \,\]
if $H_\fd$ and $f_\fd$ are related by 
\[ \log f_\fd = \sum_{\ell \geqslant 0} \ell H_\ell z^{\ell m(H_\fd)} \,.\]
The formalism of \cite{MR2846484} is more general because it treats the Calabi-Yau case and not just a holomorphic symplectic case. 
In the present paper, focused on a holomorphic symplectic situation, using  the Hamiltonians $H_\fd$ rather than the
functions $f_\fd$ makes the quantization step transparent. 
The quantum version of the functions 
$f_\fd$ will be studied and used in 
\cite{bousseau2018quantum}.

\section{Gromov-Witten theory of log Calabi-Yau surfaces}
\label{section_log_calabi_yau}

Our main result, Theorem \ref{precise_main_thm_ch2}, 
is an enumerative interpretation of a class of quantum scattering diagrams, as introduced in Section 
\ref{section_scattering}, 
in terms of higher genus log Gromov-Witten invariants of a class of log Calabi-Yau surfaces.
In Section \ref{section_log_cy_surface}
we review the definition of these
log Calabi-Yau surfaces, following 
\cite{MR2667135}. 
We define the relevant higher genus 
log Gromov-Witten invariants
in Sections
\ref{section_curve_classes_log_cy}
and
\ref{section_log_gw_log_cy}. 
We give a 3-dimensional 
interpretation of these invariants in Section \ref{section_3d}.
Finally, we give a generalization of these invariants to a certain orbifold context 
in Section
\ref{section_orbifold_gw}.

\subsection{Log Calabi-Yau surfaces}
\label{section_log_cy_surface}

We fix $\fm =(m_1, \dots, m_n)$ an $n$-tuple of primitive non-zero vectors of 
$M=\Z^2$. 
The fan in $\R^2$ with rays  
$-\R_{\geqslant 0} m_1, \dots, -
\R_{\geqslant 0} m_n$ defines a toric surface $\overline{Y}_\fm$.
Let $D_{m_1}, \dots, D_{m_n}$
be the corresponding toric 
divisors.
If $m_1, \dots, m_n$ do not span $M$, ie if 
$\overline{Y}_\fm$ is non-compact, we add some extra rays to the fan to make it span $M$ 
and we still denote $\overline{Y}_\fm$ the corresponding compact toric surfaces. 
The choice of the added
rays will be irrelevant for us
because of the log birational invariance result in logarithmic Gromov-Witten theory proved in 
\cite{MR3778185}. We denote by $\partial \overline{Y}_\fm$ the toric boundary divisor of $\overline{Y}_\fm$.

For every $1 \leqslant j \leqslant n$, we blow-up $Y_\fm$ at a point
$x_j$ on the toric divisor $D_{m_j}$
and not on any other toric divisor. We also assume that all the points $x_j$
are distinct.
By deformation invariance of log Gromov-Witten invariants, the precise choice of $x_j$ will be irrelevant for us. 
Note that it is possible to have 
$\R_{\geqslant 0} m_j = \R_{\geqslant 0} m_{j'}$, and so $D_{m_j}=D_{m_{j'}}$, 
for $j \neq j'$, and that in this case we blow-up several distinct points on the same toric divisor. 
We denote by $Y_{\fm}$ the resulting
projective surface and 
$\nu \colon 
Y_{\fm} \rightarrow 
\overline{Y}_\fm$ the blow-up morphism.
Let $E_j \coloneqq \nu^{-1}(x_j)$ be the exceptional divisor over $x_j$.
We denote by
$\partial Y_{\fm}$ the strict transform of the toric 
boundary divisor. The divisor 
$\partial Y_{\fm}$ is an anticanonical cycle of rational curves and so the pair $(Y_{\fm}, \partial Y_{\fm})$
is an example of a log Calabi-Yau surface with maximal boundary.

\subsection{Curve classes}
\label{section_curve_classes_log_cy}
We want to consider curves in $Y_{\fm}$
meeting $\partial Y_{\fm}$ in a unique point. We first explain how to parametrize 
the relevant curve classes 
in terms of their intersection numbers
$p_j$ with the exceptional divisors $E_j$.

Let $p\coloneqq (p_1, \dots, p_n)  
\in P \coloneqq \NN^n$. We assume that 
$\sum_{j=1}^n p_j m_j \neq 0$ and so we can uniquely write 
\[ \sum_{j=1}^n p_j m_j = \ell_p m_p \,,\]
for some $m_p \in M$ primitive and some 
$\ell_p \in \NN$. 

We explain now how to define a curve class 
$\beta_p \in H_2(Y_\fm, \Z)$. In short, 
$\beta_p$ is the class of a curve in 
$Y_\fm$ having for every $1 \leqslant j \leqslant n$ intersection number 
$p_j$ with the exceptional divisor $E_j$, and exactly one intersection 
point with the anticanonical cycle
$\partial Y_\fm$.

More precisely, 
the vector $m_p \in M$ 
belongs to some cone of the fan of $\overline{Y}_\fm$
and we write the corresponding decomposition
\[ m_p=a_p^L m_p^L + a_p^R m_p^R\,,\]
where $m_p^L, m_p^R \in M$ are primitive generators of rays of the fan of $\overline{Y}_\fm$ and where 
$a_p^L, a_p^R \in \NN$. 
Note that there is only one term in this decomposition if the ray $\R_{\geqslant 0} m_p$ coincides with one of the
rays of the fan of $\overline{Y}_\fm$.
Let 
$D_p^L$ and $D_p^R$ be the toric divisors corresponding to the rays 
$\R_{\geqslant 0}m_p^L$ and 
$\R_{\geqslant 0}m_p^R$.
Let $\beta \in H_2(\overline{Y}_\fm,\Z)$
be determined by the following intersection numbers with the toric divisors:
\begin{itemize}
\item The intersection numbers with those 
$D_{m_j}$ for $1 \leqslant j \leqslant n$ which are 
distinct from $D_p^L$ and $D_p^R $:
\[ \beta \cdot D_{m_j} = \sum_{j', D_{m_{j'}} =D_{m_j}} p_{j'} \,.\]
\item The intersection number with $D_p^L$: 
\[ \beta \cdot D_p^L=\ell_p a_p^L+\sum_{j, D_{m_j}=D_p^L} p_j \,.\]
\item The intersection number with $D_p^R$:
\[\beta \cdot D_p^R=\ell_p a_p^R+\sum_{j,D_{m_j}=D_p^R}
p_j \,.\]
\item 
The intersection number with every toric divisor $D$
different
from $D_{m_j}$ for every $1 \leqslant j \leqslant n$,
and from $D_p^L$ and $D_p^R$:
\[\beta \cdot D=0\,.\]
\end{itemize}

Such class $\beta \in H_2(\overline{Y}_\fm,\Z)$ exists and is unique by standard toric geometry because of 
the relation $\sum_{j=1}^n p_j m_j=\ell_p m_p$.
Finally, we define 
\[ \beta_p \coloneqq \nu^{*} \beta 
- \sum_{j=1}^n p_j E_j \in 
H_2(Y_{\fm},\Z) \,.\]
By construction, we have 
\[\beta_p \cdot E_j=p_j \,,\]
for every $1 \leqslant j \leqslant n$, 
\[\beta_p \cdot D_p^L=\ell_p a_p^L \,\,\,\, \text{and} \,\,\,\,
\beta_p \cdot D_p^R=\ell_p a_p^R \,,\]
and 
\[\beta_p \cdot D=0 \,,\]
for every component $D$ of 
$\partial Y_{\fm}$ distinct from
$D_p^L$ and $D_p^R$.

\subsection{Log Gromov-Witten invariants}
\label{section_log_gw_log_cy}

For every $p=(p_1, \dots, p_n) \in P=\NN^n$,
we defined in the previous Section 
\ref{section_curve_classes_log_cy} positive integers $\ell_p$, $a_p^L$ and $a_p^R$, 
some components $D_p^L$ and $D_p^R$ of the divisor $\partial Y_\fm$
 and a curve class $\beta_p \in H_2(Y_\fm,\Z)$.
We would like to consider genus $g$ stable maps $f \colon C \rightarrow Y_{\fm}$
of class $\beta_p$ that meet $\partial Y_m$ in a unique point.
At this point, such a map necessarily has an intersection number $\ell_p a_p^L$ with $D_p^L$ and $\ell_p a_p^R$
with $D_p^R$.

The space of such stable maps is not proper in general: a limit of stable maps intersecting $\partial Y_{\fm}$ in a unique point does not 
necessarily intersect 
$\partial Y_{\fm}$ in a unique point. For example, a component of the limit curve could map entirely inside $\partial Y_\fm$. 
A nice compactification of this space is obtained by considering
stable log maps. The idea is to allow maps with components possibly mapping entirely inside
$\partial Y_\fm$, 
but to also remember some additional information under the form of log structures, 
which give a way to make sense of tangency conditions for points on such components mapping entirely inside $\partial Y_\fm$. 
The theory of stable log maps has been developed by Gross and Siebert \cite{MR3011419},
and Abramovich and Chen
\cite{MR3224717}, \cite{MR3257836}.
By stable log maps, we always mean basic stable log maps in the sense of
\cite{MR3011419}.
We refer to Kato 
\cite{MR1463703} for elementary notions of log geometry.
We consider the divisorial log structure on 
$Y_{\fm}$ defined by the divisor 
$\partial Y_{\fm}$ and use it to view $Y_\fm$ as a smooth log scheme.

Let 
$\overline{M}_{g, p} (Y_{\fm} / \partial Y_{\fm})$ be the moduli space of 
genus $g$ stable log maps to
$Y_{\fm}$, of class 
$\beta_p$, with contact order along $\partial Y_{\fm}$ given by $\ell_p m_p$.
According to the main results of 
\cite{MR3011419}, this is a proper Deligne-Mumford stack
coming with a $g$-dimensional virtual fundamental class
\[ [\overline{M}_{g, p} (Y_{\fm}/\partial Y_{\fm})]^{\virt}\,.\]
If $\pi \colon \cC \rightarrow 
\overline{M}_{g, p} (Y_{\fm}/\partial Y_{\fm})$ is the universal curve,
of relative dualizing sheaf 
$\omega_\pi$, then the Hodge bundle
\[ \E\coloneqq \pi_{*}\omega_\pi \]
is a rank $g$ vector bundle over 
$\overline{M}_{g, p} (Y_{\fm}/\partial Y_{\fm})$. Its Chern classes 
are classically called the lambda classes
\cite{MR717614} and denoted by
$\lambda_j \coloneqq c_j(\E)$
for $0 \leqslant j \leqslant g$.
Finally, we define 
the genus $g$ log Gromov-Witten invariants
$N_{g,p}^{Y_\fm} \in \Q$ of 
$Y_\fm$ which will be of interest for us by 
\[N_{g,p}^{Y_\fm}\coloneqq
\int_{[\overline{M}_{g, p} (Y_{\fm}/\partial Y_{\fm})]^{\virt}} 
(-1)^g \lambda_g  \,.\]

Note that the top lambda class $\lambda_g$ has exactly the right degree to cut down the virtual dimension from $g$ to $0$, so that 
$N_{g,p}^{Y_\fm}$ is not obviously zero.

The fact that the top lambda class should be 
the natural insertion to consider for some higher genus version of \cite{MR2667135}  
was already suggested in Section 5.8 of \cite{MR2667135}. 
From our point of view, 
higher genus invariants with the top lambda class inserted are the correct objects because it is to them that the correspondence tropical theorem of 
\cite{bousseau2017tropical} applies.
In Section \ref{section_cecotti_vafa}, we will explain how our main result 
Theorem \ref{precise_main_thm_ch2} fits into an 
expected story for higher genus open holomorphic curves in Calabi-Yau 3-folds. 
This is probably the most conceptual understanding of the role of the invariants $N_{g,\beta}^{Y_\fm}$: they are really 
higher genus invariants of the log Calabi-Yau
3-fold $Y_\fm \times \PP^1$, and the top lambda class is simply a measure of the difference between surface and 3-fold obstruction theories.
This will be made precise in the following Section \ref{section_3d}, whose analogue for K3 surfaces is well-known, see Lemma 7 of \cite{MR2746343}.

\subsection{3-dimensional interpretation of the invariants $N_{g,p}^{Y_\fm}$}
\label{section_3d}

In this Section, we rewrite the log Gromov-Witten invariants $N_{g,p}^{Y_\fm}$ of the log Calabi-Yau surface $Y_\fm$ in terms
of 3-dimensional geometries, first 
$Y_\fm \times \C$ and then $Y_\fm \times \PP^1$. 

We endow the 3-fold $Y_\fm \times \C$ with the smooth log structure given by the divisorial log structure along the divisor 
$\partial Y_\fm \times \C$. 
Let 
\[ \overline{M}_{g, p} (Y_{\fm} \times \C/\partial Y_{\fm} \times \C)\] 
be the moduli space of 
genus $g$ stable log maps to
$Y_{\fm}$, of class 
$\beta_p$, with contact order along $\partial Y_{\fm} \times \C$ given by $\ell_p m_p$.
This is a Deligne-Mumford stack
coming with a $1$-dimensional
virtual fundamental class
\[ [\overline{M}_{g, p} (Y_{\fm}/\partial Y_{\fm})]^{\virt}
\,.\]
Because $\C$ is not compact, 
$\overline{M}_{g, p} (Y_{\fm} \times \C/\partial Y_{\fm} \times \C)$ is not compact 
and so one cannot simply integrate over the virtual class. 
Using the standard action of 
$\C^{*}$ on $\C$, fixing $0 \in \C$, 
we obtain an action of $\C^{*}$
on $\overline{M}_{g, p} (Y_{\fm} \times \C/\partial Y_{\fm} \times \C)$, with its perfect obstruction theory, whose
fixed point locus is the space of stable log maps mapping to $Y_\fm \times \{0\}$, ie
$\overline{M}_{g, p} (Y_{\fm}/\partial Y_{\fm})$, with its natural perfect
obstruction theory.
Given the virtual localization formula
\cite{MR1666787}, it is natural to define
equivariant log Gromov-Witten invariants
\[ N_{g,p}^{Y_\fm \times \C} 
\coloneqq \int_{[\overline{M}_{g, p} (Y_{\fm}/\partial Y_{\fm})]^{\virt}}
\frac{1}{
e(\mathrm{Nor}^{\virt})} 
\in \Q(t)
\,,
\]
where 
$\mathrm{Nor}^{\virt}$
is the equivariant virtual normal bundle of 
$\overline{M}_{g, p} (Y_{\fm}/\partial Y_{\fm})$
in $\overline{M}_{g, p} (Y_{\fm} \times \C/\partial Y_{\fm} \times \C)$, 
$e(\mathrm{Nor}^{\virt})$ is its equivariant Euler class, and $t$ is the generator of the 
$\C^{*}$-equivariant cohomology of a point.

\begin{lem} \label{lem_3d_1}
\[N_{g,p}^{Y_\fm \times \C} = \frac{1}{t}
N_{g,p}^{Y_\fm} \,.\]
\end{lem}

\begin{proof}
Because the 3-dimensional geometry 
$Y_\fm \times \C$, including the log/tangency conditions, is obtained from the 2-dimensional geometry $Y_\fm$ 
by a trivial product with a trivial factor $\C$, with $\C^{*}$ scaling this 
trivial factor, the virtual normal at a stable log map $f \colon C \rightarrow Y_\fm$ is 
$H^0(C,f^{*} \cO)-H^1(C,f^{*}\cO)
=t - \E^\vee \otimes t$ so 
\[\frac{1}{
e(\mathrm{Nor}^{\virt})} = \frac{1}{t}\left(\sum_{i=0}^g
(-1)^i \lambda_i t^{g-i}\right) \,,\]
and
\[N_{g,p}^{Y_\fm \times  \C}
= \int_{[\overline{M}_{g, p} (Y_{\fm},
\partial Y_{\fm})]^{\virt}} \frac{(-1)^g 
\lambda_g}{t} = \frac{1}{t} N_{g,p}^{Y_\fm} \,.\]
\end{proof}

The proof of Lemma
\ref{lem_3d_1} is identical to the proof of  \mbox{Lemma 7} in \cite{MR2746343} up to a small point: in \cite{MR2746343}, 
counts of expected dimensions work because of the use of a 
reduced Gromov-Witten theory of K3 surfaces, whereas for us, counts of expected dimensions
work because of the use of log Gromov-Witten theory.

We consider now the 3-fold $Z_\fm 
\coloneqq Y_\fm \times \PP^1$
with the smooth log structure given by the 
divisorial log structure along the divisor
\[ \partial Z_\fm \coloneqq (\partial Y_\fm \times \PP^1) \cup 
(Y_\fm \times \{0\}) \cup 
(Y_\fm \times \{ \infty \})\,.\]
The divisor $\partial Z_\fm$ is anticanonical, containing zero-dimensional strata, 
and so the pair $(Z_\fm, \partial Z_\fm)$ is an example of log Calabi-Yau 3-fold with maximal boundary.

Let 
\[ \overline{M}_{g, p} (Z_{\fm}/\partial Z_\fm)\] 
be the moduli space of 
genus $g$ stable log maps to
$Z_{\fm}$, of class 
$\beta_p$, with contact order along $\partial Z_{\fm}$ given by $\ell_p m_p$.
This is a proper Deligne-Mumford stack
coming with a $1$-dimensional virtual fundamental class
\[ [\overline{M}_{g, p} (Z_{\fm}/\partial Z_{\fm})]^{\virt}\,.\]
Composing the evaluation map at the contact point with $ \partial Z_\fm$ 
and
the projection $\partial Z_\fm \rightarrow \PP^1$, we obtain a map 
$\rho \colon 
\overline{M}_{g, p} (Z_{\fm}/\partial Z_{\fm})
\rightarrow \PP^1$ and we define log Gromov-Witten 
invariants 
\[N_{g,p}^{Z_\fm}
\coloneqq \int_{[\overline{M}_{g, p} (Z_{\fm}/\partial Z_{\fm})]^{\virt}} 
\rho^{*}(\pt)  \,, \]
where 
$\pt \in A^1(\PP^1)$ is the class of a point.

\begin{lem} \label{lem_3d_2}
We have 
\[N_{g,p}^{Z_\fm} = N_{g,p}^{Y_\fm} \,.\]
\end{lem}

\begin{proof}
We use virtual 
localization 
\cite{MR1666787} with respect to the action of $\C^{*}$ on the $\PP^1$-factor with weight $t$ at $0$ and weight $-t$ at 
$\infty$.
We choose $\pt_0$
as equivariant lift of 
$\pt \in A^1(\PP^1)$. 
Because of the insertion of 
$\pt_0 =t$, only the fixed point $0 \in \PP^1$, 
and not $\infty \in \PP^1$, contributes to the localization formula, and so we obtain
\[N_{g,p}^{Z_\fm} 
= t N_{g,p}^{Y_\fm \times \C}  \,,\]
and hence the result by Lemma
\ref{lem_3d_1}.
\end{proof}

\subsection{Orbifold Gromov-Witten theory}
\label{section_orbifold_gw}

We give an orbifold generalization of 
Sections \ref{section_log_cy_surface},
\ref{section_curve_classes_log_cy}, \ref{section_log_gw_log_cy}, which will be 
necessary to state Theorem 
\ref{precise_main_thm_orbifold} 
in Section \ref{section_BP_orbifold}. 

As in Section \ref{section_log_cy_surface}, we fix 
$\fm = (m_1, \dots, m_n)$ an $n$-tuple of primitive non-zero vectors of $M=\Z^2$ 
and this defines a toric surface $\overline{Y}_\fm$, with toric divisors 
$D_{m_j}$ for $1 \leqslant j \leqslant n$. For every 
$\fr = (r_1,\dots, r_n)$ an $n$-tuple of positive integers, we define a projective surface 
$Y_{\fm,\fr}$ by blowing-up a subscheme of length $r_j$ in general position on the toric divisor $D_{m_j}$, 
for every $1 \leqslant j \leqslant n$.
For $\fr=(1, \dots, 1)$, we simply have $Y_{\fm,\fr}
=Y_\fm$, which was defined in Section \ref{section_log_cy_surface}.

Let $\nu \colon Y_{\fm,\fr} \rightarrow \overline{Y}_\fm$ be the blow-up morphism. If $r_j \geqslant 2$, then 
$Y_{\fm,\fr}$ has an $A_{r_j-1}$-singularity on the exceptional divisor $E_j \coloneqq \nu^{-1}(x_j)$. 
We will consider $Y_{\fm, \fr}$ as a Deligne-Mumford stack by taking the natural structure of 
smooth Deligne-Mumford stack on a $A_{r_j-1}$ singularity. The exceptional divisor $E_j$ is then a
stacky projective line 
$\PP^1[r_j,1]$, with a single $\Z/r_j$ stacky point
$0 \in \PP^1[r_j,1]$. The normal bundle to $E_j$ in $Y_{\fm,\fr}$ is the orbifold line bundle
$\cO_{\PP^1[r_j,1]}(-[0]/(\Z/r_j))$ of degree $-1/r_j$, and in particular we have $E_j^2=-1/r_j$.

Denote by $P_{\fr}$ the set of $(p_1, \dots, p_n)
\in P=\NN^n$ such that $r_j$ divides $p_j$, for every 
$1 \leqslant j \leqslant n$.
Exactly as in Section \ref{section_curve_classes_log_cy}, we define for every $p \in P_\fr$ a curve class 
$\beta_p \in H_2(Y_{\fm,\fr},\Z)$. The only difference is that now we have 
\[ \beta_p \cdot E_j = \frac{p_j}{r_j} \,.\]

We denote by $\partial Y_{\fm, \fr}$ the strict transform of the toric boundary divisor $\partial \overline{Y}_\fm$ of $
\overline{Y}_\fm$, and we endow $Y_\fm$ with the divisorial log structure define by 
$\partial Y_\fm$. So we view $Y_{\fm, \fr}$ as a smooth Deligne-Mumford log stack. 
Because the 
non-trivial stacky structure is disjoint from the divisor $\partial Y_{\fm, \fr}$ supporting the non-trivial log structure, 
there is no difficulty in combining orbifold Gromov-Witten theory,
\cite{MR2450211}, \cite{MR1950941}, with log Gromov-Witten theory,
\cite{MR3011419},
\cite{MR3224717}, \cite{MR3257836}, 
to get a 
moduli space $\overline{M}_{g,p}(Y_{\fm,\fr}/\partial Y_{\fm,\fr})$ of genus $g$ stable log maps to 
$Y_{\fm,\fr}$, of class $\beta_p$, with contact order along $\partial Y_{\fm,\fr}$ given by 
$\ell_p m_p$. It is a proper Deligne-Mumford stack 
coming with a $g$-dimensional virtual fundamental class
\[ [\overline{M}_{g,p}(Y_{\fm,\fr}/\partial Y_{\fm,\fr})]^\virt \,.\]
We finally define genus $g$ orbifold log Gromov-Witten invariants  $N_{g,p}^{Y_{\fm,\fr}}
\in \Q$ of $Y_{\fm,\fr}$
by 
\[ N_{g,p}^{Y_{\fm,\fr}} 
\coloneqq \int_{[\overline{M}_{g,p}(Y_{\fm,\fr}/\partial Y_{\fm,\fr})]^\virt}
(-1)^g \lambda_g \,.\]

\section{Main results}
\label{section_main_result}

In Section \ref{section_statement}, 
we state the main result of the present paper,
\mbox{Theorem 
\ref{precise_main_thm_ch2}}, precise form of
Theorem \ref{main_thm_ch2} mentioned in the 
Introduction. 
In Section \mbox{\ref{section_examples}}, we give 
elementary examples illustrating 
Theorem \ref{precise_main_thm_ch2}. 
In \mbox{Section \ref{section_orbifold_gen}}, 
we state Theorem \ref{precise_main_thm_orbifold}, a generalization of Theorem \ref{precise_main_thm_ch2} including 
orbifold geometries. Finally, we give in Section \ref{section_more_general} some brief comments about the level of generality
of Theorems \ref{precise_main_thm_ch2} and 
\ref{precise_main_thm_orbifold}.

\subsection{Statement}
\label{section_statement}
Using the notations of Section
\ref{section_scattering}, we define a family of consistent quantum scattering diagrams.
Our main result, Theorem
\ref{precise_main_thm_ch2}, is that the Hamiltonians attached to the rays of these quantum scattering diagrams 
are generating series of the higher genus log Gromov-Witten invariants defined in Section 
\ref{section_log_calabi_yau}.

We fix $\fm=(m_1, \dots, m_n)$ an $n$-tuple of primitive non-zero vectors of $M$.
We denote $P \coloneqq \NN^n$ and we take
$R \coloneqq \C[\![P]\!]
=\C [\![t_1, \dots, t_n]\!]$
as complete local $\C$-algebra. 
Let $\hat{\fD}_{\fm}$ be the quantum scattering diagram over $R$
consisting of 
incoming rays 
$(\fd_j, \hat{H}_{\fd_j})$ for $1 \leqslant j \leqslant n$, where 
\[ \fd_j = - \R_{\geqslant 0} m_j \,,\]
and 
\[ \hat{H}_{\fd_j} = -i
\sum_{\ell \geqslant 1}\frac{1}{\ell} \frac{(-1)^{\ell-1}}{2 \sin \left( \frac{\ell \hbar}{2} \right)} t_j^\ell \hat{z}^{\ell m_j} = 
\sum_{\ell \geqslant 1}\frac{1}{\ell} \frac{(-1)^{\ell-1}}{q^{\frac{\ell}{2}}-q^{-\frac{\ell}{2}}} t_j^\ell \hat{z}^{\ell m_j} \,,\]
where $q=e^{i \hbar}$.

Let $S(\hat{\fD}_{\fm})$ be the corresponding consistent
quantum scattering diagram
given by 
Proposition \ref{prop_consistency}, 
obtained by adding outgoing rays to $\hat{\fD}_{\fm}$. 
We can assume that, for every 
$m \in M-\{0\}$ primitive, $S(\hat{\fD}_{\fm})$
contains a unique outgoing ray of support 
$\R_{\geqslant 0}m$.

For every 
$m \in M-\{0\}$ primitive, let $P_m$ be the subset of 
$p=(p_1, \dots, p_n) \in P=\NN^n$ such that 
$\sum_{j=1}^n p_j m_j$ is positively collinear with $m$:
\[ \sum_{j=1}^n p_j m_j = \ell_p m \]
for some $\ell_p \in \NN$.

Recall that in Section
\ref{section_log_calabi_yau}, for
every 
$\fm=(m_1,\dots, m_n)$, we introduced
a log Calabi-Yau surface $Y_{\fm}$
and for every 
$p=(p_1, \dots, p_n) \in P=\NN^n$, 
we defined a certain genus $g$ log Gromov-Witten $N_{g,p}^{Y_\fm}$ of $Y_{\fm}$.

\begin{thm} \label{precise_main_thm_ch2}
For every $\fm
=(m_1, \dots, m_n)$
an $n$-tuple 
of primitive non-zero vectors in $M$ and every
$m \in M-\{0\}$ primitive, the Hamiltonian $\hat{H}_m$
attached to the outgoing ray 
$\R_{\geqslant 0}m$ in the consistent quantum scattering 
diagram $S(\hat{\fD}_{\fm})$ is given by 
\[\hat{H}_m =
\left(-\frac{i}{\hbar}\right) 
\sum_{p \in P_m}
 \left(\sum_{g \geqslant 0}
N_{g,p}^{Y_\fm} \, \hbar^{2g} \right) \left( \prod_{j=1}^n t_j^{p_j} 
\right) \hat{z}^{\ell_p m} \,.\]
\end{thm}

In the classical limit 
$\hbar \rightarrow 0$, Theorem 
\ref{precise_main_thm_ch2} reduces to the main 
result (Theorem 5.4) of 
\cite{MR2667135}, expressing the classical
scattering diagram 
$S(\fD_\fm)$ in terms of the genus zero log 
Gromov-Witten invariants $N_{0,p}^{Y_\fm}$.
In \cite{MR2667135}, the genus 
$0$ invariants are defined as relative Gromov-Witten invariants of some open geometry. 
The fact that they coincide with genus $0$ log Gromov-Witten invariants follows from arguments 
of the type of those used in Section 5 of \cite{bousseau2017tropical}.

The proof of Theorem \ref{precise_main_thm_ch2}
takes Sections 
\ref{section_gw_toric},
\ref{section_degeneration}
\ref{section_tropical},
and \ref{section_end_proof}.
In \mbox{Section 
\ref{section_log_gw_log_cy}}, we define 
higher genus log Gromov-Witten 
invariants $N_{g,w}^{\overline{Y}_\fm}$
of toric surfaces $\overline{Y}_\fm$. 
In Section \ref{section_degeneration}, we prove a degeneration formula expressing the 
log Gromov-Witten invariants $N_{g,p}^{Y_\fm}$ of the log Calabi-Yau surface $Y_{\fm}$
in terms of log Gromov-Witten invariants
$N_{g,w}^{\overline{Y}_\fm}$ of 
the toric surface $\overline{Y}_\fm$.
In Section \ref{section_tropical}, we review, 
following \cite{MR3383167}, the relation between quantum scattering diagrams and Block-Göttsche 
$q$-deformed tropical curve count. 
In Section \ref{section_end_proof}, we conclude the proof by using the main result of 
\cite{bousseau2017tropical} relating $q$-deformed tropical curve count and higher genus 
log Gromov-Witten invariants of toric surfaces.

The consistency of the quantum scattering diagram $S(\hat{\fD}_\fm)$
translates into the fact that the product,
ordered according to the phase of the rays, of the 
elements $\hat{\Phi}_{\hat{H}_j}$ and $\hat{\Phi}_{\hat{H}_m}$ of the quantum tropical vertex group 
$\hat{\V}^{\hbar}_R$ is equal to the identity. Therefore, one can paraphrase Theorem
\ref{precise_main_thm_ch2} by saying that 
the log Gromov-Witten invariants $N_{g,p}^{Y_\fm}$ produce relations in the quantum tropical vertex group 
$\hat{\V}^{\hbar}_R$, or conversely that relations in $\hat{\V}^{\hbar}_R$ give constraints 
on the log Gromov-Witten invariants $N_{g,p}^{Y_\fm}$.

The automorphism $\hat{\Phi}_{\hat{H}_j}$
attached to the incoming rays 
$\fd_j$ of the quantum scattering diagram 
$S(\hat{\fD}_\fm)$ are conjugation by $e^{\hat{H}_{\fd_j}}$, ie
by 
\[ \exp \left(
\sum_{\ell \geqslant 1} \frac{1}{\ell}
\frac{(-1)^{\ell-1}}{q^{\frac{\ell}{2}}
-q^{-\frac{\ell}{2}}} t_j^\ell \hat{z}^{\ell m_j}\right) \,,\]
which can be written as $\Psi_q(-t_j
\hat{z}^{m_j})$
where
\[\Psi_q(x) \coloneqq \exp \left( -\sum_{\ell \geqslant 1} 
\frac{1}{\ell} \frac{x^\ell}{q^{\frac{\ell}{2}}-
q^{-\frac{\ell}{2}}} \right)
= 
\prod_{k \geqslant 0} \frac{1}{1-q^{k+\frac{1}{2}}x} \,,\]
is the quantum dilogarithm.
We warn that various conventions are used for the quantum dilogarithm throughout the literature.
We refer for example to 
\cite{MR2290758} for a nice review of the many aspects of the dilogarithm, 
including its quantum version.

As the incoming rays of $S(\hat{\fD}_{\fm})$ are expressed in terms of quantum dilogarithms, it is natural to ask if the 
outgoing rays, which by 
\mbox{Theorem 
\ref{precise_main_thm_ch2}} are generating series of higher genus log Gromov-Witten invariants, 
can be naturally expressed in terms of quantum dilogarithms. 
This question is related to the multicover/BPS structure of higher genus log Gromov-Witten theory 
and is fully answered by Theorem \ref{main_thm_integ} in Section
\ref{section_integrality}.

\subsection{Examples}
\label{section_examples}

In this Section, we give some elementary examples illustrating Theorem \ref{precise_main_thm_ch2}.

\subsubsection{Trivial scattering: propagation of a ray.} \label{example_propagation}
We take $n=1$ and $\fm=(m_1)$ with 
$m_1 =(1,0) \in M =\Z^2$. In this case, $R=\C[\![t_1]\!]$, and the quantum scattering diagram 
$\hat{\fD}_\fm$ contains a unique incoming ray: 
$\fd_1=-\R_{\geqslant 0}(1,0)=\R_{\geqslant 0}(-1,0)$
equipped with 
\[\hat{H}_{\fd_1}= -i \sum_{\ell \geqslant 1} 
\frac{1}{\ell} \frac{(-1)^{\ell-1}}{2 \sin \left( 
\frac{\ell \hbar}{2}
\right)} t_1^\ell \hat{z}^{(\ell,0)} \,.\]
Then the consistent scattering diagram $S(\hat{\fD}_\fm)$ is obtained by simply 
propagating the incoming ray, ie by adding the outgoing ray $\R_{\geqslant 0}(1,0)$
equipped with 
\[\hat{H}_{(1,0)}= -i \sum_{\ell \geqslant 1} 
\frac{1}{\ell} \frac{(-1)^{\ell-1}}{2 \sin \left( 
\frac{\ell \hbar}{2}
\right)} t_1^\ell \hat{z}^{(\ell,0)} \,.\]
We start with a fan consisting of the ray 
$\R_{\geqslant 0}(-1,0)$. To get a proper toric surface, we add to the fan the rays
$\R_{\geqslant 0}(1,0)$, $\R_{\geqslant 0}(0,1)$
and $\R_{\geqslant 0}(0,-1)$.
The corresponding toric surface $\overline{Y}_\fm$ is simply $\PP^1 \times \PP^1$. We obtain
$Y_\fm$ by blowing-up a point on 
$\{0\} \times \C^{*}$, eg $\{0\}
\times \{1\}$. Denote by $E$ the exceptional divisor and $F$ the strict transform of 
$\PP^1 \times \{1\}$. We have 
$E^2=F^2=-1$ and $E\cdot F=1$.
For $\ell \in P=\NN$, we have $\beta_\ell = \ell [F]$.
So, according to Theorem 
\ref{precise_main_thm_ch2}, one should have, for every $\ell \geqslant 1$,  
\[  \sum_{g \geqslant 0}
N_{g,\ell}^{Y_\fm} \hbar^{2g-1} 
=
\frac{1}{\ell} \frac{(-1)^{\ell-1}}{2 \sin \left( 
\frac{\ell \hbar}{2}
\right)}  \,.\]
As $F$ is rigid, contributions to $N_{g,\ell}$ only come from $\ell$ to $1$ multicoverings
of $F$ and the computation of $N_{g,\ell}$ can be reduced to a computation in relative Gromov-Witten theory of $\PP^1$. 
Using Theorem 5.1 of
\cite{MR2115262}, one can check that the above formula is indeed correct. We refer for more details to Lemma 
\ref{lem_multicover} which plays a crucial role in the proof of Theorem \ref{precise_main_thm_ch2}.

\subsubsection{Simple scattering of two rays}
\label{example_simple}
We take $n=2$ and $\fm=(m_1,m_2)$
with $m_1=(1,0)\in M=\Z^2$ and $m_2=(0,1) \in M=\Z^2$. In this case, $R=\C[\![t_1,t_2]\!]$, 
and the quantum scattering diagram 
$\hat{\fD}_{\fm}$ contains two incoming rays 
$\fd_1=\R_{\geqslant 0}(-1,0)$ and 
$\fd_2=\R_{\geqslant 0}(0,-1)$,
respectively equipped with 
\[\hat{H}_{\fd_1}= -i \sum_{\ell \geqslant 1} 
\frac{1}{\ell} \frac{(-1)^{\ell-1}}{2 \sin \left( 
\frac{\ell \hbar}{2}
\right)} t_1^\ell \hat{z}^{(\ell,0)} \,,\]
and 
\[\hat{H}_{\fd_2}= -i \sum_{\ell \geqslant 1} 
\frac{1}{\ell} \frac{(-1)^{\ell-1}}{2 \sin \left( 
\frac{\ell \hbar}{2}
\right)} t_2^\ell \hat{z}^{(0,\ell)} \,.\]
Then, because of the
Faddeev-Kashaev \cite{MR1264393}
pentagon identity
\[\Psi_q(z^{(1,0)}) \Psi_q(z^{(0,1)})
=\Psi_q(z^{(0,1)}) \Psi_q(z^{(1,1)})  \Psi_q(z^{(1,0)})\]
satisfied by the quantum dilogarithm 
$\Psi_q$,
the consistent scattering diagram $S(\hat{\fD}_\fm)$ is obtained by propagation of the two incoming rays in outgoing rays, 
as in \ref{example_propagation}, and by addition of a third outgoing ray 
$\R_{\geqslant 0}(1,1)$ equipped with 
\[\hat{H}_{(1,1)}= -i \sum_{\ell \geqslant 1} 
\frac{1}{\ell} \frac{(-1)^{\ell-1}}{2 \sin \left( 
\frac{\ell \hbar}{2}
\right)} t_1^\ell t_2^\ell \hat{z}^{(\ell,\ell)} \,.\]
We start with the fan consisting of the rays $\R_{\geqslant 0}(-1,0)$ and
$\R_{\geqslant 0}(0,-1)$. To get a proper 
toric surface, we can for example add to 
the fan the ray $\R_{\geqslant 0} (1,1)$. The corresponding toric surface $\overline{Y}_\fm$ is simply 
$\PP^2$, with its toric divisors $D_1$, $D_2$, $D_3$. 
We obtain $Y_\fm$ by blowing a point $p_1$ on $D_1$ and a point $p_2$ on $D_2$, both away from the torus fixed points. 
We denote by $E_1$ and $E_2$ the corresponding exceptional divisors and $F$ the strict transform of the unique line in $\PP^2$ passing through $p_1$ and $p_2$. 
We have $E_1^2=E_2^2=F^2=-1$ and 
$E_1 \cdot F=E_2 \cdot F=1$. 
For $\ell \in \NN$ and $(\ell,\ell) \in P=\NN^2$, we have $\beta_{(\ell,\ell)} = \ell [F]$. So according to Theorem 
\ref{precise_main_thm_ch2}, one should have, for every 
$\ell \geqslant 1$,
\[\sum_{g \geqslant 0} N_{g,(\ell,\ell)}^{Y_\fm} \hbar^{2g-1} 
= \frac{1}{\ell} \frac{(-1)^{\ell-1}}{2 \sin \left( 
\frac{\ell \hbar}{2}
\right)}  \,.\]
As $F$ is rigid, contributions to $N_{g,(\ell,\ell)}$
only come from $\ell$ to $1$ multicoverings of $F$ and the computation of $N_{g,(\ell,\ell)}$
reduces to a computation identical to the one used for $N_{g,\ell}$ in 
\ref{example_propagation}.

\subsubsection{More complicated scatterings}
Already at the classical level of 
\cite{MR2667135}, general scattering diagrams 
can be very complicated. A fortiori, general quantum scattering diagrams are extremely complicated. 
Direct computation of the higher genus log Gromov-Witten invariants $N_{g,p}^{Y_\fm}$ is
a difficult problem in general. In particular, unlike what happens in  
\ref{example_propagation} and \ref{example_simple}, linear systems defined by 
$\beta_p$ and the tangency condition contain
in general curves of positive genus, and so genus $g>0$ stable log maps appearing in the 
moduli space defining $N_{g,p}^{Y_\fm}$
do not factor through genus $0$ curves in general. As consistent scattering diagrams can be algorithmically computed, one can view
Theorem \ref{precise_main_thm_ch2} as an 
answer to the problem of effectively computing the higher genus log Gromov-Witten invariants 
$N_{g,p}^{Y_\fm}$.

\subsection{Orbifold generalization}
\label{section_orbifold_gen}

As in Section 5.5 and 5.6 of \cite{MR2667135} for the classical case,
we can give an enumerative interpretation of 
quantum scattering diagrams more general than those considered in Theorem 
\ref{precise_main_thm_ch2} if we allow ourselves to work with orbifold Gromov-Witten invariants. 

We fix $\fm=(m_1,\dots, m_n)$
an $n$-tuple of primitive non-zero vectors of $M=\Z^2$ and $\fr=(r_1, \dots, r_n)$ an 
$n$-tuple of positive integers. We denote 
$P \coloneqq \NN^n$ and we take 
$R \coloneqq \C[\![P]\!]
\coloneqq \C [\![t_1, \dots, t_n]\!]$ as 
complete local $\C$-algebra. Let $P_{\fr}$ be the set of $p=(p_1, \dots, p_n) \in P$ 
such that $r_j$ divides $p_j$ for every $1 \leqslant j \leqslant n$. 
Let $\hat{\fD}_{\fm,\fr}$ be the quantum scattering diagram over $R$ consisting of incoming rays 
$(\fd_j, \hat{H}_{\fd_j})$, 
$1 \leqslant j \leqslant n$, where 
\[ \fd_j=-\R_{\geqslant 0}m_j \,,\]
and 
\[ \hat{H}_{\fd_j}
=-i \sum_{\ell \geqslant 1}
\frac{1}{\ell} \frac{(-1)^{\ell-1}}{
2 \sin \left( \frac{r_j \ell \hbar}{2} \right)} t_j^{r_j \ell} \hat{z}^{r_j \ell m_j}
= \sum_{\ell \geqslant 1}
\frac{1}{\ell} \frac{(-1)^{\ell-1}}{q^{\frac{r_j \ell}{2}}
-q^{-\frac{r_j \ell}{2}}} t_j^{r_j \ell} \hat{z}^{r_j \ell m_j}\,,\]
where $q=e^{i \hbar}$. 
Let $S(\hat{\fD}_{\fm,\fr})$ be the corresponding consistent
quantum scattering diagram
given by 
Proposition \ref{prop_consistency}, 
obtained by adding outgoing rays to $\hat{\fD}_{\fm,\fr}$. 
For every 
$m \in M-\{0\}$, let $P_{\fr,m}$ be the subset of 
$p=(p_1, \dots, p_n) \in P_\fr$ such that 
$\sum_{j=1}^n p_j m_j$ is positively collinear with $m$:
\[ \sum_{j=1}^n p_j m_j = \ell_p m \]
for some $\ell_p \in \NN$.

Recall that in Section
\ref{section_orbifold_gw}, for
every 
$\fm=(m_1,\dots, m_n)$ and $\fr=(r_1,\dots,r_n)$, we introduced
an orbifold log Calabi-Yau surface $Y_{\fm,\fr}$
and for every 
$p=(p_1, \dots, p_n) \in P_\fr$, 
we defined a certain genus $g$ orbifold log Gromov-Witten $N_{g,p}^{Y_{\fm,\fr}}$ of $Y_{\fm,\fr}$.

\begin{thm} \label{precise_main_thm_orbifold}
For every $\fm
=(m_1, \dots, m_n)$
an $n$-tuple 
of primitive non-zero vectors in $M$,
every $\fr=(r_1, \dots, r_n)$ an 
$n$-tuple of positive integers and every
$m \in M-\{0\}$ primitive, the Hamiltonian $\hat{H}_m$
attached to the outgoing ray 
$\R_{\geqslant 0}m$ in the consistent quantum scattering 
diagram $S(\hat{\fD}_{\fm,\fr})$ is given by 
\[\hat{H}_m = 
\left(-\frac{i}{\hbar} \right)
\sum_{p \in P_{\fr,m}}
 \left(\sum_{g \geqslant 0}
N_{g,p}^{Y_{\fm,\fr}} \, \hbar^{2g} \right) \left( \prod_{j=1}^n t_j^{p_j} 
\right) \hat{z}^{\ell_p m} \,.\]
\end{thm}

For $\fr=(1, \dots, 1)$, Theorem 
\ref{precise_main_thm_orbifold} reduces to 
Theorem \ref{precise_main_thm_ch2}.
In the classical limit 
$\hbar \rightarrow 0$, Theorem 
\ref{precise_main_thm_ch2} reduces to Theorem 5.6 of 
\cite{MR2667135}.
The proof of Theorem 
\ref{precise_main_thm_orbifold} is entirely parallel to the proof of its special case Theorem 
\ref{precise_main_thm_ch2}. The key point is that orbifold and logarithmic questions never interact in a non-trivial way. 
The only major needed modification is an orbifold version of the multicovering formula of Lemma \ref{lem_multicover}. 
This is done in \mbox{Lemma 
\ref{lem_BP_orbifold}}, Section \ref{section_BP_orbifold}.

\subsection{More general quantum scattering diagrams}
\label{section_more_general}
We still fix $\fm=(m_1,\dots,m_n)$ an
$n$-tuple of primitive vectors of 
$M=\Z^2$ and we continue to denote $P=\NN^n$, so that $R =\C[\![P]\!]=\C[\![t_1, \dots, t_n]\!]$. 
One could try to further generalize 
Theorem \ref{precise_main_thm_orbifold}
by starting with a quantum scattering 
diagram over $R$ consisting of incoming rays $(\fd_j,\hat{H}_{\fd_j})$
for 
$1 \leqslant j \leqslant n$, where 
$\fd_j=-\R_{\geqslant 0}m_j$ and 
\[ \hat{H}_{\fd_j} = \sum_{\ell \geqslant 1}
\hat{H}_{\fd_j,\ell} t_j^{\ell} \hat{z}^{\ell m_j} \,,\]
for arbitrary 
\[ \hat{H}_{\fd_j,\ell} \in \C[\![\hbar]\!] \,.\] 

In the classical limit $\hbar \rightarrow 0$, Theorem 5.6 of \cite{MR2667135}, classical limit of our Theorem 
\ref{precise_main_thm_orbifold}, is enough to give an enumerative interpretation
of the resulting consistent scattering diagram in such generality.
Indeed, the genus $0$ orbifold Gromov-Witten story takes as input classical Hamiltonians  
\[ H_r=\sum_{\ell\geqslant 1}
\frac{(-1)^{\ell-1}}{r\ell^2}t^{r\ell}z^{r\ell}
=\frac{1}{r} (tz)^r + \cO((tz)^{r+1}) \,,\]
for all $r \geqslant 0$, which form a basis of $\C[(tz)]$. In particular, at 
every finite order in $\fm_R$,
every classical scattering diagram consisting of $n$ incoming rays meeting at $0 \in \R^2$ coincides
with a classical scattering diagram whose consistent completion has an enumerative interpretation in terms of genus $0$ orbifold Gromov-Witten invariants.
In the quantum case, things are more complicated due to the extra dependence in $\hbar$.
More precisely, Theorem \ref{precise_main_thm_orbifold} only covers a particular class of Hamiltonians $\hat{H}_{\fd_j}$ 
whose form is dictated by the 
multicovering structure of higher genus orbifold Gromov-Witten theory.

\section{Gromov-Witten theory of toric surfaces}
\label{section_gw_toric}

For every $\fm=(m_1, \dots, m_n)$
an $n$-tuple of primitive non-zero vectors 
in $M=\Z^2$, we defined in 
Section \ref{section_log_cy_surface} a log Calabi-Yau surface 
$Y_\fm$ obtained as the blow-up of some toric surface $\overline{Y}_\fm$, and we introduced in
Section \ref{section_log_gw_log_cy} 
a collection of log Gromov-Witten 
invariants $N_{g,p}^{Y_\fm}$ of $Y_\fm$. 
In the present Section, we define log Gromov-Witten invariants 
$N_{g,w}^{\overline{Y}_\fm}$ of the toric surface $\overline{Y}_\fm$. 
In Section \ref{section_degeneration}, we will compare the invariants $N_{g,p}^{Y_\fm}$ of 
$Y_\fm$ and 
$N_{g,w}^{\overline{Y}_\fm}$ of 
$\overline{Y}_\fm$.

\subsection{Curve classes on toric surfaces} \label{section_curve_classes_toric}

Recall from Section \ref{section_log_cy_surface} that $\overline{Y}_\fm$ 
is a proper toric surface whose fan contains the rays $-\R_{\geqslant 0}m_j$ for 
every $1 \leqslant j \leqslant n$. We denote by
$\partial \overline{Y}_\fm$ the toric boundary divisor of 
$\overline{Y}_\fm$. We will consider curves in 
$\overline{Y}_\fm$ meeting $\partial 
\overline{Y}_\fm$ in a
number of prescribed points with prescribed tangency conditions and at one
unprescribed point with prescribed tangency condition.
In this Section, we explain how to parametrize the relevant curve classes in terms of the prescribed tangency conditions $w_j$
at the prescribed points.

Let $s$ be a positive integer and let 
$w=(w_1,\dots, w_s)$ be an $s$-tuple of
non-zero vectors in $M$ such that for every 
$1 \leqslant r \leqslant s$, 
there exists $1 \leqslant j \leqslant n$ such that $-\R_{\geqslant 0}w_r
=-\R_{\geqslant 0}m_j$.
In particular, the ray $-\R_{\geqslant 0}w_r$
belongs to the fan of $\overline{Y}_\fm$ and we denote by
$D_{w_r}$ the corresponding toric divisor 
of $\overline{Y}_\fm$. Note that we can have 
$D_{w_r}=D_{w_{r'}}$ even if $r \neq r'$.
We denote by $|w_r| \in \NN$
the divisibility of $w_r \in M=\Z^2$, ie the largest positive integer $k$ such that one can write $w_r = k v$ with $v \in M$.
One should think about $w_r$ as defining a toric divisor $D_{w_r}$ and an intersection number $|w_r|$ with $D_{w_r}$ for a curve in 
$\overline{Y}_\fm$.

We assume that $\sum_{r=1}^s w_r \neq 0$ and so we can uniquely write  
\[ \sum_{r=1}^s w_r = \ell_w m_w \,,\]
with $m_w \in M$ primitive and 
$\ell_w \in \NN$.

We explain now how to define a curve class 
$\beta_w \in H_2(\overline{Y}_\fm, \Z)$.
In short, $\beta_w$ is the class of a curve in 
$\overline{Y}_\fm$ having for every $1 \leqslant r \leqslant s$, an intersection point of intersection number 
$|w_r|$ with $D_{w_r}$, and exactly one other intersection point with the toric boundary 
$\partial \overline{Y}_\fm$. 

More precisely,
the vector $m_w \in M$ belongs 
to some cone of the fan of $\overline{Y}_\fm$
and we write the corresponding decomposition 
\[ m_w=a_w^L m_w^L + a_w^R m_w^R\,,\] 
where 
$m_w^L, m_w^R \in M$
are primitive generators of rays of the fan of 
$\overline{Y}_\fm$ and where 
$a_w^L, a_w^R \in \NN$.
Note that there is only one term in this decomposition if the 
ray $\R_{\geqslant 0}m_w$ coincides with one of the rays of the fan of 
$\overline{Y}_\fm$. Let $D_w^L$ and 
$D_w^R$ be the toric divisors of 
$\overline{Y}_\fm$ corresponding to the rays 
$\R_{\geqslant 0} m_w^L$ and 
$\R_{\geqslant 0} m_w^R$.
Let $\beta_w \in H_2(\overline{Y}_\fm,\Z)$ be the curve class uniquely determined by the following intersection numbers with the toric divisors:
\begin{itemize}
\item The intersection numbers with those $D_{w_r}$
for 
$1 \leqslant r \leqslant s$ that are distinct from 
$D_w^L$ and $D_w^R$:
\[ \beta_w \cdot D_{w_r} = \sum_{r', D_{w_{r'}}
=D_{w_r}} |w_{r'}| \,,\]
\item The intersection number with $D_w^L$: 
\[\beta_w \cdot D_w^L=\ell_w a_w^L
+ \sum_{r, D_{w_r}=D_w^L} |w_r|\,.\]
\item The intersection number with $D_w^R$:
\[\beta_w \cdot D_w^R=\ell_w a_w^R
+\sum_{r, D_{w_r}=D_w^R}|w_r|\,. \]
\item The intersection number with 
every toric divisor $D$
different from $D_{w_r}$ for every
$1 \leqslant r \leqslant s$, and from 
$D_w^L$ and $D_w^R$: 
\[ \beta_w \cdot D=0\,.\]
\end{itemize}

Such class $\beta_w \in H_2(\overline{Y}_\fm, \Z)$ 
exists by standard toric geometry because of the 
relation $\sum_{r=1}^s w_r = \ell_w m_w$, and is unique. 

\subsection{Log Gromov-Witten invariant of toric surfaces} \label{section_log_gw_toric}

In the previous Section, given $w=(w_1,\dots, w_s)$ a $s$-tuple of non-zero vectors in $M$, we defined certain positive integers $\ell_w$, $a_w^L$ and 
$a_w^R$, certain toric divisors
$D_w^L$ and $D_w^R$ of $\overline{Y}_\fm$,
and a curve class 
$\beta_w \in H_2(\overline{Y}_\fm, \Z)$. 

We would like to consider genus $g$ stable maps
$f \colon C \rightarrow \overline{Y}_\fm$ of class $\beta_w$, intersecting $\partial Y_\fm$ in 
$s+1$ points, $s$ of them being on the divisors $D_{w_r}$ with
contact order $|w_r|$ for $1 \leqslant r \leqslant s$, 
and the last one having contact number 
$\ell_w a_w^L$ with the divisor
$D_w^L$ and 
contact order $\ell_w a_w^R$
with the divisor $D_w^R$. We also would like to fix the position of the $s$ intersection numbers with the divisors $D_{w_r}$. It is easy to check that the expected dimension of this enumerative problem is $g$. As in Section
\ref{section_log_gw_log_cy}, we will cut down the virtual dimension from $g$ to zero by integration of the top lambda class.

As in Section \ref{section_log_gw_log_cy}, 
to get proper moduli spaces, we work with stable log maps. 
We consider the divisorial 
log structure on $\overline{Y}_\fm$ defined by the toric divisor $\partial \overline{Y}_\fm$ and use it to view
$\overline{Y}_\fm$ as a smooth log scheme.
Let $\overline{M}_{g,w}(\overline{Y}_\fm,
\partial \overline{Y}_\fm)$ be the moduli space 
of genus $g$ stable log maps to 
$\overline{Y}_\fm$, of class $\beta_w$, 
with $s+1$ tangency conditions along $\partial \overline{Y}_\fm$ defined by the $s+1$ vectors $-w_1, \dots, -w_s, \ell_w m_w$ in $M$. 
It is a proper Deligne-Mumford stack coming with a $(g+s)$-dimensional
 virtual fundamental class
\[ [\overline{M}_{g,w}(\overline{Y}_\fm,
\partial \overline{Y}_\fm)]^{\virt}\,.\]
For every $1 \leqslant r \leqslant s$, 
we have an evaluation map
\[ \ev_r \colon \overline{M}_{g,w}(\overline{Y}_\fm,
\partial \overline{Y}_\fm) \rightarrow D_{w_r} \,.\]
If $\pi \colon \cC \rightarrow 
\overline{M}_{g,w}(\overline{Y}_\fm,
\partial \overline{Y}_\fm)$ is the universal curve,
of relative dualizing sheaf 
$\omega_\pi$, then the Hodge bundle 
$\E \coloneqq \pi_{*}\omega_\pi$
is a rank $g$ vector bundle over 
$\overline{M}_{g,w}(\overline{Y}_\fm,
\partial \overline{Y}_\fm)$, of top Chern class $\lambda_g \coloneqq c_g(\E)$.

We define log Gromov-Witten invariants $N_{g,w}^{\overline{Y}_\fm}  \in \Q$ by 
\[N_{g,w}^{\overline{Y}_\fm} 
\coloneqq \int_{[\overline{M}_{g,w}(\overline{Y}_\fm,
\partial \overline{Y}_\fm)]^{\virt}} 
(-1)^g \lambda_g \prod_{r=1}^s \ev_r^{*}(\pt_r) \,,\]
where $\pt_r \in A^1(D_{w_r})$ is the class of a point on $D_{w_r}$. This is a rigorous definition of the enumerative problem sketched at the beginning of this Section.

\section{Degeneration from log Calabi-Yau to toric}
\label{section_degeneration}

\subsection{Degeneration formula: statement}
\label{section_degeneration_statement}

We fix $\fm=(m_1, \dots, m_n)$
an $n$-tuple of primitive non-zero vectors 
in $M=\Z^2$. In
Section \ref{section_log_cy_surface}, we defined a log Calabi-Yau surface 
$Y_\fm$ obtained as blow-up of some toric surface $\overline{Y}_\fm$.
In 
\mbox{Section \ref{section_log_gw_log_cy}}, we introduced 
a collection of log Gromov-Witten 
invariants $N_{g,p}^{Y_\fm}$ of $Y_\fm$,
indexed by $n$-tuples $p=
(p_1,\dots,p_n) \in P=\NN^n$. 
In Section \ref{section_log_gw_toric}, we defined log Gromov-Witten invariants 
$N_{g,w}^{\overline{Y}_\fm}$ of the toric surface $\overline{Y}_\fm$
indexed by $s$-tuples 
$w=(w_1, \dots, w_s) \in 
M^s$.
The main result of this section, 
Proposition 
\ref{prop_degeneration}, is an explicit formula expressing the invariants 
$N_{g,p}^{Y_\fm}$ in terms of the 
invariants $N_{g,w}^{\overline{Y}_\fm}$.

We first need to introduce some notations to relate the indices 
$p=(p_1, \dots, p_n)$ in the invariants 
$N_{g,p}^{Y_\fm}$ and the indices $w=(w_1, \dots, w_s)$ in the invariants 
$N_{g,w}^{\overline{Y}_\fm}$. The way it goes is imposed by the degeneration formula in Gromov-Witten theory and hopefully will become conceptually clear in Section
\ref{section_contributing_tropical_curves}.

We fix $p=(p_1, \dots, p_n) \in P=\NN^n$. We call 
$k$ a partition of $p$, and we write 
$k \vdash p$, if $k$ is an $n$-tuple $(k_1, \dots, k_n)$, with $k_j$ a partition of $p_j$
for every $1 \leqslant j \leqslant n$.
We encode a partition $k_j$ of $p_j$ as a sequence $k_j=(k_{\ell, j})_{\ell \geqslant 1}$ of nonnegative integers, all zero except finitely many of them, such that 
\[ \sum_{\ell \geqslant 1} \ell k_{\ell, j}
= p_j \,.\]  
Given a partition $k$ of $p$, we define
\[ s(k)\coloneqq \sum_{j=1}^n \sum_{\ell \geqslant 1} k_{\ell, j}\,.\] 
We now define, given a partition $k$ of $p$, a $s(k)$-tuple 
\[ w(k)=(w_1(k), \dots, w_{s(k)}(k))\] of non-zero vectors in 
$M=\Z^2$, by the following formula:
\[w_r(k)\coloneqq \ell m_j\]
if 
\[ 1+\sum_{j'=1}^j \sum_{ \ell'=1}^{\ell-1}k_{\ell',j'}
\leqslant r \leqslant  k_{\ell, j} + \sum_{j'=1}^j \sum_{ \ell'=1}^{\ell-1 }k_{\ell', j'}\,.\]
In particular, for every 
$1 \leqslant j \leqslant n$ and 
$\ell \geqslant 1$, the $s(k)$-tuple $w(k)$ contains $k_{\ell, j}$ copies of the vector $\ell m_j \in M$. Note that because $m_j$
is primitive in $M$, we have $\ell=|w_r(k)|$, where 
$|w_r(k)|$ is the divisibility of 
$w_r$ in $M$.
Remark also that 
\[ \sum_{r=1}^{s(k)} w_r(k) = \sum_{j=1}^n
\sum_{\ell \geqslant 1} k_{\ell, j} \ell m_j
=\sum_{j=1}^n p_j m_j = \ell_p m_p  \,,\]
and so, 
comparing notations of Sections
\ref{section_curve_classes_log_cy} and
\ref{section_curve_classes_toric},  $\ell_{w(k)}=\ell_p$ and $m_{w(k)}=m_p$.

Using the above notations, we can now state Proposition 
\ref{prop_degeneration}.

\begin{prop} \label{prop_degeneration}
For every 
$\fm=(m_1, \dots, m_n)$ an 
$n$-tuple of primitive non-zero vectors 
in $M=\Z^2$, and 
every $p =(p_1, \dots, p_n) \in P=\NN^n$, 
the log Gromov-Witten invariants 
$N_{g,p}^{Y_\fm}$ of the log Calabi-Yau surface 
$Y_\fm$ are expressed in terms of the log Gromov-Witten invariants $N_{g,w}^{\overline{Y}_\fm}$
of the toric surface $\overline{Y}_\fm$ by the following formula:
\[ \sum_{g \geqslant 0}
N_{g,p}^{Y_\fm} \hbar^{2g-1}\]
\[ = \sum_{k \vdash p} \left(
\sum_{g \geqslant 0} N_{g,w(k)}^{\overline{Y}_\fm} \hbar^{2g-1+s(k)}
\right)  
\prod_{j=1}^n 
\prod_{\ell \geqslant 1} 
\frac{1}{k_{\ell, j}!} \ell^{k_{\ell, j}} \left( \frac{(-1)^{\ell-1}}{\ell} \frac{1}{2 \sin \left( 
\frac{\ell \hbar}{2} \right)}
\right)^{k_{\ell, j}} \,,\]
where the first sum is over all 
partitions $k$ of $p$. 
\end{prop}

The proof of Proposition \ref{prop_degeneration} takes the remainder of Section
\ref{section_degeneration}.
We consider the degeneration from $Y_\fm$ to $\overline{Y}_\fm$ introduced in 
\mbox{Section 5.3} of 
\cite{MR2667135} and we apply a higher genus version of the argument of \cite{MR2667135}. 
Because the general degeneration formula in log Gromov-Witten theory is not yet available, we give a proof of the needed degeneration formula following the general strategy used in 
\cite{bousseau2017tropical}, which uses specific vanishing properties of the top lambda class.
We assume for simplicity that 
$m_p$ is distinct from all $-m_j$. It is easy to adapt the argument in this special case.

\subsection{Degeneration set-up}
\label{section_deg_set_up}

We first review the construction of the degeneration considered in Section 5.3 of 
\cite{MR2667135}. 

We fix $\fm=(m_1, \dots, m_n)$
an $n$-tuple of primitive non-zero vectors 
in $M=\Z^2$.
Recall from Section \ref{section_log_cy_surface} that $
\overline{Y}_\fm$ 
is a proper toric surface whose fan contains the rays $-\R_{\geqslant 0}m_j$ for 
$1 \leqslant j \leqslant n$ and that we denote by
$D_{m_j}$ the corresponding toric 
divisors. For every $1 \leqslant j \leqslant n$, we also choose a point $x_j$ in general position on the toric divisor $D_{m_j}$.
Let $\overline{Y}_\fm \times \C \rightarrow \C$ be the trivial
family over $\C$ and let $\{x_j\} \times \C$ be the sections determined by the points $x_j$.
Up to doing some toric blow-ups, which do not change the log Gromov-Witten invariants
that we are considering by \cite{MR3778185},
we can assume that the divisors 
$D_{m_j}$ are disjoint.

The degeneration of $\overline{Y}_\fm$ to the normal cone of 
$D_{m_1} \cup \dots \cup D_{m_n}$,
\[\epsilon_{\overline{\cY}_\fm} \colon \overline{\cY}_\fm \rightarrow \C \,,\]
is obtained by blowing-up the loci
$D_{m_1}, \dots, D_{m_n}$
over $0 \in \C$ in $\overline{Y}_\fm
\times \C$.
The special fiber is given by 
\[\epsilon_{\overline{\cY}_\fm}^{-1}(0)=
\overline{Y}_\fm \cup \bigcup_{j=1}^n 
\PP_j \,,\]
where, denoting by $N_{D_{m_j}|\overline{Y}_\fm}$ the normal line bundle to $D_{m_j}$ in $
\overline{Y}_\fm$,
$\PP_j$ is the projective bundle over $D_{m_j}$ obtained by projectivization of the rank two vector bundle $\cO_{D_{m_j}} \oplus N_{D_{m_j}|
\overline{Y}_\fm}$ over $D_{m_j}$.
The embeddings 
\[ \cO_{D_{m_j}} \hookrightarrow \cO_{D_{m_j}} \oplus N_{D_{m_j}|\overline{Y}_\fm}
\,\,\,\, \text{and} \,\,\,\, 
N_{D_{m_j}|\overline{Y}_\fm} \hookrightarrow \cO_{D_{m_j}} \oplus N_{D_{m_j}|\overline{Y}_\fm} \] 
induce two sections of $\PP_j \rightarrow D_{m_j}$ that we denote respectively by
$D_{m_j,\infty}$ and $D_{m_j, 0}$.
In $\epsilon_{\overline{\cY}_\fm}^{-1}(0)$, the divisor $D_{m_j}$ in $\overline{Y}_\fm$ is glued to the divisor $D_{m_j,0}$ in $\PP_j$.
The strict transform of the section 
$\{x_j\} \times \C$ of $\overline{Y}_\fm \times \C$
is a section $S_j$ of $\epsilon_{\overline{\cY}_\fm}$, whose intersection with $\epsilon_{\overline{\cY}_\fm}^{-1}(0)$ is a point $x_{j,\infty} \in D_{m_j,\infty}$.

For every $1 \leqslant j \leqslant n$, we blow-up the section $S_j$ in $\overline{\cY}_\fm$ and
we obtain a family
\[ \epsilon_{\cY_\fm}  \colon 
\cY_\fm \rightarrow \C \,,\]
whose fibers away from zero are isomorphic to the surface $Y_\fm$, and whose special fiber 
is given by 
\[\cY_{m,0}\coloneqq\epsilon_{\cY_\fm}^{-1}(0)
=\overline{Y}_\fm \cup \bigcup_{j=1}^n \tilde{\PP}_j \,,\]
where $\tilde{\PP}_j$ is the blow-up of 
$\PP_j$ at all the points $x_{j',\infty}$ such that $\PP_{j'}=\PP_j$. 
We denote by $E_{j'}$ the corresponding exceptional divisor in $\tilde{\PP}_j$ and $C_{j'}$ the strict 
transform in $\tilde{\PP}_{j}$ of the unique 
$\PP^1$-fiber of $\PP_{j} \rightarrow
D_{m_j}$ containing $x_{j',\infty}$.
We have $E_{j'} \cdot C_{j'}=1$ in 
$\tilde{\PP}_j$. We still denote by 
$D_{m_j,0}$ and $D_{m_j,\infty}$ the strict transforms of $D_{m_j,0}$ and $D_{m_j,\infty}$. We denote by 
$\partial \tilde{\PP}_j$ the ``boundary" of 
$\tilde{\PP}_j$, which is the union of $D_{m_j,0}$, 
$D_{m_j,\infty}$, and of the strict transforms of the two $\PP^1$-fibers of 
$\PP_j \rightarrow D_{m_j}$ over the two intersection points of 
$D_{m_j}$ with the remaining part of 
$\partial \overline{Y}_{\fm}$.

We would like to obtain Proposition
\ref{prop_degeneration} by application of 
a degeneration formula in
log Gromov-Witten theory to the family
\[\epsilon_{\cY_\fm} \colon \cY_{\fm}
\rightarrow \C \,,\]
to relate the invariants $N_{g,p}^{Y_\fm}$ of the general fiber $Y_\fm$ to the invariants $N_{g,w}^{\overline{Y}_\fm}$ 
of $\overline{Y}_\fm$ which appears as component of the special fiber $\cY_{\fm,0}$.
In \cite{MR2667135}, Gross, Pandharipande and Siebert work with an ad hoc definition of the genus 0 invariants as relative 
Gromov-Witten invariants of some open geometry and they only need to apply 
the usual degeneration formula in relative Gromov-Witten theory.  
In our present setting, with log Gromov-Witten invariants in arbitrary genus, we cannot follow exactly the same path.

Because the general degeneration formula in log Gromov-Witten theory is not yet available, we follow the strategy used
in \cite{bousseau2017tropical}.
We first apply the decomposition formula of 
Abramovich, Chen, Gross and Siebert
\cite{abramovich2017decomposition}.
We then use the vanishing property of the top 
lambda class to restrict the terms appearing in the decomposition formula  
and to prove a gluing formula by working only with 
torically transverse stable log maps.
We review the decomposition formula of 
\cite{abramovich2017decomposition} in Section 
\ref{section_decomp_formula}. In
Section \ref{section_rigid}, we derive constraints on the terms contributing to the decomposition formula. 
In Section \ref{section_gluing}, we prove a 
gluing formula computing each of these terms.
We end the classification of the terms contributing to the decomposition formula in 
Section 
\ref{section_contributing_tropical_curves}.
We conclude the proof of Proposition 
\ref{prop_degeneration} in \mbox{Section 
\ref{section_end_proof_deg}}.

\subsection{Statement of the decomposition formula} \label{section_decomp_formula}

We consider $\cY_\fm$ as a smooth log scheme for the divisorial log structure defined by the union of the ``vertical" divisor 
$\cY_{\fm,0}$ with the strict transform of the ``horizontal" divisor $
\partial \overline{Y}_\fm \times \C$.
Viewing $\C$ as a smooth log scheme for the divisorial log structure defined by the divisor 
$\{0\} \subset \C$, we obtain that $\epsilon_{\cY_\fm}$
is a log smooth morphism.
Restricting to the special fiber gives a structure of log scheme on $\cY_{\fm,0}$
and a log smooth morphism 
\[\epsilon_{\cY_{\fm, 0}}
\colon \cY_{\fm,0}
\rightarrow \pt_\NN \,,\]
where the standard log point $\pt_{\NN}$ is obtained by restriction to $\{0\} \subset \C$ of the divisorial log structure on 
$\C$.

Let 
$\overline{M}_{g,p}(\cY_{\fm,0})$
be the moduli space of genus $g$
stable log maps to $\epsilon_{\cY_{\fm, 0}}
\colon
\cY_{\fm,0}
\rightarrow \pt_\NN$, of class $\beta_p$, with a marked point of contact order $\ell_p m_p$. 
This is a proper Deligne-Mumford stack coming with a $g$-dimensional virtual fundamental class
\[ [\overline{M}_{g,p}(\cY_{\fm,0})]^{\virt} \,.\] 
By deformation invariance of the virtual fundamental class on 
moduli spaces of stable log maps in log smooth families, we have 
\[N_{g,p}^{Y_\fm} 
= \int_{[\overline{M}_{g,p}(\cY_{\fm,0})]^{\virt}}(-1)^g \lambda_g \,.\]

The decomposition formula of \cite{abramovich2017decomposition} gives a decomposition 
of $[\overline{M}_{g,p}(\cY_{\fm,0})]^{\virt}$ indexed by tropical curves mapping to the tropicalization of
$\cY_{\fm,0}$. These tropical curves encode 
the intersection patterns of irreducible components of stable log maps mapping to the special fiber of the degeneration.
We refer to Appendix B of
\cite{MR3011419} 
and Section 2 of \cite{abramovich2017decomposition}
for the general notion of tropicalization of a log scheme. We denote by $\Sigma (X)$
the tropicalization of a log scheme $X$. The tropicalization of a log scheme is a cone complex, ie an abstract gluing of cones.

We start by describing the tropicalization 
$\Sigma (\cY_{\fm,0})$ of 
$\cY_{\fm,0}$. Tropicalizing the log morphism $\epsilon_{\cY_{\fm, 0}}
\colon
\cY_{\fm,0}
\rightarrow \pt_\NN$, we obtain a morphism 
of cone complexes
\[ \Sigma(\epsilon_{\cY_{\fm, 0}})
\colon
\Sigma(\cY_{\fm,0})
\rightarrow \Sigma(\pt_\NN) \,.\]
We have 
$\Sigma(\pt_\NN) = \R_{\geqslant 0}$
and $\Sigma(\cY_{\fm,0})$ is naturally
identified with the cone over the fiber 
$\Sigma(\epsilon_{\cY_{\fm, 0}})^{-1}(1)$ at $1 \in \R_{\geqslant 0}$. It is thus enough to describe the polyhedral complex 
\[\cY_{\fm,0}^\trop \coloneqq\Sigma(\epsilon_{\cY_{\fm, 0}})^{-1}(1) \,.\]
The polyhedral complex $\cY_{\fm,0}^\trop$
has one vertex $v_0$ dual to $\overline{Y}_\fm$ and vertices $v_j$ dual to $\tilde{\PP}_j$ for all $1 \leqslant j \leqslant n$. 
For every 
$1 \leqslant j \leqslant n$, there is an edge $e_{j,0}$ of integral length $1$, connecting $v_0$ and $v_j$, 
dual to $D_{m_j,0}$, and an unbounded edge $e_{j,\infty}$ attached to $v_j$, dual to $D_{m_j,\infty}$.

The best way to understand  $\cY_{\fm,0}^\trop$ probably 
is to think about it as a modification of the tropicalization of $
\overline{Y}_\fm$.
As $\overline{Y}_\fm$ is simply a toric surface, its tropicalization
$\Sigma(\overline{Y}_\fm)$ can be naturally identified with $\R^2$ endowed with the fan decomposition. In particular, $\Sigma(
\overline{Y}_\fm)$ has one 
vertex $v_0 = 0 \in \R^2$ and unbounded edges $-\R_{\geqslant 0} m_j$, attached to $v_0$ and dual to the toric boundary divisors $D_{m_j}$. 
To go from $\Sigma(\overline{Y}_\fm)$ to  $\cY_{\fm,0}^\trop$, 
one adds a vertex $v_j$ on each primitive integral point of $-\R_{\geqslant 0} m_j$, which has the effect of 
splitting 
$-\R_{\geqslant 0} v_j$ into a bounded edge $e_{j,0}$ 
and an unbounded edge $e_{j,\infty}$. 
One still has to cut along $e_{j, \infty}$ and to insert there two two-dimensional cones dual to the two ``corners" of $\partial \tilde{\PP}_j$
which are on $D_{m_j,\infty}$.
In particular, for every $1 \leqslant j \leqslant n$, the vertex $v_j$ is 4-valent and looks locally like the fan of the Hirzebruch surface $\PP_j$. 
In general, there is no global linear embedding of  $\cY_{\fm,0}^\trop$ in $\R^2$.

\begin{center}
\setlength{\unitlength}{1cm}
\begin{picture}(6,6)
\thicklines
\put(3,3){\circle*{0.1}}
\put(3,3){\line(1,0){4}}
\put(3,3){\line(1,2){1.5}}
\put(3,3){\line(-2,1){4}}
\put(3,3){\line(-2,-1){4}}
\put(3,3){\line(1,-2){1.5}}
\put(5,3.2){$-\R_{\geqslant 0}m_1$}
\put(2.8,3.2){$0$}
\put(-0.5,5){$-\R_{\geqslant 0}m_2$}

\end{picture}
\end{center}

\begin{center}
Figure: tropicalization of $\overline{Y}_\fm$.
\end{center}

\begin{center}
\setlength{\unitlength}{1cm}
\begin{picture}(6,6)
\thicklines
\put(3,3){\circle*{0.1}}
\put(3,3){\line(1,0){4}}
\put(3,3){\line(1,2){1.5}}
\put(3,3){\line(-2,1){4}}
\put(3,3){\line(-2,-1){4}}
\put(3,3){\line(1,-2){1.5}}
\put(6,3){\circle*{0.1}}
\put(6,3){\line(1,2){1}}
\put(6,3){\line(1,-2){1}}
\put(1,4){\circle*{0.1}}
\put(1,4){\line(1,2){1}}
\put(1,4){\line(-2,-1){2}}
\put(2.7,3.2){$v_0$}
\put(4,3.2){$e_{1,0}$}
\put(6.5,3.2){$e_{1,\infty}$}
\put(5.75,2.7){$v_1$}
\put(1.2,4){$v_2$}
\put(1.5,3.3){$e_{2,0}$}
\put(-1,4.5){$e_{2,\infty}$}
\end{picture}
\end{center}

\begin{center}
Figure: picture of 
$\cY_{\fm,0}^\trop$.
\end{center}

We refer to Definition 2.5.3 of 
\cite{abramovich2017decomposition}
for the general definition of parametrized 
tropical curve 
$h \colon \Gamma \rightarrow 
\cY_{\fm,0}^\trop$.
It is a natural generalization of the notion of parametrized tropical curve in 
$\R^2$ that we will use and review in
Section \ref{section_tropical_count}.
In particular, $\Gamma$ is a finite graph, with bounded and unbounded edges mapped by $h$ to $\cY_{\fm,0}^\trop$ in an affine linear way and vertices $V$ of $\Sigma$ are decorated by some genus $g(V)$. The total genus $g$ of the parametrized tropical curve 
is defined by $g_\Gamma + \sum_V g(V)$, where $g_\Gamma$ is the genus of the graph 
$\Gamma$. 

Some distinction between 
$\cY_{\fm,0}^\trop$ and $\R^2$, related to the fact that the components 
$\tilde{\PP}_j$ of 
$\cY_{\fm,0}$ are non-toric, is that the usual form of the balancing condition for a tropical curve in $\R^2$ is not necessarily valid at vertices of $\Gamma$ mapping to a vertex $v_j$ of 
$\cY_{\fm,0}^\trop$ for some $1 \leqslant j \leqslant n$. 
For vertices of 
$\Gamma$ mapping away from the vertices $v_1,\dots,v_n$ of 
$\cY_{\fm,0}^\trop$, the usual balancing condition applies. 

Following Definition 4.2.1 of 
\cite{abramovich2017decomposition}, a decorated parametrized tropical curve 
is a parametrized tropical curve 
$h \colon \Gamma \rightarrow \cY_{\fm,0}^\trop$ where each vertex $V$ has a further decoration by a curve class
$\beta(V)$ in the stratum of $\cY_{\fm,0}$ dual to the stratum of $\cY_{\fm,0}^\trop$ where this vertex is mapped.
The curve class $\beta(V)$ is assumed to be numerically compatible with the tangency conditions imposed by the edges incident to 
$V$. In short, a decorated parametrized tropical curve to $\cY_{\fm,0}^\trop$
encodes all the necessary combinatorial information to be a fiber of the tropicalization of a stable log maps to
$\cY_{\fm,0}$.

The decomposition formula of \cite{abramovich2017decomposition}
involves decorated 
parametrized tropical curves which are rigid in their combinatorial type. 
This is easy to understand intuitively: the decomposition formula is supposed to describe how the moduli space of stable log maps breaks into pieces under degeneration. 
If the moduli space of tropical curves were the tropicalization, and so the dual intersection complex, of the moduli space of stable log maps, components of the 
moduli space of stable log maps should be in bijection with the zero dimensional strata of the moduli space of tropical curves, 
ie with rigid tropical curves. 
According to the decomposition formula of \cite{abramovich2017decomposition}, this intuitive picture is correct at the virtual level.

The tropical curves relevant to the study of $\overline{M}_{g,p}(\cY_{\fm,0})$ are genus $g$ decorated parametrized tropical curves 
$\Gamma \rightarrow \cY_{\fm,0}^\trop$ of type $p$, ie with only one unbounded edge, of weight $\ell_p$ and of direction $m_p$, and with total curve class $\beta_p$.
According to Section 4.4 
of \cite{abramovich2017decomposition}, 
for every such $h \colon 
\Gamma \rightarrow \cY_{\fm,0}^{\trop}$, we have the notion of stable log map marked by $h$, and a moduli space  
$\overline{M}_{g,p}^h(\cY_{\fm,0})$
of stable log maps marked by $h$,  
which is a proper Deligne-Mumford stack
equipped with a virtual fundamental class
$[\overline{M}_{g,p}^h(\cY_{\fm,0})]^\virt$. Forgetting the marking by 
$h$ gives a morphism
\[i_{h} \colon 
\overline{M}_{g,p}^{h}(\cY_{\fm,0}) 
\rightarrow \overline{M}_{g,p}^{h}(\cY_{\fm,0})\,.\]
We can finally state the decomposition formula, Theorem 4.8.1 of 
\cite{abramovich2017decomposition}:
\[[\overline{M}_{g,p}(\cY_{\fm,0})]^\virt
= \sum_{h}  
\frac{n_h}{|\Aut (h)|}
(i_{h })_* [\overline{M}_{g,p}^{h}(\cY_{\fm,0})]^\virt \,,\]
where the sum is over 
rigid genus $g$ decorated parametrized 
tropical curves 
$h \colon \Gamma \rightarrow 
\cY_{\fm,0}^\trop$
of type $p$, $n_h$ is the smallest positive integer such that after scaling by $n_h$, 
$h$ gets integral vertices and integral lengths, and $|\Aut(h)|$ is the order of the automorphism group of $h$.

\subsection{Constraints on rigid tropical curves}
\label{section_rigid}

In order to extract 
some explicit information
from the decomposition formula, the first step is to identify the rigid decorated parametrized 
tropical curves $h \colon 
\Gamma \rightarrow 
\cY_{\fm, 0}^\trop$ of type $p$. 
This is in general a difficult question.
But because we are
only interested in numerical invariants obtained by integration of the top lambda class 
$\lambda_g$, and not in the full virtual class, the situation is much simpler by the following vanishing result.

\begin{lem} \label{lem_rigid_higher_genus}
Let $h \colon \Gamma \rightarrow 
\cY_{\fm,0}^\trop$ be a rigid 
genus $g$ decorated parametrized
tropical curve of type $p$ with 
$\Gamma$ of positive genus. Then we have 
\[ \int_{[\overline{M}^h_{g,p}(\cY_{\fm,0})]^\virt} (-1)^g
\lambda_g = 0 \,.\] 
\end{lem}

\begin{proof}
If $f \colon C \rightarrow 
\cY_{\fm,0}^\trop$ is a 
stable log map in 
$\overline{M}^h_{g,p}(\cY_{\fm,0})$, then, by definition of the 
marking by $h$, the dual intersection complex of $C$ retracts onto $\Gamma$ and in particular, has genus bigger than the genus of $\Gamma$, which is positive by hypothesis. It follows that $C$ contains a cycle
of irreducible components.
It is then a general property of $\lambda_g$ that it vanishes on families of curves containing cycles of irreducible components
(eg see Lemma 8 of 
\cite{bousseau2017tropical}).
\end{proof}

For every  $h \colon \Gamma \rightarrow 
\cY_{\fm,0}^\trop$ a rigid
genus $g$ decorated parametrized
tropical curve of type $p$, we define
\[N_{g,p}^{h}\coloneqq\int_{[\overline{M}_{g,p}^{h}(\cY_{\fm,0})]^\virt}
(-1)^g \lambda_g \,.\]

\begin{prop} \label{prop_decomp_corrected}
We have 
\[N_{g,p}^{Y_\fm} = 
\sum_h
 \frac{n_h}{|\Aut (h)|}
N_{g,p}^{h} \,,\]
where the sum is over rigid genus $g$
decorated parametrized tropical curves 
$h \colon \Gamma \rightarrow \cY_{\fm,0}^\trop$ of type $p$ with $\Gamma$ of genus $0$.
\end{prop}

\begin{proof}
This follows from integrating 
$(-1)^g \lambda_g$ over the decomposition formula of \cite{abramovich2017decomposition},
reviewed at the end of Section 
\ref{section_decomp_formula}.
By \mbox{Lemma \ref{lem_rigid_higher_genus}},
rigid tropical curves $h \colon 
\Gamma \rightarrow \cY_{\fm,0}^\trop$ with 
$\Gamma$ of positive genus do not contribute.
\end{proof}

\begin{prop} \label{prop_rigid_corrected}
Let
$h \colon \Gamma \rightarrow \cY_{\fm,0}^\trop$ 
be a rigid genus $g$ decorated parametrized tropical curve 
of type $p$, with $\Gamma$ of genus $0$ and $N_{g,p}^h \neq 0$.
Then,
\begin{itemize}
\item for every vertex $V$ of $\Gamma$, we have $h(V)=v_j$ for some $0 \leqslant j \leqslant n$,
\item every bounded edge $E$ of $\Gamma$ connects a vertex $V_0$ with $h(V_0)=v_0$ and a vertex $V_+$ with $h(V_+)=v_j$ for some 
$1 \leqslant j \leqslant n$. In particular, the image $h(E)$ is the bounded edge $e_{j,0}$ of 
$\mathcal{Y}_{\fm,0}^{\trop}$.
\item for every vertex $V$ of 
$\Gamma$ such that $h(V)=v_j$ for some 
$1 \leqslant j \leqslant n$, there exists a unique 
index $j(V)$ such that $1 \leqslant j(V) \leqslant n$ and the curve class $\beta(V)$ is a positive multiple of the curve class $[C_{j(V)}]$. In particular, we have 
$h(V)=v_{j(V)}$.
\end{itemize}
\end{prop}

The proof of Proposition \ref{prop_rigid_corrected} takes the remainder of Section \ref{section_rigid}.
The argument is similar to the one used in the proof of Proposition 11 of 
\cite{bousseau2017tropical}, itself a
tropical version of the properness argument
given in
Proposition 4.2 of \cite{MR2667135}.
By iterative application of the balancing condition, we will argue that the source $\Gamma$ of  a rigid decorated parametrized tropical curve 
$h \colon \Gamma \rightarrow \cY_{\fm,0}^\trop$ violating the conclusions of 
Proposition \ref{prop_rigid_corrected}
necessarily contains a cycle and so has positive genus. We refer to Proposition 1.15 of \cite{MR3011419} for the general form of the balancing condition in log Gromov-Witten theory.

Let $h \colon \Gamma \rightarrow \cY_{\fm,0}^\trop$ be a
rigid
genus $g$ parametrized tropical curve of type $p$. As $h$ is rigid, there is no edge of $\Gamma$ contracted by $h$. 
The fact that $h$ has type $p$ implies that $h$ has only one unbounded edge and this unbounded edge has
weight $\ell_p$
and direction 
$m_p$.

\begin{lem} \label{lem_loop}
If there exists a vertex $V$ of 
$\Gamma$ such that 
$h(V) \notin \{v_0, v_1, \dots, v_n\}$,
then $\Gamma$ has positive genus.
\end{lem}

\begin{proof}
We first assume 
$h(V)$ is contained in the interior of one of the 
two-dimensional cones $\cC$ of 
$\cY_{\fm,0}^\trop$.
Because $h(V)$ is away from the vertices 
$v_j$, the situation is locally toric and the balancing condition has to be satisfied at $h(V)$. If $h(V) \notin \R_{\geqslant 0} m_p$, there is no unbounded edge of $\Gamma$ incident to $V$, and so by balancing, not all edges attached to $h(V)$ can point towards the vertex of $\cC$, ie at least one of those edges points towards a boundary ray of $\cC$.
If $h(V) \in \R_{\geqslant 0}m_p$, we can get the same conclusion: if all edges passing through $h(V)$ were parallel to 
$\R_{\geqslant 0} m_p$, this would contradict the rigidity of $h$ because one could move $h(V)$ along $\R_{\geqslant 0}m_p$.

Next, we follow the proof of Proposition 11 of 
\cite{bousseau2017tropical}. Fixing a cyclic orientation on the collection of cones and rays of $\cY_{\fm,0}^\trop$, we can assume that this edge points towards the left (from the point of view of the vertex of the cone, looking inside the cone) ray of $\cC$. If this 
edge ends on some vertex still contained in the interior of $\cC$, then the balancing condition still applies and so there is still an edge attached to this vertex pointing towards the left ray of $\cC$.
Because $\Gamma$ has finitely many vertices, iterating this construction finitely many times, we construct a path starting from $h(V)$ and ending at some vertex $h(V')$ on the left boundary ray of 
$\cC$. 

Let $\cC'$ be the two-dimensional cone of $\cY_{\fm,0}^\trop$ incident to 
$\cC$ near $h(V')$. Then we claim that by the balancing condition, there exists an edge attached to $h(V')$ pointing towards the left ray of $\cC'$. Indeed, the only case for which the balancing condition is not a priori satisfied is if $h(V')=v_j$ for some $j$. But at $v_j$, the non-toric nature of $\tilde{\PP}_j$ only modifies the balancing condition in the direction parallel to $e_{j,0}$ and $e_{j,\infty}$: if there is an  incoming edge with non-zero transversal direction, then there is still an outgoing edge with non-zero transversal direction. Indeed, $\tilde{\PP}_j$ is obtained from the Hirzebruch surface $\PP_j \rightarrow D_{m_j}$ by blowing-up points on the divisors $D_{m_j,\infty}$: this does not affect the fact that the general fibers of 
$\tilde{\PP}_j \rightarrow D_{m_j}$ are still linearly equivalent.

Iterating this construction, we obtain a
path in $\Gamma$ whose image by $h$ in $\cY_{\fm,0}^\trop$ is a path which intersects successive rays in the anticlockwise order.
Because $\Gamma$ has finitely many edges, this path has to close eventually and so 
$\Gamma$ contains a non-trivial cycle, ie $\Gamma$ has positive genus.

It remains to treat the case where $h(V)$ is in the interior
of a one dimensional ray of $\cY_{\fm,0}^\trop$. If all the edges attached to $h(V)$ were parallel to the ray, this would contradict the rigidity of $h$ because one could move $h(V)$ along the ray. So at least one of the edges attached to $h(V)$ is not parallel to the ray and by balancing, we can assume that there is an edge attached to $h(V)$ pointing towards the 2-dimensional cone of $\cY_{\fm,0}^\trop$ left to the ray. We can then apply the iterative argument described above.
\end{proof}

We continue the proof of Proposition \ref{prop_rigid_corrected}.
We are assuming that $\Gamma$ has genus $0$. By Lemma 
\ref{lem_loop}, every vertex $V$ of $\Gamma$ maps to one of the vertices $v_0, v_1, \dots, v_n$ of
$\Gamma$. If there were an edge connecting a vertex mapped to $v_j$ with a vertex mapped to $v_{j'}$, $1 \leqslant j,j' \leqslant n$ with $j \neq j'$, then we could apply the iterative argument used in the proof of 
Lemma \ref{lem_loop}, and this would contradict the assumption that $\Gamma$ has genus $0$. 
It follows that every bounded edge in $\Gamma$ incident to some vertex mapped to $v_j$ for some 
$1 \leqslant j \leqslant n$
is also incident to some vertex mapped to $v_0$.   

Let $V$ be a vertex of $\Gamma$ such that 
$h(V)=v_j$ for some $1 \leqslant j \leqslant n$.
As we are assuming that $m_p \neq -m_j$ for every $j$, the unique unbounded edge of $\Gamma$ is not incident to $V$, and so all edges incident to $V$ are also 
incident to a vertex mapped to $v_0$. The only curve classes numerically compatible with these tangency conditions are positive multiples of the classes $[C_j]$. Therefore, there exists a unique $1 \leqslant j(V) \leqslant n$ such that $\beta(V)$ is a positive multiple of 
$[C_{j(V)}]$. This concludes the proof of
Proposition \ref{prop_rigid_corrected}.

\subsection{Gluing formula}
\label{section_gluing}

According to 
Proposition
\ref{prop_decomp_corrected}, 
the computation of the log Gromov-Witten invariants $N_{g,p}^{Y_\fm}$ is reduced 
to the 
computation of the invariants 
\[N_{g,p}^{h}\coloneqq\int_{[\overline{M}_{g,p}^{h}(\cY_{\fm,0})]^\virt}
(-1)^g \lambda_g \,,\]
where $h \colon \Gamma 
\rightarrow \cY_{\fm,0}^{\trop}$
is a rigid genus $g$ decorated parametrized tropical curve of type $p$
with $\Gamma$ of genus $0$, and  
where $\overline{M}_{g,p}^h(\cY_{\fm,0})$
is a moduli space of stable log maps
to $\cY_{\fm,0}$
marked by $h$, ie whose 
tropicalization is equipped with a retraction to $h$.

According to Proposition \ref{prop_rigid_corrected}, 
the image by $h$ of the bounded edges of 
$\Gamma$ is contained in the 1-dimensional part of the polyhedral decomposition of 
$\mathcal{Y}_{\fm,0}^{\trop}$.
Therefore, cutting the bounded edges, we obtained well-defined numerical data to define moduli spaces $M_V$ of stable log maps attached to the vertices $V$ of $\Gamma$. More precisely, we define the moduli spaces $M_V$ as follows.

For $V$ a vertex such that $h(V)=v_0$, 
$M_V$ is the moduli space of genus $g(V)$
and class $\beta(V)$ stable log maps to $\overline{Y}_\fm$ with tangency conditions along 
$\partial \overline{Y}_\fm$ 
defined by the weighted edges of $\Gamma$ attached to $V$. As $\overline{Y}_\fm$ is toric, the class $\beta(V)$ is uniquely determined by the numerical compatibility with the tangency conditions along 
$\partial \overline{Y}_{\fm}$.

Let $V$ be a vertex such that $h(V) \neq v_0$.
By Proposition
\ref{prop_rigid_corrected}, there exists a unique $1 \leqslant j(V) \leqslant n$
such that $h(V)=v_{j(V)}$ 
and $\beta(V)$ is a positive multiple of 
$[C_{j(V)}]$. We endow 
$\tilde{\PP}_{j(V)}$ with the divisorial log structure defined by 
$\partial \tilde{\PP}_{j(V)}$
Then,  
$M_V$ is the moduli space of genus $g(V)$
class $\beta(V)$ stable log maps to $\tilde{\PP}_{j(V)}$ 
with contact orders along $\partial \tilde{\PP}_j$ defined by the weighted edges of $\Gamma$ attached to $V$. 
As such an edge always connects $V$ with a vertex mapped to $v_0$, these contact orders are only non-trivial along the divisor $D_{m_{j(V)},0}$. 
The positive integer 
$\ell(V)$ such that $\beta(V)=\ell(V)[C_{j(V)}]$ is necessarily equal to the sum of weights of edges incident to $V$ 
by numerical compatibility of $\beta(V)$ with the tangency conditions.

For every bounded edge $E$ of $\Gamma$, we define $j(E) \coloneqq j(V)$, where $V$ 
is the unique vertex of $\Gamma$ incident to $E$ such that $h(V) \neq v_0$, and we set $D_E \coloneqq D_{m_{j(E)}}$.
 We denote by $w(E)$ the weight of $E$.

By general log Gromov-Witten theory
\cite{MR3011419}, 
each moduli space $M_V$ comes with a virtual fundamental class $[M_V]^{\virt}$.
For $V$ a vertex with $h(V)=v_0$, we define
\[ N_V \coloneqq \int_{[M_V]^{\virt}}(-1)^{g(V)} \lambda_{g(V)}
\prod_{V \in E} 
\ev_E^{*}(\pt_E) \,,\]
where the product is over the bounded edges $E$ incident to $V$, $\ev_E$ is the evaluation map at the corresponding contact point with the divisor $D_E$, and $\pt_E \in A^1(D_E)$ is the class of a point on $D_E$. 
For a vertex $V$ with $h(V) \neq v_0$, we define
\[ N_V \coloneqq \int_{[M_V]^{\virt}}(-1)^{g(V)} \lambda_{g(V)}
 \,.\]

\begin{prop} \label{prop_gluing_corrected}
For every $h \colon \Gamma 
\rightarrow \cY_{\fm,0}^{\trop}$
a rigid genus $g$ decorated parametrized tropical curve of type $p$
with $\Gamma$ of genus $0$ and $N_{g,p}^h \neq 0$, we have 
\[ N_{g,p}^h = 
\frac{
\prod_E w(E)}
{\lcm \{w(E)\}_E}
\prod_V N_V \,,\]
where the index $E$ runs over bounded edges of $\Gamma$ and the index 
$V$ over vertices of $\Gamma$.
\end{prop}

The proof of Proposition \ref{prop_gluing_corrected} takes the remainder of Section \ref{section_gluing}.
The moduli space
$\overline{M}_{g,p}^h(\cY_{\fm,0})$ parametrizes stable log maps
marked by $h$, and so in general contains stable log maps whose tropicalization is not $h$ but only retracts onto 
$h$. Such stable log maps might interact in a complicated way with the log structure of $\cY_{\fm,0}$ and their general cutting and gluing properties have not been worked out yet. 

We go around this issue by following the strategy used in \mbox{Section 6} of 
\cite{bousseau2017tropical}. On an open locus of torically transverse stable maps, the above mentioned problems do not arise and the difficulty of the gluing problem is of the same level as the
usual degeneration formula in relative Gromov-Witten theory. The log version of this gluing problem has been recently treated in full detail by Kim, Lho and Ruddat 
\cite{kim2018degeneration}. On the complement of the nice locus of torically transverse stable log maps, a combinatorial argument 
described in Proposition 11 of
\cite{bousseau2017tropical} 
implies that one of the relevant curves will always contain a non-trivial cycle of components. 
By standard vanishing properties of the lambda class, it follows that we can ignore this bad locus if we only care about numerical invariants obtained by integration of a top lambda class, which is our case.
 
We give now an outline of the proof, referring to 
\cite{kim2018degeneration} and \mbox{Section 6} of 
\cite{bousseau2017tropical} for some of the steps.
We have an evaluation morphism
\[ \ev 
\colon 
\prod_V M_V
 \rightarrow 
\prod_E D_E^2 \,,\]
where the left product is over the vertices of $\Gamma$ and where the right product is over the bounded edges of $\Gamma$.
Let 
\[ \delta \colon
\prod_E D_E
\rightarrow 
\prod_E D_E^2
 \]
be the diagonal morphism.
Using the morphisms $\ev$ and $\delta$, we define the 
fiber product 
\[ \cM \coloneqq 
\left(\prod_V M_V \right)
\times_{\prod_E D_E^2}
\left( \prod_E D_E \right) \,.\]
We define a cycle class $[\cM]^\virt$ on 
$\cM$
by 
\[ [\cM]^\virt \coloneqq 
\delta^! \left(
\prod_V [M_V]^\virt
\right) \,,\]
where $\delta^!$ is the refined Gysin morphism
(see Section 6.2 of \cite{MR1644323}) defined by $\delta$.

The following Lemma will play for us the same role
played by 
\mbox{Lemma 16} in Section 6 of \cite{bousseau2017tropical}.

\begin{lem} \label{lem_edges_cy}
Let 
\begin{center}
\begin{tikzcd}
C \arrow{r}{f} \arrow{d}{\pi}
& \cY_{\fm,0} \arrow{d}{\epsilon_{\cY_{\fm,0}}}\\
W \arrow{r}{g} & \mathrm{pt}_{\NN} \,,
\end{tikzcd}
\end{center}
be a point of $\overline{M}_{g,p}^h(\cY_{\fm,0})$.
Let 
\begin{center}
\begin{tikzcd}
\Sigma(C) \arrow{r}{\Sigma(f)} \arrow{d}{\Sigma(\pi)}
& \Sigma(\cY_{\fm,0}) \arrow{d}{\Sigma(\epsilon_{\cY_{\fm,0}})}\\
\Sigma(W) \arrow{r}{\Sigma(g)} & \Sigma(\mathrm{pt}_{\NN}) \,.
\end{tikzcd}
\end{center}
be its tropicalization.
For every $b \in \Sigma(g)^{-1}(1)$,
let 
\[ \Sigma(f)_b \colon
\Sigma(C)_b \rightarrow \Sigma(\epsilon_{\cY_{\fm,0}})^{-1}
(1) =\cY_{\fm,0}^{\trop} \]
be the fiber of $\Sigma(f)$ over $b$.
For every bounded edge $E$ of 
$\Gamma$, let
$E_{\Sigma(f)_b}$ be the edge of $\Sigma(f)_b$
marked by 
$E$. Then, we have 
\[h(E_{\Sigma(f)_b}) \subset h(E)=e_{j(E),0} \,.\]
\end{lem}

\begin{proof}
This follows from the fact that $C_j$ is the unique curve in $\tilde{\PP}_j$ of class $[C_j]$.  
\end{proof}

Given a stable log map 
$f \colon C \rightarrow \cY_{\fm,0}$
marked by 
$h$, we have nodes of $C$
in correspondence with the bounded edges of 
$\Gamma$. Cutting 
$C$ along these nodes, we obtain a 
morphism 
\[ \mathrm{cut}
\colon 
\overline{M}_{g,p}^h
(Y_\fm/\partial Y_\fm)
\rightarrow \cM \,. \]
By Lemma \ref{lem_edges_cy}, each cut is locally identical to the corresponding cut in a degeneration along a smooth divisor and so we can refer to Section 6 of \cite{bousseau2017tropical} or Section 5 of \cite{kim2018degeneration} for a 
precise definition of the $\mathrm{cut}$ morphism dealing with log structures. 

We say that a stable log map 
$f \colon C\rightarrow \overline{Y}_\fm$ is torically transverse
if its image does not contain any of the torus fixed points of $\overline{Y}_\fm$, ie if its image does not pass through the ``corners" of the toric boundary divisor
$\partial \overline{Y}_\fm$, ie if its tropicalization has no vertex mapping in the interior of one of the two-dimensional cones of the fan of $\overline{Y}_\fm$.

For every vertex $V$ of $\Gamma$
such that $h(V)=v_0$, let
$M^0_V$
be the open substack of $M_V$ consisting of torically transverse stable log maps.
We define
\[ \cM^0 \coloneqq 
\left(\prod_{\{V \colon h(V)=v_0\}} 
M^0_V \times 
\prod_{\{V \colon h(V) \neq v_0\}} 
M_V \right)
\times_{\prod_E D_E^2}
\left( \prod_E D_E \right) \,,\]
\[ \overline{M}_{g,p}^{h, 0}(Y_\fm/\partial Y_\fm) \coloneqq \mathrm{cut}^{-1}(\cM^0) \,,\]
and we denote by
\[\mathrm{cut}^0 \colon \overline{M}_{g,p}^{h,0}(Y_\fm/\partial Y_\fm)
\rightarrow \cM^0 \]
the corresponding restriction of the $\mathrm{cut}$
morphism.

\begin{lem} \label{lem_etale_cy}
The morphism 
\[\mathrm{cut}^0 \colon \overline{M}_{g,p}^{h,0}(Y_\fm/\partial Y_\fm)
\rightarrow \cM^0 \]
is étale of degree
\[ \frac{
\prod_E w(E)}
{\lcm \{w(E)\}_E}
\,,\]
where the index $E$ runs over bounded edges of $\Gamma$ and $w(E)$ is the weight of $E$.
\end{lem}

\begin{proof}
Because of the restriction to the torically transverse locus, the gluing question is locally isomorphic to the corresponding gluing question
in a degeneration along a smooth divisor, and so the result follows from
formula (6.13) and \mbox{Lemma 9.2} of 
\cite{kim2018degeneration}.
In the corresponding argument in Section 6
of \cite{bousseau2017tropical}, the denominator of the formula did not appear because the relevant tropical curves 
had edges all of integral length. 
\end{proof}

Restricted to the torically transverse locus,
the comparison of obstruction theories on 
$\overline{M}_{g,p}^h
(Y_\fm/\partial Y_\fm)$ and $\cM$ reduces to the same
question studied in \mbox{Section 9} of
\cite{kim2018degeneration}
for a degeneration along a smooth divisor.
In particular, combining Lemma 
\ref{lem_etale_cy} with
formula 9.14 of \cite{kim2018degeneration}, we obtain that the cycle classes 
\[ (\mathrm{cut})_* 
([\overline{M}_{g,p}^h
(Y_\fm/\partial Y_\fm)]^\virt)\]
and 
\[\frac{
\prod_E w(E)}
{\lcm \{w(E)\}_E}
[\cM]^\virt\] have the same restriction to the open substack $\cM^0$ of $\cM$.
By
\cite[Proposition 1.8]{MR1644323},
it follows that their difference is rationally equivalent to a cycle supported on the closed
substack
\[ Z \coloneqq \cM - \cM^0 \,.\]
At a point of $Z$, the corresponding stable log map 
$f \colon C \rightarrow \overline{Y}_\fm$ to $\overline{Y}_\fm$ is not torically transverse. Using Lemma
\ref{lem_edges_cy}, we can apply \mbox{Proposition 11}
of \cite{bousseau2017tropical} to get that $C$ contains a non-trivial cycle of components. As $\lambda_g=0$ for a family of curves containing a non-trivial cycle of components (see eg Lemma 8
of \cite{bousseau2017tropical}), we deduce as in Section 6
of \cite{bousseau2017tropical} that 
\[N_{g,p}^h= \frac{\prod_E w(E)}
{\lcm \{w(E)\}_E} \int_{[\cM]^{\virt}} (-1)^g \lambda_g\,,\]
and using the gluing properties of lambda classes (see eg Lemma 9 
of \cite{bousseau2017tropical})
that 
\[ \int_{[\cM]^{\virt}} (-1)^g \lambda_g = \prod_V N_V \,.\]
For this last step, in order to compute $[\cM]^{\virt}$, we had to insert the class 
$1 \times \pt_E + \pt_E \times 1$ of the diagonal $D_E \hookrightarrow D_E^2$. 
Each bounded edge $E$ in $\Gamma$ connects a vertex $V_0$ with $h(V_0)=v_0$ and a vertex $V_+$ with $h(V_+) \neq v_0$. A stable log map in $M_{V_+}$ is a cover of the curve $C_{j(V_+)}$, which intersects the divisor $D_E$ in a specific point. Therefore, the only term in the class of the diagonal leading to a possibly non-vanishing contribution is the one with 
the insertion of $\pt_E$ on $M_{V_0}$ and the insertion of $1$ on $M_{V_+}$, which is exactly the way we defined the invariants $N_V$. This concludes the proof of 
Proposition \ref{prop_gluing_corrected}.

We remark that the most general form of 
the gluing formula in log Gromov-Witten theory, work in progress of Abramovich, Chen, Gross and Siebert, requires the use of punctured Gromov-Witten invariants;
see \cite{abramovich2017punctured}.
We do not see punctured invariants in our gluing formula because the only contributing rigid tropical curves are contained
in the 1-dimensional part of the polyhedral decomposition of 
$\cY_{\fm,0}^\trop$.

\subsection{Classification of contributing rigid tropical curves}
\label{section_contributing_tropical_curves}

\begin{lem} \label{lem_multicover}
Let $h \colon \Gamma \rightarrow \cY_{\fm,0}^{\trop}$ be a rigid genus $g$ decorated 
parametrized tropical curve of type $p$ with $\Gamma$ of genus $0$ and $N_{g,p}^h \neq 0$.
Let $V$ be a vertex of $\Gamma$ such that $h(V) \neq v_0$, so that $\beta(V)=\ell(V)[C_{j(V)}]$. Let $n_V$ be the number of bounded edges of $\Gamma$
incident to $V$. If $n_V>1$, then we have $N_V=0$. If $n_V=1$, then $N_V$ is the coefficient of $\hbar^{2g(V)-1}$ in 
\[ \frac{(-1)^{\ell(V)-1}}{\ell(V)} \frac{1}{2 \sin \left( \frac{\ell(V) \hbar}{2} \right)}\,.\]
\end{lem}

\begin{proof}
As the curve $C_{j(V)} \simeq \PP^1$ is rigid in 
$\tilde{\PP}_{j(V)}$, with normal bundle $\cO_{\PP^1}(-1)$, every stable log map
defining a point of $M_V$ factors through $C_{j(V)}$. Therefore, the moduli space $M_V$
coincides with the moduli space $M_V(\PP^1/\infty)$
of genus $g(V)$ degree $\ell(V)$ stable log maps to $\PP^1$, relative to $\infty \in \PP^1$
with $n_V$ contact points above $\infty$ and contact orders given by the weights of the $n_V$ bounded edges incident to $V$. However, the surface obstruction theory on $M_V$ defining the class $[M_V]^{\virt}$ differs from the curve obstruction theory 
on $M_V(\PP^1/\infty)$ defining the class $[M_V(\PP^1/\infty)]^{\virt}$.
Denoting by $\pi \colon \cC \rightarrow M_V(\PP^1/\infty)$ the universal source log curve and by $f \colon \cC \rightarrow 
\PP^1$ the universal log map, the two obstruction theories differ by 
\[R^1 \pi_* f^{*} N_{C_{j(V)}|\tilde{\PP}_{j(V)}}=R^1 \pi_* f^{*} \cO_{\PP^1}(-1)\,,\]
and so
\[N_V
=\int_{[M_V(\PP^1/\infty)]^\virt}
e \left(R^1 \pi_* f^* (\cO_{\PP^1} \oplus
\cO_{\PP^1}(-1))\right) \,, \]
where $e(-)$ is the Euler class. The virtual dimension of $[M_V(\PP^1/\infty)]^\virt$
is \[2g(V)-2+\ell(V)+n_V\,,\] whereas the complex degree of the integrand is 
\[ 2g(V)-1+\ell(V)\,.\] Therefore, we have $N_V=0$ for $n_V \neq 1$. To compute 
$N_V$ for $n_V=1$, we remark that log Gromov-Witten invariants of $(\PP^1,\infty)$
coincide with relative Gromov-Witten invariants of $(\PP^1,\infty)$ by the general log/relative comparison theorem of \cite{abramovich2012comparison} for a smooth divisor.
The corresponding relative Gromov-Witten invariants of $(\PP^1,\infty)$ have been computed by Bryan and Pandharipande in \cite{MR2115262}; see the proof of Theorem 5.1 in \cite{MR2115262}.
\end{proof}

Let $h \colon \Gamma \rightarrow \cY_{\fm,0}^{\trop}$ be a rigid genus $g$ decorated 
parametrized tropical curve of type $p$ with $\Gamma$ of genus $0$ and $N_{g,p}^h \neq 0$. 
By Proposition \ref{prop_rigid_corrected}, every vertex $V$ of $\Gamma$ is mapped by $h$ 
to some $v_j$ for some $0 \leqslant j \leqslant n$, and every bounded edge $E$ of $\Gamma$ connects a vertex 
$V_0$ with $h(V_0)=v_0$ and a vertex $V_+$ with $h(V_+)=v_j$ for some $1 \leqslant j \leqslant n$.
Furthermore, the combination of Proposition \ref{prop_gluing_corrected} and 
Lemma \ref{lem_multicover} shows that every vertex $V$ of $\Gamma$ with $h(V) \neq v_0$
is incident to a single bounded edge. As $\Gamma$ is connected, this implies that there is a unique vertex $V_0$ in $\Gamma$ with $h(V_0)=v_0$.
In other words, $h \colon \Gamma \rightarrow \cY_{\fm,0}^{\trop}$
is one of the rigid genus $g$ decorated parametrized tropical curves $h_{k,\vec{g}} \colon \Gamma_{k,\vec{g}}
\rightarrow \cY_{\fm,0}^\trop$ defined as follows.

Recall that we defined
in
\mbox{Section \ref{section_degeneration_statement}}
 what a partition of $p$ is and that 
we associated to such partition $k$ of $p$ a 
positive integer $s(k)$ and a 
$s(k)$-tuple 
$(w_1(k),\dots,w_{s(k)}(k))$ of non-zero
vectors in $M=\Z^2$.
In particular, each $w_r(k)$ 
can be naturally written 
$w_r(k)=\ell m_j$ for some $\ell \geqslant 0$ and some $1 \leqslant j \leqslant n$.
For every partition $k$ of $p$
and for every $\vec{g}=
(g_0, g_1, \dots, g_{s(k)})$, 
an $(s(k)+1)$-tuple of nonnegative integers such that 
$|\vec{g}| \coloneqq g_0+\sum_{r=1}^{s(k)}
g_r = g$,
we define a rigid genus $g$ decorated 
parametrized tropical curve 
$h_{k,\vec{g}} \colon \Gamma_{k,\vec{g}}
\rightarrow \cY_{\fm,0}^\trop$.

Let $\Gamma_{k,\vec{g}}$ 
be the genus $0$ graph consisting of 
vertices $V_0, V_1, \dots, V_{s(k)}$,
for every $1 \leqslant r \leqslant s(k)$
bounded edges 
$E_r$
connecting 
$V_0$ to $V_r$, and an unbounded edge $E_p$ attached to $V_0$.
We define a structure of tropical curve on 
$\Gamma_{k,\vec{g}}$ by assigning:
\begin{itemize}
\item Genera to the vertices. We assign $g_0$ to $V_0$, and 
$g_r$ to $V_r$, for all 
$1 \leqslant r \leqslant s(k)$.
\item Lengths to the bounded edges. We assign the length 
\[ \ell(E_r) \coloneqq \frac{1}{|w_r(k)|}=\frac{1}{\ell} \] 
to the bounded edge $E_r$, 
for all $1 \leqslant r \leqslant s(k)$.
\end{itemize}

Finally, we define a decorated parametrized tropical curve 
\[ h_{k,\vec{g}} \colon
\Gamma_{k,\vec{g}}
\rightarrow \cY_{\fm,0}^\trop\] by the following data:

\begin{itemize}
\item We define $h_{k,\vec{g}}(V_0)\coloneqq v_0$, and, writing $w_r(k)=\ell m_j$, $h_{k,\vec{g}}(V_r) \coloneqq v_j$, 
for all 
$1 \leqslant r \leqslant s(k)$.
\item Edge markings of bounded edges. We define $v_{V_0, E_r}
\coloneqq w_r$ for all 
$1 \leqslant r \leqslant s(k)$.
In particular, the bounded edge $E_r$ has weight $|w_r(k)|=\ell$.
This is a valid choice because 
\[ h(V_r)-h(V_0) = m_j = \frac{1}{\ell}\ell m_j
=\ell(E_r) v_{V_0,E_r}\,.\] 
This uniquely specifies an affine linear map $h_{k,\vec{g}}|_{E_r}$. 
\item Edge marking of the unbounded edge.
We define $v_{V_0,E_p}
\coloneqq \ell_p m_p$. In particular, the unbounded edge $E_p$ has weight 
$\ell_p$.  This uniquely specifies an affine linear map 
$h_{k,\vec{g}}|_{E_p}$.
\item Decoration of vertices by curve classes. We decorate $V_0$
with the curve class 
$\beta_{w(k)}
\in H_2(\overline{Y}_\fm,\Z)$.
Writing $w_r(k)=\ell m_j$, we decorate the vertex $V_r$ with the curve class 
$\ell [C_j] \in H_2(\tilde{\PP}_j,\Z)$.
\end{itemize}

\begin{center}
\setlength{\unitlength}{1cm}
\begin{center}
Figure: picture of $\Gamma_{k,\vec{g}}$.\\
\end{center}

\begin{picture}(6,6)
\thicklines
\put(3,3){\circle*{0.1}}
\put(3,3){\line(1,0){3}}
\put(3,3){\line(1,2){1.5}}
\put(3,3){\line(-2,1){2}}
\put(3,3){\line(1,-2){1.5}}
\put(6,3){\circle*{0.1}}
\put(1,4){\circle*{0.1}}
\put(4.5,0){\circle*{0.1}}
\put(4,0){$V_2$}
\put(3.2,1.5){$E_2$}
\put(2.7,3.2){$V_0$}
\put(4.5,3.2){$E_1$}
\put(5.75,3.2){$V_1$}
\put(0.5,4){$V_3$}
\put(1.5,3.3){$E_3$}
\put(3.5,5){$E_p$}
\end{picture}
\end{center}

We summarize our classification of contributing rigid tropical curves by
the following Proposition.

\begin{prop} \label{prop_classification}
Every rigid genus $g$ decorated parametrized tropical curve \[ h \colon \Gamma \rightarrow \cY_{\fm,0}^\trop\] of type $p$ with $\Gamma$ of genus zero
and $N_{g,p}^h \neq 0$,
is of the form $h_{k,\vec{g}}$ for some
partition $k$ of $p$ and some
$(s(k)+1)$-tuple
$\vec{g}=(g_0,g_1,\dots
, g_{s(k)})$ 
of nonnegative integers such that 
$|\vec{g}|=g$. 
\end{prop}

In the remainder of Section \ref{section_contributing_tropical_curves}, we compute
for $h=h_{k,\vec{g}}$ 
the numerical factors $n_h$ and $|\Aut(h)|$ entering in the decomposition formula of 
Proposition \ref{prop_decomp_corrected}.
We fix $k$ a partition of $p$ and $\vec{g}$ such that $|\vec{g}|=g$ and we consider the decorated parametrized tropical
curve $h \colon \Gamma_{k,\vec{g}}
\rightarrow \cY_{\fm,0}^\trop$. 

\begin{lem} \label{lem_n}
We have 
\[ n_{h_{k,\vec{g}}} = \lcm \{|w_r(k)|, 1 \leqslant r \leqslant s(k) \} \,.\]
\end{lem}

\begin{proof}
Recall that $n_{h_{k,\vec{g}}}$ is the 
smallest positive integer such that after
scaling by $n_{k_{k,\vec{g}}}$, $h_{k,\vec{g}}$ gets integral vertices and integral lengths. By definition of $h_{k,\vec{g}}$, vertices of $h_{k,\vec{g}}$
are already mapped to integral points of 
$\cY_{\fm,0}^\trop$. On the other hand, bounded edges $E_r$ of $\Gamma_{k,\vec{g}}$ have fractional lengths $1/|w_r(k)|$. It follows that $n_{k,\vec{g}}$ is the least common 
multiple of the positive integers
$|w_r(k)|$, $1 \leqslant r \leqslant s(k)$.
\end{proof}

For 
$1 \leqslant j \leqslant n$, 
$\ell \geqslant 1$
and $a \geqslant 0$, denote by
$k_{\ell, j,a}$ the number of vertices
of $\Gamma_{k,\vec{g}}$ having genus $a$ among the $k_{\ell, j}$ ones having curve class decoration 
$\ell [C_j]$. 
Note that we have 
\[k_{\ell, j} = \sum_{a \geqslant 0} 
k_{\ell, j,a} \,,\]
and 
\[ \sum_{r=1}^{s(k)}
g_r = \sum_{j=1}^n
\sum_{\ell \geqslant 1}
\sum_{a \geqslant 0} a k_{\ell, j,a}\,.\]

\begin{lem} \label{lem_aut}

The order of the automorphism group of 
the decorated parametrized tropical curve 
$h_{k, \vec{g}} \colon \Gamma_{k,\vec{g}} \rightarrow 
\cY_{\fm,0}^\trop$ is given by 
\[ |\Aut (h_{k,
\vec{g}})|
= \prod_{j=1}^n \prod_{\ell \geqslant 1} \prod_{a \geqslant 0}
k_{\ell, j,a}! \,.\]
\end{lem}

\begin{proof}
For every $1 \leqslant j \leqslant n$, 
$\ell \geqslant 1$ and 
$a \geqslant 0$, there are $k_{\ell, j,a}$ of the vertices $V_r$ having the same curve class decoration $\ell [C_j]$, the same genus $a$, and the attached edges have the same weight $\ell m_j$, so permutations of these $k_{\ell, j,a}$ vertices define automorphisms of the decorated tropical curve $h_{k, \vec{g}}$. Any other permutation of the vertices of $\Sigma_k$ permutes vertices having different curve class decorations and/or different genus, and so cannot be an automorphism of the decorated tropical curve.
\end{proof}

\subsection{End of the proof of the degeneration formula}
\label{section_end_proof_deg}

In this Section, we end the proof of Proposition \ref{prop_degeneration}.
By Proposition \ref{prop_decomp_corrected}, 
we have 
\[N_{g,p}^{Y_\fm} = 
\sum_h
 \frac{n_h}{|\Aut (h)|}
N_{g,p}^{h} \,,\]
where the sum is over rigid genus $g$
decorated parametrized tropical curves 
$h \colon \Gamma \rightarrow \cY_{\fm,0}^\trop$ of type $p$ with $\Gamma$ of genus $0$.
By Proposition \ref{prop_classification}, such tropical curves $h$ with 
$N_{g,p}^h$ are necessarily of the form 
$h_{k,\vec{g}}$ for some $k$
partition of $p$ and $\vec{g}=(g_0,g_1,\dots
, g_{s(k)})$ some
$(s(k)+1)$-tuple of nonnegative integers such that 
$|\vec{g}|=g$. For $h=h_{k,\vec{g}}$, the factors $n_h$
and $|\Aut(h)|$ are computed by
 Lemmas 
\ref{lem_n} and \ref{lem_aut}. On the other hand, 
$N_{g,p}^{h_{k,\vec{g}}}$ can be computed using the gluing formula
of Proposition \ref{prop_gluing_corrected} and Lemma \ref{lem_multicover}.
After simplification, we obtain the formula of Proposition \ref{prop_degeneration}.

\section{Scattering and tropical curves} \label{section_tropical}

In this Section, we review the connection established in \cite{MR3383167}
between quantum scattering diagrams and refined tropical curve counting.

\subsection{Refined tropical curve counting}
\label{section_tropical_count}
In this Section, we review the
definition of the refined tropical curve counts used in 
\cite{MR3383167}. The relevant tropical curves are identical 
to those considered in
\cite{MR2667135}. The only difference is that they are counted with the Block-Göttsche 
refined multiplicity
\cite{MR3453390}, $q$-deformation of the usual 
Mikhalkin multiplicity
\cite{MR2137980}.

We first recall the definition of a parametrized tropical curve to $\R^2$
by simply repeating the presentation we gave in 
\cite{bousseau2017tropical}.
For us, a graph $\Gamma$ has a finite set 
$V(\Gamma)$ of vertices, a finite set $E_f(\Gamma)$ 
of bounded edges connecting pairs of vertices and a finite set 
$E_\infty (\Gamma)$ of legs attached to vertices that we view as unbounded edges. 
By edge, we refer to a bounded or unbounded edge.
We will always consider connected graphs.
 
A parametrized tropical curve $h \colon \Gamma \rightarrow \R^2$ 
is the following data:
\begin{itemize}
\item A nonnegative integer $g(V)$ for each vertex $V$, called the genus of $V$.
\item A labeling of the elements  of the set $E_\infty(\Gamma)$.
\item A vector $v_{V,E} \in \Z^2$ for every vertex $V$ 
and $E$ an edge incident to $V$. If $v_{V,E}$
is not zero, the divisibility $|v_{V,E}|$ of $v_{V,E}$ in $\Z^2$ 
is called the weight of $E$ and is denoted
by $w(E)$.
We require that $v_{V,E} \neq 0$ if $E$ is unbounded and 
that for every vertex $V$, the following balancing condition is 
satisfied:
\[\sum_E v_{V,E} =0 \,,\]
where the sum is over the edges $E$ incident to 
$V$. If $E$ is an unbounded edge, we write 
$v_E$ for $v_{V,E}$, where $V$ is the unique vertex to which $E$ is attached.
\item A nonnegative real number $\ell(E)$ for every bounded
edge of $E$, called the length of $E$.
\item A proper map $h \colon \Gamma \rightarrow \R^2$
such that 
\begin{itemize}
\item If $E$ is a bounded edge connecting the vertices $V_1$ and $V_2$, 
then $h$ maps $E$ affine linearly on the line segment connecting 
$h(V_1)$ and $h(V_2)$, and 
\[h(V_2)-h(V_1) = \ell(E)v_{V_1,E}\,.\]
\item If $E$ is an unbounded edge of vertex $V$, then 
$h$ maps $E$ affine linearly to the ray 
$h(V)+\R_{\geqslant 0} v_{V,E}$.
\end{itemize}
\end{itemize}
The genus $g_h$ of a parametrized tropical 
curve $h \colon \Gamma \rightarrow \R^2$ 
is defined by
\[g_h \coloneqq g_\Gamma + \sum_{V \in V(\Gamma)} g(V) \,,\]
where $g_\Gamma$ is the genus of the graph 
$\Gamma$.

Let $w=(w_1,\dots,w_s)$ be a $s$-tuple
of non-zero vectors in $M$. We fix 
$x=(x_1, \dots, x_s)$
and element of $(\R^2)^s$. 
We say that a parametrized tropical curve 
$h \colon \Gamma \rightarrow \R^2$ is of type $(w,x)$ if $\Gamma$ has exactly 
$s+1$ unbounded edges, labeled 
$E_0, E_1, \dots, E_s$, such that 
\begin{itemize}
\item $v_{E_0}=\sum_{r=1}^s w_r$,
\item $v_{E_r}=-w_r$ for every
$1 \leqslant r \leqslant s$,
\item $E_r$ asymptotically coincides with 
the half-line
$-\R_{\geqslant 0}w_r + x_r$ for every 
$1 \leqslant r \leqslant s$.
\end{itemize}

Let $T_{w,x}$ be the set of genus $0$ parametrized tropical curves 
$h \colon \Gamma \rightarrow \R^2$ of type 
$(w,x)$ without contracted edges.  
If $x \in (\R^2)^s$ is general enough
(in some appropriate open dense subset), then it follows from \cite{MR2137980}
or \cite{MR2259922} that $T_{w,x}$ is a finite set, and that if 
$(h \colon \Gamma \rightarrow \R^2) \in T_{w,x}$, then 
$\Gamma$ is trivalent and $h$ is an immersion (distinct vertices have distinct images and two distinct edges have at most one point in common in their images).

For $h \colon \Gamma \rightarrow \R^2$ a 
parametrized tropical curve in $\R^2$ and $V$ a trivalent vertex with incident edges
$E_1$, $E_2$ and $E_3$, the multiplicity of 
$V$ is the integer defined by 
\[ m(V) \coloneqq |\det (v_{V,E_1}, v_{V,E_2})| \,.\]
Thanks to the balancing condition 
\[v_{V,E_1}+v_{V,E_2}+v_{V,E_3}=0 \,,\]
this definition is symmetric in 
$E_1, E_2, E_3$.
The Block-Göttsche
\cite{MR3453390} multiplicity of $V$
is a Laurent polynomial in a formal variable $q^{\frac{1}{2}}$:
\[ [m_V]_q \coloneqq
\frac{q^{\frac{m(V)}{2}}-q^{-\frac{m(V)}{2}}}{q^{\frac{1}{2}}-q^{-\frac{1}{2}}} 
=q^{-\frac{m(V)-1}{2}}(1+q+\dots 
+ q^{\frac{m(V)-1}{2}}) 
\in \NN[q^{\pm \frac{1}{2}}] \,. \]

For $h \colon \Gamma \rightarrow \R^2$ a parametrized tropical curve with $\Gamma$ trivalent, the refined multiplicity of $h$ is defined by 
\[ m_h (q^{\frac{1}{2}}) \coloneqq \prod_{V
\in V(\Gamma)} [m(V)]_q \,,\]
where the product is over the
vertices of $\Gamma$.

If $x \in (\R^2)^s$ is in general position, we count the elements of $T_{w,x}$ with refined multiplicities and we obtain a refined count of tropical curves:
\[ N^{\trop}_{w,x}(q^{\frac{1}{2}}) \coloneqq
\sum_{h \colon \Gamma \rightarrow \R^2}
m_h(q^{\frac{1}{2}}) \in \NN[q^{\pm \frac{1}{2}}] \,.\]
According to Itenberg-Mikhalkin 
\cite{MR3142257}, $N_{w,x}^{\trop}(q^{\frac{1}{2}})$ does not depend on $x$ if $x$ is general.
This also follows from the correspondence theorem of 
\cite{bousseau2017tropical}. 
Therefore, we simply denote by $N_{w}^\trop (q^{\frac{1}{2}})$ the corresponding invariant.

\subsection{Elementary quantum scattering}
Let $m_1$ and $m_2$ be two
non-zero vectors in $M=\Z^2$. Let 
$\hat{\fD}$ be the quantum scattering diagram over an Artinian ring $R$ consisting of two 
incoming rays 
$-\R_{\geqslant 0} m_1$ and $-\R_{\geqslant 0} m_2$
equipped with the 
Hamiltonians
\[ \hat{H}_1= \frac{f_1}{q^{\frac{1}{2}} - q^{-\frac{1}{2}}}\hat{z}^{m_1} \,,\]
and 
\[ \hat{H}_2= \frac{f_2}{q^{\frac{1}{2}} - q^{-\frac{1}{2}}} \hat{z}^{m_2}\,,\]
where $f_1, f_2 \in R$
satisfy 
$f_1^2=f_2^2=0$.
Let $S(\hat{\fD})$ be the resulting 
consistent quantum scattering diagram
given by Proposition \ref{prop_consistency}. 
The following result is Lemma 4.3 of 
\cite{MR3383167}.

\begin{lem} \label{lem_elementary_scattering}
The consistent quantum scattering diagram $S(\hat{\fD})$ is obtained from
$\hat{\fD}$ by adding three outgoing rays:
\begin{itemize}
\item $(\R_{\geqslant 0} m_1, \hat{H}_1)$
\item $(\R_{\geqslant 0} m_2, \hat{H}_2)$
\item $(\R_{\geqslant 0} (m_1+m_2), \hat{H}_{12})$, where 
\[ \hat{H}_{12} \coloneqq [\langle m_1, m_2 \rangle]_q \frac{f_1 f_2}{q^{\frac{1}{2}} - q^{-\frac{1}{2}}}
\hat{z}^{m_1+m_2} \,,\]
and 
\[ [\langle m_1,m_2 \rangle]_q
\coloneqq \frac{q^{\frac{\langle m_1,m_2 \rangle}{2}}-q^{-\frac{\langle m_1, m_2 \rangle}{2}}}{q^{\frac{1}{2}}-q^{-\frac{1}{2}}}\,.\]
\end{itemize}
\end{lem}

\begin{proof}
Using 
\[ [\hat{z}^{m_1},\hat{z}^{m_2}]=
\left(q^{\frac{\langle m_1,m_2 \rangle}{2}}-q^{-\frac{\langle m_1, m_2 \rangle}{2}} \right) \hat{z}^{m_1+m_2} \,,\]
we compute that 
\[[\hat{H}_1, \hat{H}_2]
=[\langle m_1,m_2 \rangle]_q
\frac{f_1 f_2}{
q^{\frac{1}{2}}-q^{-\frac{1}{2}}} 
\hat{z}^{m_1+m_2} \,.\]
As $f_1^2=f_2^2=0$, it follows 
that $\hat{H}_1$ and $\hat{H}_2$
commute with $[\hat{H}_1, \hat{H}_2]$.
Using an easy case of the Baker-Campbell-Hausdorff formula, according to which 
$e^a e^b=e^{a+b+\frac{1}{2}[a,b]}$
if $a$ and $b$ commute with 
$[a,b]$, we obtain 
\[e^{\hat{H}_2}e^{-\hat{H_1}}e^{-\hat{H}_2}e^{\hat{H}_1}
=e^{[\hat{H}_1, \hat{H}_2]} \,,\]
and so 
\[ 
\hat{\Phi}_{\hat{H}_2}^{-1} \hat{\Phi}_{\hat{H}_1} \hat{\Phi}_{\hat{H}_2} \hat{\Phi}_{\hat{H}_1}^{-1} 
=\hat{\Phi}_{[\hat{H}_1, \hat{H}_2]}\,,\]
hence the result 
\end{proof}
\subsection{Quantum scattering from refined tropical curve counting}

In this section, we review the result of Filippini and Stoppa \cite{MR3383167}
expressing the Hamiltonians attached to the rays of the consistent quantum scattering diagram $S(\hat{\fD}_m)$,
defined in Section
\ref{section_statement},
in tropical terms.
We use the notations introduced at the beginning of \mbox{Section
\ref{section_degeneration_statement}}.

\begin{prop} \label{prop_scattering_tropical}
For every $\fm = (m_1,\dots,m_n)$
an $n$-tuple of primitive non-zero vectors in $M$
and for every
$m \in M-\{0\}$,
the Hamiltonian $\hat{H}_m$
attached to the outgoing ray $\R_{\geqslant 0}m$
in the consistent quantum scattering diagram 
$S(\hat{\fD}_{\fm})$ is given 
by 
\[\hat{H}_m =
 \sum_{p \in P_m}
\sum_{k \vdash p} N_{w(k)}^{\mathrm{trop}}
(q^{\frac{1}{2}})
\left( \prod_{j=1}^n
\prod_{\ell \geqslant 1} 
\frac{1}{k_{\ell, j}!}
\left( \frac{(-1)^{\ell-1}}{\ell}
\frac{q^{\frac{1}{2}}-q^{-\frac{1}{2}}}{q^{\frac{\ell}{2}}-q^{-\frac{\ell}{2}}} \right)^{k_{\ell, j}}
\right)
\left( \prod_{j=1}^n t_j^{p_j}
\right)
\frac{\hat{z}^{\ell_p m}}{
q^{\frac{1}{2}}-q^{-\frac{1}{2}}}\,,\]
where $q=e^{i\hbar}$, and the innermost sum is over all
partitions $k$ of $p$.
\end{prop}

\begin{proof}
This follows from the main result, 
Corollary 4.9, of 
\cite{MR3383167}, which is a $q$-deformed version of the proof of Theorem 2.8 of 
\cite{MR2667135}.
A higher dimensional generalization of this argument has been given by Mandel in 
\cite{mandel2015refined}. For completeness and because we organize the combinatorics in a slightly different way, we provide a proof.

By definition, 
$S(\hat{\fD}_{\fm})$ is the consistent quantum
scattering diagram obtained from the 
quantum scattering diagram 
$\hat{\fD}_\fm$ consisting of incoming rays 
$(\fd_j, \hat{H}_{\fd_j})$, 
$1 \leqslant j \leqslant n$, where 
\[ \fd_j = - \R_{\geqslant 0} m_j \,,\]
and 
\[ \hat{H}_{\fd_j} = 
\sum_{\ell \geqslant 1}\frac{1}{\ell} \frac{(-1)^{\ell-1}}{q^{\frac{\ell}{2}}-q^{-\frac{\ell}{2}}} t_j^l \hat{z}^{\ell m_j} \,.\]
Let us work over the ring 
$\C[t_1, \dots,t_n]/(t_1^{N+1},\dots, 
t_n^{N+1})$.
We embed this ring into 
\[ \C[ \{u_{ja} | 1 \leqslant j \leqslant n, 1 \leqslant a \leqslant N\}]/
\langle \{ u_{ja}^2 | 1 \leqslant j \leqslant n, 
1 \leqslant a \leqslant N\}\rangle \]
by
\[ t_j=\sum_{a=1}^N u_{ja}\]
for all 
$1 \leqslant j \leqslant n$.
We then have
\[ t_j^\ell =  \sum_{\substack{A \subset \{1,\dots,N\} \\
|A|=\ell}} \ell ! \prod_{a \in A} u_{ja} \,,\]
and so 
\[\hat{H}_{\fd_j}
=\sum_{\ell=1}^N \sum_{\substack{A \subset \{1,\dots,N\} \\
|A|=\ell}} 
\left( \frac{1}{\ell} \frac{(-1)^{\ell-1}}{q^{\frac{\ell}{2}}-q^{-\frac{\ell}{2}}}
\right) \ell ! 
 \left( 
\prod_{a \in A} u_{ja} 
\right)
\hat{z}^{\ell m_j} \,. \]
This suggests we consider the
quantum scattering diagram 
$\hat{\fD}_{\fm}^{\spl}$ consisting of incoming rays 
$(\fd_{j\ell A}, \hat{H}_{\fd_{j \ell A}})$
for $1 \leqslant \ell
\leqslant N$, 
$A \subset \{1, \dots, N\}$, $|A|=\ell$,
where 
\[ \fd_{j\ell A} = -\R_{\geqslant 0} m_j + c_{j \ell A} \,,\]
for $c_{j \ell A} \in \R^2$ in general position, and 
\[ \hat{H}_{\fd_{j \ell A}}
=
\left( \frac{1}{\ell} \frac{(-1)^{\ell-1}}{q^{\frac{\ell}{2}}-q^{-\frac{\ell}{2}}}
\right) \ell! 
 \left( 
\prod_{a \in A} u_{ja} 
\right)
\hat{z}^{\ell m_j} \,. \]
If we had taken all $c_{j \ell A}$ to be $0$, then $\hat{\fD}_\fm^\spl$ would have been equivalent to $\hat{\fD}_\fm$.
 But for 
$c_{j \ell A} \in \R^2$ in general position, $\hat{\fD}_\fm^\spl$ 
is a perturbation of $\hat{\fD}_{\fm}$: each ray $(\fd_j, \hat{H}_{\partial_j})$ of $\fD_{\fm}$ splits into various rays 
$(\fd_{j\ell A}, \hat{H}_{\fd_{j\ell A}})$ of 
$\hat{\fD}_\fm^\spl$.

 The key simplifying fact is that the consistent scattering diagram 
$S(\hat{\fD}_{\fm}^{\spl})$ can be obtained from 
$\hat{\fD}_{\fm}^{\spl}$ by a recursive procedure involving
only elementary scatterings in the sense of Lemma
\ref{lem_elementary_scattering}. When two rays of $\hat{\fD}_{\fm}^{\spl}$ intersect, we are in the situation of Lemma \ref{lem_elementary_scattering} because 
$u_{ja}^2=0$. The local consistency at this intersection is then guaranteed by emitting a third ray according to 
\mbox{Lemma 
\ref{lem_elementary_scattering}}. Further intersections of the old and newly created rays can similarly be treated by application of 
\mbox{Lemma \ref{lem_elementary_scattering}}.
Indeed, the assumption of general position of the $c_{j\ell A}$ guarantees that only double intersections occur.

The asymptotic scattering diagram of $S(\hat{\fD}_{\fm}^{\spl})$ is the scattering diagram obtained by taking all the rays of 
$S(\hat{\fD})$ and placing their origin at $0 \in \R^2$. By uniqueness of the consistent completion, the asymptotic scattering diagram
is precisely $S(\hat{\fD}_{\fm})$.
To get the Hamiltonian $\hat{H}_m$ attached to an outgoing ray $\R_{\geqslant 0}m$ in $S(\hat{\fD}_{\fm})$, it is then enough to collect the various contributions to the corresponding asymptotic ray of $S(\hat{\fD}_{\fm}^{\spl})$ coming from the recursive construction of $S(\hat{\fD}_{\fm}^{\spl})$.

Let us study how the recursive construction of $S(\hat{\fD}_{\fm}^{\spl})$ can produce a ray $\fd$ asymptotic to $\R_{\geqslant 0} m$
and equipped with a function 
$\hat{H}_\fd$ proportional to $\hat{z}^{\ell_\fd m}$, for some 
$\ell_\fd \geqslant 1$. 
Such a ray is obtained by successive applications of Lemma \ref{lem_elementary_scattering} starting from a subset of the initial incoming rays of $\hat{\fD}_{\fm}^{\spl}$.

We focus on one particular sequence of 
such elementary scatterings. Such sequence naturally defines a graph
$\bar{\Gamma}$ in $\R^2$. This graph starts with unbounded edges given by the initial rays taking part in the sequence of scatterings. When two of these rays meet, they scatter and produce a third ray given by Lemma 
\ref{lem_elementary_scattering}. If this third ray does not contribute to further scatterings ultimately contributing to 
$\hat{H}_\fd$, we do not include it in $\bar{\Gamma}$ and we continue $\bar{\Gamma}$ by propagating the two initial rays.
In particular, $\bar{\Gamma}$ contains a 4-valent vertex given by the two initial rays crossing without non-trivial interaction.

If the third ray does contribute to further scatterings ultimately contributing to 
$\hat{H}_\fd$, we include it in $\bar{\Gamma}$ and we do not propagate the two initial rays. In particular, $\bar{\Gamma}$ gets a trivalent vertex given by the two initial rays meeting and producing the third ray.
Iterating this construction, we obtain one trivalent vertex for each elementary scattering ultimately giving a contribution to $\hat{H}_\fd$. At the end of this process, the last elementary scattering produces the ray 
$\fd$ which becomes an unbounded edge of the graph. 

The graph $\bar{\Gamma}$ has two kinds of vertices: trivalent vertices where a non-trivial elementary scattering happens and 4-valent vertices where two rays cross without non-trivial interaction. 
For each 4-valent vertex, we can separate the two rays crossing, and we obtain a trivalent graph $\Gamma$ and a map
$h \colon \Gamma \rightarrow \bar{\Gamma}
\subset \R^2$ which is one to one except over the 4-valent vertices of $\bar{\Gamma}$ where it is two to one.
It follows from the iterative construction that the trivalent graph $\Gamma$ is a tree, ie a graph of genus $0$. 

The function attached to initial ray of  $\hat{\fD}_{\fm}^{\spl}$
is a monomial in $\hat{z}$, whose power  
is proportional to the direction of the ray.
By Lemma
\ref{lem_elementary_scattering}, this property is preserved under elementary scattering. Each edge $E$ of our
$\Gamma$ is thus equipped with a function proportional to 
$\hat{z}^{m_E}$ for some $m_E \in M=\Z^2$ proportional to the direction of $E$.
Furthermore, in an elementary scattering of two edges $E_1$ and $E_2$ equipped with $m_{E_1}$ and $m_{E_2}$, the produced edge $E_3$ is equipped 
with $m_{E_1}+m_{E_2}$ by
\mbox{Lemma 
\ref{lem_elementary_scattering}}.
In other words, the balancing condition is satisfied at each vertex and so we can view
$h \colon \Gamma \rightarrow \R^2$ as a
parametrized tropical curve to 
$\R^2$ in the sense of Section 
\ref{section_tropical_count}.

For every $1 \leqslant j \leqslant n$ and $\ell \geqslant 1$, there is a number $k_{\ell, j}$ of subsets $A$ of $\{1, \dots, n \}$, of size $\ell$, such that $\fd_{j\ell A}$ is one of the initial rays appearing in $\Gamma$. Denote by 
$\cA_{j \ell}^\Gamma$ this set of subsets of 
$\{1, \dots, n\}$. 
Writing $p_j \coloneqq \sum_{\ell \geq 1} \ell k_{\ell, j}$, we have, by the balancing condition,
\[\sum_{j=1}^n p_j = \ell_\fd m \,,\]
and, in particular, $\ell_\fd = \ell_p$.

It follows from an iterative 
application of \mbox{Lemma 
\ref{lem_elementary_scattering}}
that the contribution of $\Gamma$ 
to $\hat{H}_\fd$ is given by 
\[m_\Gamma(q^{\frac{1}{2}})
\left(
\prod_{j=1}^n
\prod_{\ell \geqslant 1} 
\left( \frac{(-1)^{\ell-1}}{\ell}
\frac{q^{\frac{1}{2}}-q^{-\frac{1}{2}}}{q^{\frac{\ell}{2}}-q^{-\frac{\ell}{2}}} \right)^{k_{\ell, j}}(\ell !)^{k_{\ell, j}}
\left( \prod_{A \in 
\cA_{j \ell}^\Gamma} \prod_{a \in A}
u_{ja} \right)
\right)
 \frac{\hat{z}^{\ell_p m}}{
q^{\frac{1}{2}}-q^{-\frac{1}{2}}}\,, \]
where $m_\Gamma (q^{\frac{1}{2}})$ is the refined multiplicity of the tropical curve 
$\Gamma$.

To get the complete expression for 
$\hat{H}_\fd$, we have to sum over the possible
$\Gamma$. 

If we fix $p=(p_1, \dots, p_n)
\in P=\NN^n$, $k$ a partition of $p$ and 
for every
$1 \leqslant j \leqslant n$ and 
$\ell \geqslant 1$, a set 
$\cA_{j\ell}$ of $k_{\ell, j}$ disjoint subsets of 
$\{1, \dots, N\}$ of size $\ell$, we can consider the set 
$T_{j\ell \cA_{j\ell}}$ of genus $0$ tropical curves 
$\Gamma$ having one unbounded edge of asymptotic direction
$\R_{\geqslant 0}m$ and weight 
$\ell_p m$, and for every 
$1 \leqslant j \leqslant n$,
$\ell \geqslant 1$ and $A \in \cA_{j\ell}$, an unbounded edge of weight $\ell m_j$ asymptotically coinciding with 
$\fd_{j\ell A}$.
By \mbox{Section \ref{section_tropical_count}}, this set is finite. 

We already saw how a sequence of elementary scatterings 
contributing to $\hat{H}_\fd$ produces an element $\Gamma \in T_{j \ell \cA_{j\ell}}$. Conversely, any $\Gamma \in T_{j\ell\cA_{j\ell}}$ will define a sequence of elementary scatterings appearing in the construction of 
$S(\hat{\fD}_\fm^\spl)$ and contributing to 
$\hat{H}_\fd$.

It follows that, for every 
$m \in M-\{0\}$, we have
\[ \hat{H}_m
= \]
\[
\sum_{p \in P_m}
\sum_{k \vdash p}
\sum_{\cA_{j\ell}}
\left( \sum_{\Gamma \in 
T_{j\ell\cA_{j\ell}}}
m_\Gamma (q^{\frac{1}{2}}) \right)
\left(
\prod_{j=1}^n
\prod_{\ell \geqslant 1} 
\left( \frac{(-1)^{\ell-1}}{\ell}
\frac{q^{\frac{1}{2}}-q^{-\frac{1}{2}}}{q^{\frac{\ell}{2}}-q^{-\frac{\ell}{2}}} \right)^{k_{\ell, j}}(\ell !)^{k_{\ell, j}}
\left( 
\prod_{A \in \cA_{j\ell}} \prod_{a \in A}
u_{ja} \right)
\right)
 \frac{\hat{z}^{\ell_p m}}{
q^{\frac{1}{2}}-q^{-\frac{1}{2}}}\,, \]
But by Section
\ref{section_tropical_count}, we have 
\[ \sum_{\Gamma \in 
T_{j\ell\cA_{j\ell}}}
m_\Gamma  (q^{\frac{1}{2}})
=N^\trop_{w(k)} (q^{\frac{1}{2}}) \,,\]
which is in particular independent of 
$\cA_{j\ell}$. So we can do the sum over 
$\cA_{j\ell}$. Given an 
$\cA_{j\ell}$, we can form 
\[ B \coloneqq \bigcup_{A \in \cA_{j\ell}} A\,,\]
a subset of 
$\{1, \dots ,N\}$ of size $\sum_{\ell \geqslant 1} \ell k_{\ell, j}=p_j$. Conversely, 
the number of ways to write a set $B$ of $p_j=\sum_{\ell \geqslant 1} 
\ell k_{\ell, j}$ elements as a disjoint union of subsets, $k_{\ell, j}$ of them being of size $\ell$, is equal to 
\[\frac{p_j!}{\prod_{\ell \geqslant 1} k_{\ell, j}! (\ell!)^{k_{\ell, j}}} \,.\]
Replacing the sum over $\cA_{j\ell}$ by a sum over $B$, we obtain

\[ \hat{H}_m
= \]
\[
\sum_{p \in P_m}
\sum_{k \vdash p}
N^\trop_{w(k)}(q^{\frac{1}{2}})
\left(
\prod_{j=1}^n
\prod_{\ell \geqslant 1} 
\frac{1}{k_{\ell, j}!}
\left( \frac{(-1)^{\ell-1}}{\ell}
\frac{q^{\frac{1}{2}}-q^{-\frac{1}{2}}}{q^{\frac{\ell}{2}}-q^{-\frac{\ell}{2}}} \right)^{k_{\ell, j}} \right)
\left(
\prod_{j=1}^n
\sum_{\substack{B \subset \{1, \dots, N\}
\\ |B|=p_j}}p_j! 
\prod_{b \in B} 
u_{jb} \right)
 \frac{\hat{z}^{\ell_p m}}{
q^{\frac{1}{2}}-q^{-\frac{1}{2}}}\,. \]
Finally, using that 
\[ t_j^{p_j} 
= \sum_{\substack{B \subset \{1,\dots, N \} \\ |B|=p_j}}
p_j! \prod_{b \in B} u_{j b} \,,\]
we obtain the desired formula for 
$\hat{H}_m$.

\end{proof}

\begin{cor} \label{corollary_tropical}
We have 
\[\hat{H}_m =
\sum_{p \in P_m}
\sum_{k \vdash p} N_{w(k)}^{\mathrm{trop}}
(q^{\frac{1}{2}}) 
\left( \prod_{j=1}^n
\prod_{\ell \geqslant 1} 
\frac{1}{k_{\ell, j}!}
\left( \frac{(-1)^{\ell-1}}{\ell}
\frac{1}{q^{\frac{\ell}{2}}-q^{-\frac{\ell}{2}}} \right)^{k_{\ell, j}}
\right)
(q^{\frac{1}{2}}-q^{-\frac{1}{2}})^{s(k)-1}
\hat{z}^{\ell_p m}
\,.\]
\end{cor}

\begin{proof}
We simply rearrange the 
 factors $(q^{\frac{1}{2}}-q^{-\frac{1}{2}})$ in Proposition
\ref{prop_scattering_tropical}
and use that 
\[ s(k)=\sum_{j=1}^n \sum_{\ell \geqslant 1} k_{\ell, j} \,.\]
\end{proof}

\section{End of the proof of Theorems \ref{precise_main_thm_ch2} and \ref{precise_main_thm_orbifold}}
\label{section_end_proof}

\subsection{End of the proof of Theorem \ref{precise_main_thm_ch2}}
In this Section, we finish the proof of
Theorem \ref{precise_main_thm_ch2}.
We have to
express the Hamiltonians 
attached to the rays of the consistent quantum scattering diagram 
$S(\hat{\fD}_\fm)$ in terms of 
the log Gromov-Witten invariants 
$N_{g,p}^{Y_\fm}$ of the log Calabi-Yau surface 
$Y_\fm$.
We know already:
\begin{itemize}
\item Corollary \ref{corollary_tropical},
expressing the Hamiltonians attached to the rays 
of $S(\hat{\fD}_\fm)$ in terms of the refined counts $N_w^\trop(q^{\frac{1}{2}})$ of 
tropical curves in $\R^2$.
\item Proposition \ref{prop_degeneration}, 
relating the log Gromov-Witten invariants $N_{g,p}^{Y_\fm}$
of the log Calabi-Yau surface 
$Y_{\fm}$ to the log Gromov-Witten invariants $N_{g,w}^{\overline{Y}_\fm}$ of the 
toric surface $\overline{Y}_\fm$.
\end{itemize}
It remains to connect the refined tropical counts $N_w^{\trop}(q^{\frac{1}{2}})$ to the log Gromov-Witten invariants 
$N_{g,w}^{\overline{Y}_\fm}$ of the toric
surface $\overline{Y}_\fm$. 
This is given by Proposition \ref{prop_correspondence},  
which is a special case of the main
result, Theorem 6, of 
\cite{bousseau2017tropical}.

\begin{prop} \label{prop_correspondence}
For every 
$\fm=(m_1, \dots, m_n)$ $n$-tuple of non-zero primitive vectors in 
$M= \Z^2$, every 
$p=(p_1, \dots, p_n) \in P=\NN^n$, and every $k$ partition of $p$,
we have 
\begin{align*} \sum_{g \geqslant 0} N_{g,w(k)}^{\overline{Y}_\fm} \hbar^{2g-1+s(k)} 
&= N_{w(k)}^{\mathrm{trop}}(q^{\frac{1}{2}})
\left( \prod_{r=1}^{s(k)}
\frac{1}{|w_r|} \right)
\left( 2 \sin \left( \frac{\hbar}{2} \right) \right)^{s(k)-1} \\
& = N_{w(k)}^{\mathrm{trop}}(q^{\frac{1}{2}})
\left( \prod_{j=1}^n \prod_{\ell \geqslant 1}
\frac{1}{\ell^{k_{\ell, j}}} \right)
\left( 2 \sin \left( \frac{\hbar}{2} \right) \right)^{s(k)-1} \,.
\end{align*}
\end{prop}

\begin{proof}
We simply explain the change in notations needed to translate from \mbox{Theorem 6} of \cite{bousseau2017tropical}.

In \cite{bousseau2017tropical}, we
fix $\Delta$, a balanced collection of vectors in $\Z^2$, 
specifying a toric surface $X_\Delta$ and tangency conditions for a curve along the toric divisors. We also fix a subset $\Delta^F$ of $\Delta$, for which the corresponding tangency conditions happen at prescribed 
positions on the toric divisors.
Finally, we fix a nonnegative  integer $n$. Theorem 6 of 
\cite{bousseau2017tropical} is a
correspondence theorem between log Gromov-Witten invariants of $X_\Delta$, counting curves in 
$X_{\Delta}$ satisfying the tangency constraints imposed by $\Delta$ and $\Delta^F$, and passing through $n$ points in general position, and refined 
counts of tropical curves in 
$\R^2$ satisfying the tropical analogue of these constraints.

To get Proposition 
\ref{prop_correspondence}, we take 
\[\Delta = (w_1(k), \dots, w_{s(k)}(k), k_w m_w)\,,\]  
\[\Delta^F = (w_1(k),\dots, w_{s(k)}(k))\,,\] 
and $n=0$. 
We then have $X_\Delta =
\overline{Y}_\fm$ up to some toric blow ups, which do not change the relevant log Gromov-Witten invariants by 
\cite{MR3778185}.
Using the notations of 
\cite{bousseau2017tropical}, we have $|\Delta|=s(k)+1$, $|\Delta^F|=s(k)$ and 
$g_{\Delta,n}^{\Delta^F}=0$.
As the variable $u$ keeping track of the genus in 
\cite{bousseau2017tropical} is denoted $\hbar$ in the present paper, we see that Theorem 6 of 
\cite{bousseau2017tropical} reduces to Proposition \ref{prop_correspondence}.
\end{proof}

By comparison of the explicit formulas of Corollary \ref{corollary_tropical}, 
\mbox{Proposition \ref{prop_correspondence}}
and Proposition \ref{prop_degeneration}, 
and using the relation
\[ s(k)=\sum_{j=1}^n \sum_{\ell \geqslant 1}
k_{\ell, j}\] to collect the powers of $i$,
we find exactly the
formula given in 
\mbox{Theorem \ref{precise_main_thm_ch2}} for the Hamiltonians of the quantum scattering diagram 
$S(\hat{\fD}_\fm)$ in terms of the log Gromov-Witten invariants 
$N_{g,p}^{Y_\fm}$ of the log Calabi-Yau surface 
$Y_{\fm}$. 
This ends the proof of Theorem 
\ref{precise_main_thm_ch2}.

\subsection{End of the proof of Theorem \ref{precise_main_thm_orbifold}}
\label{section_BP_orbifold}

The proof of Theorem \ref{precise_main_thm_orbifold}
follows the one of Theorem \ref{precise_main_thm_ch2}, up to minor notational changes. The only needed
serious modification is an orbifold version of the multicovering formula of Lemma \ref{lem_multicover}. This is provided by Lemma \ref{lem_BP_orbifold} below.

We fix positive integers $r$ and $\ell$. Let 
$\PP^1[r,1]$ be the stacky projective line 
with a single orbifold point of isotropy group
$\Z/r$ at $0$.
Let $\overline{M}_{g,\ell}(\PP^1[r,1]/\infty)$ be the moduli space of genus $g$ orbifold stable maps to 
$\PP^1[r,1]$, relative to $\infty \in \PP^1[r,\infty]$, of degree $r \ell$, with maximal tangency order $r \ell$ along $\infty$. It is a proper Deligne-Mumford stack of virtual dimension $2g-1+\ell$, admitting a virtual
fundamental class 
\[ [\overline{M}_{g,\ell}(\PP^1[r,1]/\infty)]^\virt
\in A_{2g-1+\ell}(\overline{M}_{g,\ell}(\PP^1[r,1]/\infty),\Q) \,.\]
Let $\cO_{\PP^1[r,1]}(-[0]/(\Z/r))$ be the orbifold line bundle on $\PP^1[r,1]$ of degree $-1/r$.
Denoting $\pi \colon \cC \rightarrow 
\overline{M}_{g,\ell}(\PP^1[r,1]/\infty)$ the 
universal source curve and $f \colon \cC \rightarrow
\PP^1[r,1]$ the universal map, we define
\[ N_{g,r}^\ell 
\coloneqq 
\int_{[\overline{M}_{g,\ell}
(\PP^1[r,1]/\infty)]^\virt}
(-1)^g \lambda_g \,
e
\left(
R^1 \pi_{*} f^{*}
\left(\cO_{\PP^1[r,1]}
(-[0]/(\Z/r))
\right) \right) \,,\]
where 
$e(-)$ is the Euler class.

\begin{lem}\label{lem_BP_orbifold}
For all positive integers $r$ and $\ell$, we have 
\[ \sum_{g \geqslant 0} N_{g,r}^\ell 
\hbar^{2g-1} 
= 
\frac{(-1)^{\ell-1}}{\ell}
\frac{1}{2 \sin 
\left( \frac{r \ell \hbar}{2} \right)}\,.\]
\end{lem}

\begin{proof}
This is a
higher genus version of Proposition 5.7 of 
\cite{MR2667135} and an orbifold version of Theorem 5.1 of \cite{MR2115262}. Very similar localization computations of higher genus orbifold Gromov-Witten invariants can be found in \cite{MR2785870}. The 
main thing we need to explain is the replacement
in the orbifold case for the Mumford relation 
$c(\mathbb{E})c(\mathbb{E}^\vee)=1$ playing a
key role in the proof of Theorem 5.1 of \cite{MR2115262}. We will simply have to twist the usual Hodge theoretic argument of
\cite{MR717614} by a local system.

We consider the action of $\C^{*}$
on $\PP^1[r,1]$ with tangent weights $[1/r,-1]$
at the fixed points $[0,\infty]$. We choose the equivariant lifts of  \[\cO_{\PP^1[r,1]}(-[0]/(\Z/r))\] 
and $\cO_{\PP^1[r,1]}$ having fibers over the fixed points $[0,\infty]$ of weights $[-1/r,0]$ and $[0,0]$ respectively.
For such choices, the argument given in the proof of 
Theorem 5.1 of \cite{MR2115262} shows that only one graph $\Gamma$ contributes to the $\C^{*}$-localization formula computing $N_{g,r}^l$. The graph $\Gamma$ consists of a genus $g$ vertex over $0$, a unique edge of degree $r\ell$ and  
a degenerate genus $0$ vertex over $\infty$.

The contribution of $\Gamma$ is computed using the 
virtual localization formula of \cite{MR1666787}. 
We assume that $g>0$. The case $g=0$ is
simpler and treated in \mbox{Proposition 5.7} of 
\cite{MR2667135}.
The corresponding $\C^{*}$-fixed locus is the fiber product 
\[ \overline{M}_{g,1}(B \Z/r) \times_{\overline{I} B \Z/r} B \Z/(rd) \,,\]
where $\overline{M}_{g,1}(B \Z/r)$ is the moduli stack of 1-pointed
(with a trivial stacky structure at the marked point) genus $g$ orbifold stable maps to the classifying stack $B\Z/r$, $\overline{I}B\Z/r$ is the rigidified inertia stack of $B \Z/r$, and the
classifying stack $B \Z/(rd)$ appears as moduli space of 
$\C^{*}$-invariant Galois covers $\PP^1 \rightarrow \PP^1[r,1]$ of degree $r\ell$. This fibered product is a cover of  $\overline{M}_{g,1}(B \Z/r)$ of degree $r/(r\ell)$.

Let $\pi_0 \colon \cC_0 \rightarrow  \overline{M}_{g,1}(B \Z/r)$ be 
the universal source curve over 
$\overline{M}_{g,1}(B \Z/r)$.
The data of an orbifold stable map
$f_0 \colon C_0 \rightarrow B\Z/r$ is equivalent to the data of an (orbifold) $\Z/r$-local system $L$ on $C_0$.
We denote by $t$ the generator of the $\C^{*}$-equivariant cohomology of a point.

The computation of the inverse of the equivariant Euler class of the equivariant virtual bundle is done in Section 2.2 \cite{MR2785870} and gives
\[ e \left(R^1 (\pi_0)_* \left( \cO_{\cC_0} \otimes L \otimes \frac{t}{r} \right) \right) \frac{(r\ell)^\ell}{t^\ell \ell !}
\frac{1}{\frac{t}{r \ell}-\psi} \left( \frac{r}{t}
\right)^{\delta_{L,0}} \frac{t}{r} \,,\]
where $\delta_{L,0}=1$ if $L$ is the trivial 
$\Z/r$-local system and $0$ otherwise.
The vector bundle 
\[R^1 (\pi_0)_* \left( \cO_{\cC_0} \otimes L \otimes \frac{t}{r} \right)\]
over $\overline{M}_{g,1}(B\Z/r)$ comes from the equivariant orbifold line bundle 
$T_{\PP^1[r,1]}(-\infty)|_{[0]/(\Z/r)}$ over 
$B \Z/r$, restriction over $[0]/(\Z/r)$ of the degree
$1/r$ orbifold line bundle 
$T_{\PP^1[r,1]}(-\infty)$ over $\PP^1[r,1]$.

The contribution of the integrand in the definition of
$N_{g,r}^\ell$ is
\[ (-1)^g \lambda_g
e \left( R^1(\pi_0)_* \left(\cO_{\cC_0} \otimes \left(L \otimes \frac{t}{r}\right)^\vee \right) \right) \left( -\frac{t}{r} \right)^{1-\delta_{R,0}}(-1)^{\ell-1} \frac{(\ell-1)!}{(r\ell)^{\ell-1}}t^{\ell-1} \,.\]
The vector bundle $ R^1(\pi_0)_* \left(\cO_{\cC_0} \otimes \left(L \otimes \frac{t}{r}\right)^\vee \right)$
over $\overline{M}_{g,1}(B\Z/r)$ comes from the equivariant orbifold line bundle 
$\cO_{\PP^1[r,1]}(-[0]/(\Z/r))|_{[0]/(\Z/r)}$ over 
$B \Z/r$, restriction over $[0]/(\Z/r)$ of the degree
$-1/r$ orbifold line bundle 
$\cO_{\PP^1[r,1]}(-[0]/(\Z/r))$ over $\PP^1[r,1]$.

By Serre duality, we have 
\[ R^1(\pi_0)_* \left(\cO_{\cC_0} \otimes \left(L \otimes \frac{t}{r}\right)^\vee \right)
= \left( (\pi_0)_* \left( \omega_{\pi_0}
\otimes L \otimes \frac{t}{r} \right) \right)^\vee \,, \]
and so 
\begin{align*}
e \left( R^1(\pi_0)_* 
\left(\cO_{\cC_0} \otimes 
\left(L \otimes \frac{t}{r}\right)^\vee  
\right)
\right)
&=(-1)^{\rk} 
e \left( 
(\pi_0)_{*}
\left( 
\omega_{\pi_0} \otimes L \otimes \frac{t}{r}
\right)
\right) \\
&= (-1)^{\rk} \left(\frac{t}{r}\right)^{\rk}
\sum_{j=0}^{\rk}
\left( \frac{r}{t} \right)^j
c_j \left( 
(\pi_0)_{*}
\left( 
\omega_{\pi_0} \otimes L
\right)
\right)\\
&= (-1)^{\rk} 
\left(
\frac{t}{r}
\right)^{\rk}
c_{\frac{r}{t}}((\pi_0)_{*}
(\omega_{\pi_0} \otimes L))\,, 
\end{align*}
where $\rk$ is the rank of 
$(\pi_0)_{*}
\left( 
\omega_{\pi_0} \otimes L
\right)$, a locally constant
function on 
$\overline{M}_{g,1}(B \Z/r)$, equal to $g$ on the component with $L$ trivial and to $g-1$ on the 
components with $L$ non-trivial, 
and where 
\[c_x(E) \coloneqq \sum_{j\geqslant 0} x^j c_j(E)\]
is the Chern polynomial of a vector bundle $E$.
Similarly, we have 
\begin{align*}
e \left(R^1 (\pi_0)_* \left( \cO_{\cC_0} \otimes L \otimes \frac{t}{r} \right) \right)
&=\left(\frac{t}{r}\right)^{\rk} 
\sum_{j=0}^{\rk} 
\left(\frac{r}{t} \right)^j
c_j
\left(R^1 (\pi_0)_* 
\left( \cO_{\cC_0} \otimes L \right) 
\right) \\
&= \left(\frac{t}{r}\right)^{\rk} 
c_{\frac{r}{t}}\left(R^1 (\pi_0)_* 
\left( \cO_{\cC_0} \otimes L \right) 
\right) \,.
\end{align*}

We twist now the Hodge theoretic argument of 
\cite{MR717614} (see formulas (5.4) and (5.5))
(see also Proposition 3.2 of \cite{MR2357679})
by the local system $L$. The complex
\[ \omega_{\cC_0}^{\bullet}
\colon 0 \rightarrow \cO_{\cC_0}
\xrightarrow{d} \omega_{\pi_0} \rightarrow 0 \,,\]
twisted by $L$, gives rise to an exact sequence
\[ 0 \rightarrow (\pi_0)_* (\omega_{\pi_0} \otimes L)
\rightarrow R^1(\pi_0)_* (\omega_{\cC_0}^{\bullet} 
\otimes L)
\rightarrow R^1(\pi_0)_* (\cO_{\cC_0} \otimes L) 
\rightarrow 0 \,.\]
By Hodge theory, we have the Gauss-Manin connection on the restriction of 
\[R^1(\pi_0)_* (\omega_{\cC_0}^{\bullet} \otimes L)\]
to the open dense subset of $\overline{M}_{g,1}(B\Z/r)$ given by smooth curves, with regular singularities and nilpotent residue along the divisor of nodal curves. This is enough to imply \[ c_x \left(R^1(\pi_0)_* (\omega_{\cC_0}^{\bullet} \otimes L)\right)=1\,,\] and so 
\[c_x \left( (\pi_0)_* (\omega_{\pi_0} 
\otimes L)\right)c_x\left(R^1(\pi_0)_* (\cO_{\cC_0} \otimes L) \right)
=1 \,. \]

Using this relation to simplify the above expressions, we obtain
\[ N_{g,r}^\ell = \frac{r}{r\ell}
\int_{\overline{M}_{g,1}(B\Z/r)}
(-1)^{\ell-1} (-1)^{g+\rk+1-\delta_{L,0}}
 \left(\frac{t}{r}\right)^{2 \rk-2\delta_{L,0}+1}
 \frac{\lambda_g}{\frac{t}{r\ell}-\psi} \,. \]
Using that $\rk=g-1+\delta_{R,0}$, this can be rewritten as 
\[ N_{g,r}^\ell = 
\int_{\overline{M}_{g,1}(B\Z/r)}
\frac{(-1)^{\ell-1}}{\ell} 
 \left(\frac{t}{r}\right)^{2g-1}
 \frac{\lambda_g}{\frac{t}{r\ell}-\psi} \,. \]
As the dimension of $\overline{M}_{g,1}(B\Z/r)$ is $3g-2$, we have to extract the term proportional to 
$\psi^{2g-2}$ and we obtain
\[ N_{g,r}^\ell = \int_{\overline{M}_{g,1}(B\Z/r)}\frac{(-1)^{\ell-1}}{\ell}
\ell^{2g-1}
 \lambda_g \psi^{2g-2}   \,.
\]
The integrand is now the pullback from 
the moduli space $\overline{M}_{g,1}$ of $1$-pointed genus $g$ stable maps. 
The forgetful map 
$\overline{M}_{g,1}(B\Z/r) \rightarrow 
\overline{M}_{g,n}$
has degree
$r^{2g-1}$.
Indeed, there are $r^{2g}$ 
$\Z/r$-local systems on a smooth genus $g$ curve, each with a $\Z/r$ group of automorphisms.
Therefore, we have 
\[ N_{g,r}^\ell =\frac{(-1)^{\ell-1}}{\ell} (r\ell)^{2g-1}
\int_{\overline{M}_{g,1}} \lambda_g \psi^{2g-2} \,,\]
and the result then follows, as in 
the proof of Theorem 5.1 of \cite{MR2115262}, from the Hodge integrals
computations of \cite{MR1728879}.
\end{proof}

\section{Integrality results and conjectures}
\label{section_integrality}
In Section
\ref{section_integ_conj}, we state Conjecture \ref{conj_integrality}, a 
log BPS integrality conjecture.
In Section \ref{section_integ_result}, we state Theorem \ref{thm_integ}, precise version of \mbox{Theorem
\ref{main_thm_integ}} of the Introduction, 
establishing the validity of Conjecture 
\ref{conj_integrality} for 
$(Y_\fm,\partial Y_\fm)$.
The proof of Theorem
\ref{main_thm_integ} takes Sections
\ref{section_quad_ref} and
\ref{section_proof_integ}.
In Section \ref{section_dt}, we describe an explicit connection with refined Donaldson-Thomas theory of quivers. Finally, in Section \ref{section_del_pezzo}, we discuss del Pezzo surfaces with a smooth anticanonical divisor and we formulate 
\mbox{Conjecture \ref{conj_del_pezzo}}, 
precise form of \mbox{Conjecture \ref{conj_main}} of the Introduction.

\subsection{Integrality conjecture}
\label{section_integ_conj}
We formulate a higher genus analogue of
the log BPS integrality conjecture,  
Conjecture 6.2, of \cite{MR2667135}.
We start by formulating a rationality conjecture, Conjecture
\ref{conj_rationality}, before 
stating the integrality conjecture, 
Conjecture 
\ref{conj_integrality}.

Let $Y$ be a smooth projective surface and let 
$\partial Y \subset Y$ be a reduced normal crossing effective divisor.
We endow $Y$ with the divisorial log structure defined by $\partial Y$ and we obtain a smooth 
log scheme. Following Section 6.1 of 
\cite{MR2667135}, we say that
$(Y,\partial Y)$ is log Calabi-Yau with respect to some non-zero class 
$\beta \in H_2(Y,\Z)$
if $\beta \cdot (\partial Y)=\beta \cdot (-K_Y)$.

Two basic examples are:
\begin{itemize}
\item For every $\fm=(m_1,\dots,m_n)$ an $n$-tuple of primitive non-zero vectors 
in $M=\Z^2$, the pair
$(Y_\fm,\partial Y_\fm)$
 defined in Section 
\ref{section_log_cy_surface}. 
Strictly speaking, $Y_\fm$ is not smooth, but log smooth. We can either make $Y_\fm$ smooth by toric blow-ups or allow log smooth objects in the definition of log Calabi-Yau.
Then $(Y_\fm,\partial Y_\fm)$
is log Calabi-Yau with respect to every class $\beta \in H_2(Y_\fm,\Z)$ and so in 
particular with respect to the classes 
$\beta_p \in H_2(Y_\fm,\Z)$
defined in Section \ref{section_curve_classes_log_cy}.
\item If $Y$ is a del Pezzo surface and $\partial Y$ a smooth anticanonical divisor, then 
$(Y,\partial Y)$ is log Calabi-Yau with respect to every class 
$\beta \in H_2(Y,\Z)$.
\end{itemize}

We fix $(Y,\partial Y)$ log Calabi-Yau with respect to some $\beta \in H_2(Y,\Z)$ such that $\beta \cdot (\partial Y) \neq 0$.
Let 
$\overline{M}_{g,\beta} (Y/\partial Y)$ 
be the moduli space of genus $g$ stable 
log maps to $Y$ of class $\beta$ and full tangency of order $\beta \cdot (\partial Y)$
at a single unspecified point of $D$. 
It is a proper Deligne-Mumford stack 
coming with a $g$-dimensional virtual fundamental class
\[[\overline{M}_{g,\beta} (Y/\partial Y)]^{\virt} \,.\]
We define
\[ N_{g,\beta}^{Y/\partial Y} \coloneqq \int_{[\overline{M}_{g,\beta} (Y/\partial Y)]^{\virt}} (-1)^g \lambda_g \,.\]
If $(Y,\partial Y)$ is of the form $(Y_\fm,\partial Y_\fm)$ and $\beta$ is of the form $\beta_p$, see 
\mbox{Section
\ref{section_curve_classes_log_cy}}, then we have 
$N_{g,\beta}^{Y/\partial Y}=N_{g,p}^{Y_\fm}$, where 
$N_{g,p}^{Y_\fm}$ are the invariants defined in \mbox{Section \ref{section_log_gw_log_cy}}.

We can now formulate the rationality conjecture.

\begin{conj} \label{conj_rationality}
Let $(Y,\partial Y)$ be a log Calabi-Yau pair with respect to some class
$\beta \in H_2(Y,\Z)$ such that $\beta \cdot (\partial Y) \neq 0$. Then there exists a rational function 
\[ \overline{\Omega}_{\beta}(q^{\frac{1}{2}}) \in \Q(q^{\pm \frac{1}{2}})\,\]
such that we have the equality of power series in $\hbar$,
\[ \overline{\Omega}_{\beta}(q^{\frac{1}{2}})
=(-1)^{\beta \cdot (\partial Y) +1} \left(2 \sin
\left(\frac{\hbar}{2} \right) \right)
\left( \sum_{g \geqslant 0} N_{g,\beta}^{Y/\partial Y} \hbar^{2g-1} \right)\,,\]
after the change of variables $q=e^{i\hbar}$.
\end{conj}

Note that such rational function $\overline{\Omega}_\beta(q^{\frac{1}{2}})$ is unique if it exists.
If the rational function 
$\overline{\Omega}_\beta (q^{\frac{1}{2}})$ exists, then it is 
invariant under $q^{\frac{1}{2}}
\mapsto q^{-\frac{1}{2}}$, because its power series expansion in $\hbar$ after $q=e^{i\hbar}$ has real coefficients.
Given the 3-dimensional interpretation of the invariants 
$N_{g,\beta}^{Y,\partial Y}$ given in Section \ref{section_3d}, 
Conjecture 
\ref{conj_rationality} should follow from a log version of the MNOP conjectures,
\cite{MR2264664}, \cite{MR2264665}, 
once an appropriate theory of log Donaldson-Thomas invariants is developed.
If $\partial Y$ is smooth, then Conjecture 
\ref{conj_rationality} indeed follows from the relative MNOP conjectures, see Section
3.3 of \cite{MR2264665}.

Let $(Y,\partial Y)$ be a log Calabi-Yau pair with respect to some primitive class
$\beta \in H_2(Y,\Z)$ such that $\beta \cdot (\partial Y) \neq 0$. Let us assume that 
\mbox{Conjecture \ref{conj_rationality}} is true for all classes that are multiple of
$\beta$. So, for every $n \geqslant 1$, we have a rational function $\overline{\Omega}_{n \beta}(q^{\frac{1}{2}}) \in \Q(q^{\pm \frac{1}{2}})$.
We define a collection of rational functions 
$\Omega_{n \beta}(q^{\frac{1}{2}}) \in \Q(q^{\pm \frac{1}{2}})$, $n \geqslant 1$,
invariant under 
$q^{\frac{1}{2}}
\mapsto q^{-\frac{1}{2}}$, 
by the relations
\[ \overline{\Omega}_{n\beta} (q^{\frac{1}{2}})
= \sum_{\ell|n} \frac{1}{\ell} \frac{q^{\frac{1}{2}}-q^{-\frac{1}{2}}}{q^{\frac{\ell}{2}}-q^{-\frac{\ell}{2}}}
\Omega_{\frac{n}{\ell}\beta}(q^{\frac{\ell}{2}}) \,.\]

\begin{lem} \label{lem_mobius}
These relations have a unique solution, given by 
\[ \Omega_{n\beta} (q^{\frac{1}{2}})
= \sum_{\ell|n} \frac{\mu(\ell)}{\ell} \frac{q^{\frac{1}{2}}-q^{-\frac{1}{2}}}{q^{\frac{\ell}{2}}-q^{-\frac{\ell}{2}}}
\overline{\Omega}_{\frac{n}{\ell}\beta}(q^{\frac{\ell}{2}}) \,,\]
where $\mu$ is the M\"obius function.
\end{lem}

\begin{proof}
Indeed, we have 
\[\sum_{\ell|n} \frac{1}{\ell} \frac{q^{\frac{1}{2}}-q^{-\frac{1}{2}}}{q^{\frac{\ell}{2}}-q^{-\frac{\ell}{2}}}
\left( \sum_{\ell'|\frac{n}{\ell}} \frac{\mu(\ell')}{\ell'} \frac{q^{\frac{l}{2}}-q^{-\frac{l}{2}}}{q^{\frac{\ell \ell'}{2}}-q^{-\frac{\ell \ell'}{2}}}
\overline{\Omega}_{\frac{n}{\ell \ell'}\beta}(q^{\frac{\ell \ell'}{2}}) \right) \]
\[=\sum_{\ell|n} \sum_{\ell'|\frac{n}{\ell}} 
\frac{\mu(\ell')}{\ell \ell'} \frac{q^{\frac{1}{2}}-q^{-\frac{1}{2}}}{q^{\frac{\ell \ell'}{2}}-q^{-\frac{\ell \ell'}{2}}}
\overline{\Omega}_{\frac{n}{\ell \ell'}\beta}(q^{\frac{\ell \ell'}{2}}) 
= \sum_{m|n}
 \frac{1}{m} \frac{q^{\frac{1}{2}}-q^{-\frac{1}{2}}}{q^{\frac{m}{2}}-q^{-\frac{m}{2}}}
\Omega_{\frac{n}{m}\beta}(q^{\frac{m}{2}}) 
\left( \sum_{\ell'|m} \mu(\ell') \right)   \]
\[ = \sum_{m|n}
 \frac{1}{m} \frac{q^{\frac{1}{2}}-q^{-\frac{1}{2}}}{q^{\frac{m}{2}}-q^{-\frac{m}{2}}}
\Omega_{\frac{n}{m}\beta}(q^{\frac{m}{2}}) 
\delta_{m,1}  
=\overline{\Omega}_{n\beta}(q^{\frac{1}{2}}) \,, \]
where we used the M\"obius inversion formula
$\sum_{\ell'|m} \mu(\ell')=\delta_{m,1}$. 
\end{proof}

We can now formulate the integrality conjecture.

\begin{conj} \label{conj_integrality}
Let $(Y,\partial Y)$ be a log Calabi-Yau pair with respect to some class $\beta \in 
H_2(Y,\Z)$, such that $\beta \cdot (\partial Y)
\neq 0$, and such that the rationality Conjecture
\ref{conj_rationality} is true for all multiples of $\beta$, so that the rational functions  $\Omega_{n \beta}(q^{\frac{1}{2}}) \in \Q(q^{\pm \frac{1}{2}})$, are defined.
Then, in fact, for every $n \geqslant 1$, $\Omega_{n\beta}(q^{\frac{1}{2}})$ is a Laurent polynomial in $q^{\frac{1}{2}}$ with integer coefficients, ie
\[\Omega_{n \beta}(q^{\frac{1}{2}}) \in \Z[q^{\pm \frac{1}{2}}] \,,\]
invariant under $q^{\frac{1}{2}}
\mapsto q^{-\frac{1}{2}}$.
\end{conj}

In Section \ref{section_cecotti_vafa}, we
explain why this integrality conjecture can be interpreted in some cases as a mathematically well-defined example of the general integrality for open Gromov-Witten invariants in Calabi-Yau 3-folds
predicted by Ooguri-Vafa 
\cite{MR1765411}. In particular, the log
BPS invariants $\Omega_\beta(q^{\frac{1}{2}})$ should be thought as examples 
of Ooguri-Vafa/open BPS invariants.

In the classical limit $\hbar \rightarrow 0$, the integrality of  
 $\Omega_{n\beta} \coloneqq \Omega_{n\beta} (q^{\frac{1}{2}}=1)$ is equivalent to Conjecture 6.2 of \cite{MR2667135}.
If $\beta^2=-1$, $\beta \cdot (\partial Y)=1$, and the class $\beta$ only contains a smooth rational curve, then
it follows from the proof of Lemma \ref{lem_multicover} that 
\mbox{Conjecture \ref{conj_integrality}} is true. More precisely, we have \[\overline{\Omega}_{n\beta}(q^{\frac{1}{2}})=\frac{1}{n}
\frac{q^{\frac{1}{2}}
-q^{-\frac{1}{2}}}{q^{\frac{n}{2}}
-q^{-\frac{n}{2}}}\] 
for every 
$n \geqslant 1$, and so 
$\Omega_{\beta}(q^{\frac{1}{2}})=1$ and $\Omega_{n \beta}(q^{\frac{1}{2}})=0$ for $n>1$.

\subsection{Integrality result}
\label{section_integ_result}

\begin{lem} \label{lem_rationality}
For every $\fm=(m_1,\dots,m_n)$ an $n$-tuple of primitive non-zero vectors in $M=\Z^2$ and $p\in P=\NN^n$, the rationality Conjecture \ref{conj_rationality} is true for the log Calabi-Yau pair $(Y_\fm,\partial Y_\fm)$
with respect to the curve class $\beta_p \in H_2(Y,\Z)$.
\end{lem}

\begin{proof}
This follows from Theorem \ref{precise_main_thm_ch2}, expressing the generating series of invariants $N_{g,p}^{Y_\fm}$ as a Hamiltonian $\hat{H}_m$ attached to some ray of the quantum scattering diagram 
$S(\hat{\fD}_\fm)$, and from Proposition 
\ref{prop_scattering_tropical}, giving a formula for $\hat{H}_\fm$ whose coefficients are manifestly in 
$\Q[q^{\pm \frac{1}{2}}][(1-q^\ell)^{-1}]_{ \ell \geqslant 1}$.

Alternatively, one could argue that, because the 
initial quantum scattering diagram 
$\hat{\fD}_\fm$ is defined over $\Q[q^{\pm \frac{1}{2}}][(1-q^\ell)^{-1}]_{\ell \geqslant 1}$, the resulting consistent quantum scattering diagram $S(\hat{\fD}_\fm)$ is also 
defined over
$\Q[q^{\pm \frac{1}{2}}][(1-q^\ell)^{-1}]_{\ell \geqslant 1}$ and so \mbox{Lemma \ref{lem_rationality}} follows directly from 
\mbox{Theorem \ref{precise_main_thm_ch2}}. 
\end{proof}

By Lemma
\ref{lem_rationality}, we have rational functions 
\[ \overline{\Omega}_p^{Y_\fm}(q^{\frac{1}{2}}) \in \Q(q^{\pm \frac{1}{2}})\,,\]
such that 
\[\overline{\Omega}_p^{Y_\fm}(q^{\frac{1}{2}}) =(-1)^{\ell_p+1}
\left(2 \sin
\left(\frac{\hbar}{2} \right) \right)
\left( \sum_{g \geqslant 0} N_{g,p}^{Y_\fm} \hbar^{2g-1} \right)\,,\]
as power series in $\hbar$, after the change of variables $q=e^{i \hbar}$. Note that we used the fact that 
$\beta_p.(\partial Y_\fm)=\ell_p$.

The following result is, after Theorem
\ref{precise_main_thm_ch2}, the second main 
result of this paper. It is a precise form of Theorem \ref{main_thm_integ} in the Introduction.

\begin{thm} \label{thm_integ}
For every $\fm=(m_1,\dots,m_n)$ an $n$-tuple of primitive non-zero vectors in $M=\Z^2$ and $p \in P=\NN^n$, the 
integrality Conjecture \ref{conj_integrality} 
is true for the log Calabi-Yau pair $(Y_\fm,\partial Y_\fm)$
with respect to the class $\beta_p \in H_2(Y_m,\Z)$.
In other words, there exists 
$\Omega_p^{Y_\fm}(q^{\frac{1}{2}}) \in \Z[q^{\pm \frac{1}{2}}]$
such that 
\[\overline{\Omega}_p^{Y_\fm}(q^{\frac{1}{2}})
= \sum_{p=\ell p'} \frac{1}{\ell}
\frac{q^{\frac{1}{2}}-q^{-\frac{1}{2}}}{
q^{\frac{\ell}{2}}-q^{-\frac{\ell}{2}}} \Omega_{p'}^{Y_\fm}(q^{\frac{\ell}{2}}) 
\,.\]
\end{thm}

The proof of Theorem 
\ref{thm_integ} takes Sections
\ref{section_quad_ref} and
\ref{section_proof_integ}.

\subsection{Quadratic refinement}
\label{section_quad_ref}

According to Theorem \ref{precise_main_thm_ch2}, generating series of the log Gromov-Witten invariants $N_{g,p}^{Y_\fm}$ are Hamiltonians attached to the rays of some quantum scattering diagram $S(\hat{\fD}_{\fm})$.
Our integrality result, Theorem
\ref{thm_integ}, will follow from
a general integrality result for scattering diagrams. Our main input, the integrality result of \cite{MR2851153}, is phrased in terms of twisted quantum scattering diagrams, ie scattering diagrams valued in automorphisms of twisted quantum tori. The comparison with quiver DT
invariants, done in Section \ref{section_dt}, also requires
us to consider twisted quantum scattering diagrams. 

In the present Section, we explain how to 
compare the quantum scattering diagram 
$S(\hat{\fD}_\fm)$ with a twisted quantum 
scattering diagram 
$S(\hat{\fD}^{\tw}_\fm)$. This comparison requires the notion of quadratic refinement.
A short and to the point discussion by Neitzke can be found in \cite{neitzke2014comparing}.
Some related discussion can be found in Appendix A of 
\cite{lin2017correspondence}.

We start with $P=\NN^n=\oplus_{j=1}^n \NN e_j$. For $p=(p_1,\dots,p_n) \in P=\NN^n$, we denote
$\ord(p) \coloneqq \sum_{j=1}^n p_j$.
An $n$-tuple $\fm=(m_1,\dots,m_n)$
of primitive non-zero vectors in 
$M=\Z^2$ naturally defines an additive map 
\[ r \colon P \rightarrow M\]
\[ e_j \mapsto m_j\,.\]

For every 
$\Z[q^{\pm \frac{1}{2}}]$-algebra
$A$,
we denote by $\hat{T}_{P,\tw}^{A}$ the non-commutative ``space" whose algebra of functions is the 
algebra
$\Gamma (\cO_{\hat{T}_{P,\tw}^A})$
given by 
$A[\![P]\!]$, 
power series in $\hat{x}^p$ for
$p \in P$, with coefficients in $A$, 
with the product defined by 
\[ \hat{x}^p \cdot \hat{x}^{p'}=(-1)^{
\langle r(p),r(p') \rangle}
q^{\frac{1}{2}\langle r(p),r(p') \rangle}\hat{x}^{p+p'} \,.\]
The main difference with respect to the formalism of Section \ref{section_scattering} is the twist by the extra sign $(-1)^{
\langle r(p),r(p') \rangle}$.

We will use $A=\Z[\![q^{\pm \frac{1}{2}}]\!]$, 
$\Z(\!(q^{\frac{1}{2}})\!)$ and
$\Q(\!(q^{\frac{1}{2}})\!)$.
We have obviously the inclusions
\[ \Gamma \left(\cO_{\hat{T}_{P,\tw}^{\Z[\![q^{\pm 1/2}]\!]}}\right)
\subset 
\Gamma \left(\cO_{\hat{T}_{P,\tw}^{\Z(\!(q^{1/2})\!)}}\right)
\subset
\Gamma \left( \cO_{\hat{T}_{P,\tw}^{\Q(\!(q^{1/2})\!)}} \right) \,.\]

Every 
\[\hat{H}^\tw = \sum_{p\in P}
\hat{H}^\tw_p \hat{x}^p \in
\Gamma \left( \cO_{\hat{T}_{P,\tw}^{\Q(\!(q^{1/2})\!)}} \right) \,,\]
such that $\hat{H}^\tw =0 \mod P$, defines 
via conjugation by 
$\exp \left( \hat{H}^\tw \right)$
an automorphism 
\[\hat{\Phi}^\tw_{\hat{H}^\tw} 
= \Ad_{\exp \left( \hat{H}^\tw \right)}
=\exp \left( \hat{H}^\tw \right)
\left(-\right) \exp \left( -\hat{H}^\tw \right)\]
of 
$\Gamma \left( \cO_{\hat{T}_{P,\tw}^{\Q(\!(q^{1/2})\!)}} \right)$.

\begin{defn}
A twisted quantum scattering diagram
$\hat{\fD}^{\tw}$ over \mbox{$(r \colon P
\rightarrow M)$} is a set of rays 
$\fd$ in $M_\R$,
equipped with elements 
\[ \hat{H}^\tw_\fd
\in \Gamma \left( \cO_{\hat{T}_{P,\tw}^{\Q(\!(q^{1/2})\!)}} \right) \,,\]
such that:
\begin{itemize}
\item There exists a primitive $p \in P$
(which is necessarily unique)
such that $\hat{H}^\tw_\fd \in  
\hat{x}^p \Q(\!(q^{\frac{1}{2}})\!)[\![\hat{x}^p]\!]$ and either $r(p) \in 
-\NN_{\geqslant 1} m_\fd$ or 
$r(p) \in \NN_{\geqslant 1} m_\fd$. We say 
that the ray $(\fd,\hat{H}^\tw_{\fd})$ is ingoing if 
$r(p) \in 
-\NN_{\geqslant 1} m_\fd$ and outgoing if $r(p) \in \NN_{\geqslant 1} m_\fd$. We call $p$ the 
$P$-direction of the ray $(\fd, \hat{H}^\tw_\fd)$.
\item For every $\ell \geqslant 0$, there are only finitely many rays $\fd$ of 
$P$-direction $p$ satisfying 
$\ord(p) \leqslant \ell$.
\end{itemize}

\end{defn}

Using the automorphisms 
$\hat{\Phi}^\tw_{\hat{H}^\tw}$, we define as in Section \ref{section_scattering_diag} the notion of consistent twisted quantum scattering diagram
and one can prove that every twisted quantum scattering diagram 
$\fD^\tw$ can be canonically completed by adding only outgoing rays to form a consistent twisted quantum scattering diagram $S(\fD^\tw)$.

The following Lemma will give us a way to go back and forth between quantum scattering diagrams and twisted quantum scattering diagrams.

\begin{lem} \label{lem_quad_ref_M}
The map 
$\sigma_M \colon M \rightarrow \{ \pm 1\}$, defined by $\sigma_M (0)=1$ and 
$\sigma_M (m)=(-1)^{|m|}$ for $m \in M$
non-zero, where $|m|$ is the divisibility of 
$m$ in $M$, is a quadratic refinement of 
\[\wedge^2 M \rightarrow \{ \pm 1 \}\]
\[ (m_1,m_2) \mapsto (-1)^{\langle m_1,
m_2 \rangle} \,,\]
ie we have 
\[ \sigma_M (m_1+m_2)
=(-1)^{\langle m_1, m_2 \rangle}
\sigma_M (m_1) \sigma_M (m_2) \,,\]
for every $m_1, m_2 \in M$. It is the unique
quadratic refinement such that 
$\sigma_M(m)=-1$ for every $m\in M$ primitive.
\end{lem}

\begin{proof}
We fix a basis of $M$ and we denote by
$m=(m^x,m^y)$ the coordinates of some 
$m \in M$ in this basis. We define 
$\sigma'_M \colon M \rightarrow \{ \pm 1\}$ by 
\[ \sigma'_M(m)=(-1)^{m^x m^y + m^x + m^y} \,.\]
It is easy to check that $\sigma'_M$ is a quadratic refinement
of $(-1)^{\langle -,- \rangle}$:
the parity of 
\[(m_1^x+m_2^x)(m_1^y+m_2^y)+m_1^x +m_2^x
+m_1^y+m_2^y\]
differs from the parity of 
\[ m_1^x m_1^y +m_1^x+m_1^y+m_2^x m_2^y
+m_2^x+m_2^y\]
by $m_1^x m_2^y + m_2^x m_1^y$, which has the parity of $\langle m_1,m_2\rangle$.

If $m \in M$ is primitive, then $(m^x,m^y)$ is equal to $(1,0)$, $(0,1)$ or $(1,1)$ modulo two, 
and in all three cases, we obtain $\sigma'_M(m)=-1$. Combined with the fact that $\sigma'_M$ is a quadratic refinement, 
this implies that, for every $m \in M$, we have $\sigma'_M(m)=(-1)^{|m|}$, ie
$\sigma'_M=\sigma_M$. In particular,
$\sigma_M$ is a quadratic refinement and
$\sigma'_M$ is independent of the choice of basis.

The uniqueness statement follows from the fact that a quadratic refinement is determined by its value on a basis of $M$.
\end{proof}

Let $\hat{\fD}_\fm^\tw$ be the twisted quantum
scattering diagram consisting of 
incoming rays $(\fd_j, \hat{H}^\tw_{\fd_j})$,
$1 \leqslant j \leqslant n$, where 
\[\fd_j = -\R_{\geqslant 0}m_j \,,\]
and 
\[ \hat{H}^\tw_{\fd_j} = 
-\sum_{\ell \geqslant 1}
\frac{1}{\ell}
\frac{1}{q^{\frac{\ell}{2}}
-q^{-\frac{\ell}{2}}} \hat{x}^{\ell e_j}
\in \Gamma \left( \cO_{\hat{T}_{P,\tw}^{\Q(\!(q^{1/2})\!)}} \right) \,,\]
where we consider 
\[\frac{1}{q^{\frac{\ell}{2}}
-q^{-\frac{\ell}{2}}}=-q^{\frac{\ell}{2}}
\sum_{k \geqslant 0} q^{k\ell} \in \Q(\!(q^{\frac{1}{2}})\!)\,.\] 
Let $S(\hat{\fD}_\fm^\tw)$ be the corresponding 
consistent twisted quantum scattering diagram obtained by adding only outgoing rays.

Define $\sigma_P
\colon P \rightarrow \{\pm 1\}$
by $\sigma_P \coloneqq \sigma_M \circ r$.
It follows from Lemma \ref{lem_quad_ref_M} that $\sigma_P$ is a quadratic refinement and so
\[ \left( \prod_{j=1}^n t_j^{p_j}
\right) \hat{z}^{r(p)} \mapsto \sigma_P(p) \hat{x}^p \,,\]
is an algebra isomorphism between quantum tori and twisted quantum tori.
Using this isomorphism, we can construct a 
twisted quantum scattering diagram 
$S(\hat{\fD}_{\fm})^\tw$ from the quantum 
scattering diagram
$\hat{\fD}_\fm$.

The incoming rays of 
$S(\hat{\fD}_\fm)^\tw$ are $(\fd_j, 
\hat{H}_{\fd_j}^\tw)$, 
$1 \leqslant j \leqslant n$, where 
$\fd_j=-\R_{\geqslant 0} m_j$ and 
\[ \hat{H}_{\fd_j}^\tw=- 
\sum_{\ell \geqslant 1}\frac{1}{\ell} \frac{1}{q^{\frac{\ell}{2}}-q^{-\frac{\ell}{2}}} 
\hat{x}^{\ell e_j} \,.\]

The outgoing rays of $S(\hat{\fD}_\fm)^\tw$
are $(\R_{\geqslant 0}m, \hat{H}_m^\tw)$ where
\[\hat{H}_m^\tw
=-\sum_{p\in P_m}
\frac{\overline{\Omega}_p(q^{\frac{1}{2}})}{q^{\frac{1}{2}}-q^{-\frac{1}{2}}} \hat{x}^p
=-\sum_{p\in P_m}
\sum_{p=\ell p'} 
\frac{1}{\ell} \frac{1}{q^{\frac{\ell}{2}}
-q^{-\frac{\ell}{2}}}
\Omega_{p'}^{Y_\fm}(q^{\frac{\ell}{2}})
\hat{x}^p\,.\]

\begin{lem} \label{lem_twist}
We have $S(\hat{\fD}_\fm^\tw)
=S(\hat{\fD}_\fm)^\tw$.
\end{lem}

\begin{proof}
As $(\prod_{j=1}^n t_j^{p_j}) \hat{z}^{r(p)} \mapsto \sigma_P(p) \hat{x}^p$ is an algebra isomorphism, the twisted quantum scattering diagram 
$S(\hat{\fD}_\fm)^\tw$ is consistent and so the result follows from the uniqueness of the consistent completion of twisted quantum scattering diagrams.
\end{proof}

\subsection{Proof of the integrality theorem}
\label{section_proof_integ}

We give below the proof of Theorem 
\ref{thm_integ}. It is a combination of 
the scattering arguments of Appendix C3 of
\cite{MR3758151}
with the formalism of quantum admissible series of \cite{MR2851153}. Because of the 
structure of the induction argument, we will in fact prove a more general statement than 
Theorem \ref{thm_integ}. We will prove, 
as Proposition \ref{prop_integ}, that the consistent completion of any (twisted) quantum scattering with incoming rays equipped with Hamiltonians satisfying some BPS integrality condition 
has outgoing rays equipped with Hamiltonians 
satisfying the BPS integrality condition.

We fix $p \in P$ primitive.
Consider 
\[ \hat{H}^\tw = \sum_{\ell \geqslant 1}
\hat{H}^\tw_\ell(q^{\frac{1}{2}}) \hat{x}^{\ell p}  \in 
\hat{x}^p \Q(\!(q^{\frac{1}{2}})\!)
[\![\hat{x}^{p}]\!] \,. \]
We define
\[ \overline{\Omega}_\ell(q^{\frac{1}{2}})
\coloneqq -(q^{\frac{1}{2}}
-q^{-\frac{1}{2}}) \hat{H}_\ell^\tw(q^{\frac{1}{2}}) \in \Q(\!(q^{\frac{1}{2}})\!) \,,\]
and 
\[\Omega_\ell(q^{\frac{1}{2}})
\coloneqq \sum_{\ell'|\ell} \frac{\mu(\ell')}{\ell'}
\frac{q^{\frac{1}{2}}-q^{-\frac{1}{2}}}{
q^{\frac{\ell}{2}}-q^{-\frac{\ell}{2}}}
\overline{\Omega}_{\frac{\ell}{\ell'}}
(q^{\frac{\ell}{2}})
 \in \Q(\!(q^{\frac{1}{2}})\!) \,.
\]
It follows from Lemma \ref{lem_mobius} that we have 
\[ \hat{H}^\tw=
-\sum_{n \geqslant 1}
\sum_{\ell \geqslant 1}
\frac{1}{\ell} 
\frac{\Omega_{n}(q^{\frac{l}{2}})}
{q^{\frac{\ell}{2}}
-q^{-\frac{\ell}{2}}}
\hat{x}^{\ell np}
 \,.\]
 
\begin{defn} We say that 
$\hat{H}^\tw\in 
\hat{x}^p \Q(\!(q^{\frac{1}{2}})\!)
[\![\hat{x}^{p}]\!]$ satisfies the BPS 
integrality condition if each corresponding 
$\Omega_\ell(q^{\frac{1}{2}}) \in \Q(\!(q^{\frac{1}{2}})\!)$ 
is in fact a Laurent polynomial with integer coefficients, 
ie $\Omega_\ell(q^{\frac{1}{2}}) \in \Z[q^{\frac{1}{2}}]$.
\end{defn}

The function $\hat{H}^\tw$ satisfies the BPS 
integrality condition if and only if 
$\exp \left(\hat{H}^\tw \right)$ is admissible in the sense of 
Section 6 of \cite{MR2851153}.

It follows from the product form of the quantum dilogarithm, as recalled
in \mbox{Section \ref{section_statement}}, that if $\hat{H}^\tw$ satisfies the BPS integrality condition, then $\hat{\Phi}^\tw_{\hat{H}^\tw}$ preserves the subring  
$\Gamma \left( \cO_{\hat{T}_{P,\tw}^{\Z[\![q^{1/2}]\!]}} \right)$
of $\Gamma \left( \cO_{\hat{T}_{P,\tw}^{\Q(\!(q^{1/2})\!)}} \right)$.
We refer to the subgroup of automorphisms of 
$\Gamma \left( \cO_{\hat{T}_{P,\tw}^{\Z[\![q^{1/2}]\!]}} \right)$ generated by automorphisms of the form $\hat{\Phi}^\tw_{\hat{H}^\tw}$ with 
$\hat{H}^\tw$ satisfying the BPS 
integrality condition as the BPS quantum tropical vertex 
group; this subgroup is called the quantum tropical vertex group in
\cite{MR2851153}, .

We fix a choice of 
twisted quantum scattering diagram in each
equivalence class by considering to be distinct rays with different $P$-directions and by merging rays with coinciding supports and  with the same $P$-direction.

Recall that for $p=(p_1,\dots,p_n) \in P=\NN^n$, we denote
$\ord(p) \coloneqq \sum_{j=1}^n p_j$.
It is simply the total degree of the monomial in several variables 
$\prod_{j=1}^n t_j^{p_j}$.

\begin{lem} \label{lem_induction}
Let $N$, $n$ and $n_I$ be a positive integers.
Let
$r \colon P=\NN^n \rightarrow M$ be an additive map.
Let 
$(p^1,\dots,p^{n_I})$ be an 
$n_I$-tuple of primitive vectors in 
$P$.
Let $\hat{\fD}^\tw$ be a twisted quantum scattering diagram over 
$(r \colon P \rightarrow M)$, consisting of incoming rays $(\fd_j,\hat{H}_{\fd_j}^\tw)$ for
$1 \leqslant j 
\leqslant n_I$, with
$\fd_j=-\R_{\geqslant 0}r(p^j)$ 
and 
$\hat{H}_{\fd_j}^\tw \in \hat{x}^{p^j}
\Q(\!(q^{\frac{1}{2}})\!)[\![\hat{x}^{p^j}]\!]$
satisfying the BPS integrality condition. Then, every outgoing ray 
$(\fd, \hat{H}_\fd^\tw)$ of the consistent twisted quantum scattering diagram $S(\hat{\fD}^\tw)$, whose $P$-direction $p$ satisfies $\ord(p) \leqslant N$, is such that
$\hat{H}^\tw_\fd \in 
\hat{x}^p
\Q(\!(q^{\frac{1}{2}})\!)[\![\hat{x}^p]\!]$
satisfies the BPS integrality condition.
\end{lem}

\begin{proof}

We prove the result by induction on $N$.
The result is obviously true for $N=1$: 
the only outgoing rays with $P$-direction $p$ satisfying $\ord (p)=1$ are obtained by straight propagation of the initial rays and so satisfy the BPS integrality condition if it is the case for the initial rays.

Let $N>1$ be an integer.
We assume by induction that 
\mbox{Lemma \ref{lem_induction}}
is true for all integers strictly less than $N$ and we want to prove it for $N$.
As in Step III of Appendix C3 of
\cite{MR3758151}, up to applying the perturbation trick, 
which consists of separating transversally 
and generically the initial rays with the same support and then looking at the new local scatterings, we can assume that at most two initial rays have order one.

We now use the change of monoid trick, as in Steps I and IV of Appendix C3 of
\cite{MR3758151}. 
Write $P'=\oplus_{j=1}^{n_I} \NN e_j'$
and 
\[ r' \colon P' \rightarrow M\] 
\[e_j' \mapsto r'(e_j) \coloneqq r(p^j)\,.\]
Let $\hat{\fD}^{\tw'}$ be the twisted quantum scattering diagram over
$(r' \colon P' \rightarrow M)$
obtained by replacing
$\hat{x}^{p^j}$ by $\hat{x}^{e_j'}$
in $\hat{H}_{\fd_j}^\tw$.
Write
\[u \colon P' \rightarrow P\]
\[e_j' \mapsto p^j\,.\]

Let $(\fd, \hat{H}_\fd^\tw)$ be an outgoing ray of $S(\hat{\fD}^\tw)$, whose $P$-direction $p$ satisfies $\ord(p)=N$. Then
$(\fd, \hat{H}_\fd^\tw)$ is the sum of images by $u$ of outgoing rays
of $S(\hat{\fD}^{\tw'})$, of $P'$-direction mapping to $p$ by $u$. Let 
$(\fd', \hat{H}_{\fd'}^\tw)$ be such an outgoing ray of $S(\hat{\fD}^{\tw'})$.

Writing $p'=\sum_{j=1}^{n_I} p_j'e_j'$, where
$(p_1',\dots,p_n') \in \NN^{n_I}$, 
we have 
\[ \ord(p')=\ord \left(\sum_{j=1}^{n_I} p_j' e_j' \right)
=\sum_{j=1}^{n_I} p_j' \,,\]
whereas 
\[\ord(p)=\ord \left(\sum_{j=1}^{n_I} p_j' p^j\right) 
=\sum_{j=1}^{n_I} p_j' \ord(p^j)\,.\]
If only two $p_j'$ are non-zero, then the ray
$(\fd',\hat{H}_{\fd'}^\tw)$ belongs to 
a twisted quantum scattering diagram with two incoming rays and so its BPS integrality follows from Proposition 9 of 
\cite{MR2851153}. If more than two of the $p_j'$ are non-zero, then, at least one of the $p^j$ with $n_j \neq 0$ satisfies 
$\ord (p^j) \geqslant 2$ and so 
$\ord (p') <\ord(p)$. The BPS integrality of the ray $(\fd',\hat{H}_{\fd'}^\tw)$ then follows by the induction hypothesis.
\end{proof}

\begin{prop} \label{prop_integ}
Let $n_I$ be a positive integer and let
$(p^1,\dots,p^{n_I})$ be an 
$n_I$-tuple of primitive vectors in 
$P$.
Let $\hat{\fD}^\tw$ be a twisted quantum scattering diagram over 
$(r \colon P \rightarrow M)$, consisting of incoming rays $(\fd_j,\hat{H}_{\fd_j}^\tw)$, 
$1 \leqslant j 
\leqslant n_I$, with
$\fd_j=-\R_{\geqslant 0}r(p^j)$ 
and 
$\hat{H}_{\fd_j}^\tw \in \hat{x}^{p^j}
\Q(\!(q^{\frac{1}{2}})\!)[\![\hat{x}^{p^j}]\!]$
satisfying the BPS integrality condition. Then the consistent twisted quantum scattering diagram $S(\hat{\fD}^\tw)$ is such that for every outgoing ray
$(\fd, \hat{H}^\tw_\fd)$, of 
$P$-direction $p \in P$,  we have
that 
$\hat{H}^\tw_\fd \in 
\hat{x}^p
\Q(\!(q^{\frac{1}{2}})\!)[\![\hat{x}^p]\!]$
satisfies the BPS integrality condition.
\end{prop}

\begin{proof}
This follows immediately from 
Lemma \ref{lem_induction}.
\end{proof}

We can now finish the proof of Theorem
\ref{thm_integ}. By Theorem \ref{precise_main_thm_ch2} and Lemma \ref{lem_twist}, it is enough to show that the outgoing rays of the twisted quantum scattering diagram $S(\hat{\fD}^\tw_\fm)$ satisfy the BPS integrality condition. 
As the initial rays of $S(\hat{\fD}^\tw_\fm)$ satisfy the BPS integrality condition, the result follows from Proposition 
\ref{prop_integ}.

\subsection{Integrality and quiver DT invariants}
\label{section_dt}
We refer to \cite{kontsevich2008stability},
\cite{MR2951762},
\cite{MR2650811},
\cite{MR2801406},
\cite{meinhardt2017donaldson}
for the Donaldson-Thomas (DT) 
theory of quivers.

For every $\fm=(m_1,\dots,m_n)$ an $n$-tuple of primitive non-zero vectors in 
$M=\Z^2$, we define a quiver $Q_\fm$, with set of vertices 
$\{1,2,\dots,n\}$ and, for every $1 \leqslant j,k \leqslant n$, $\langle m_j, m_k \rangle_+
\coloneqq \max (\langle m_j, m_k \rangle, 0)$ arrows from the vertex $j$ to the vertex $k$. 
We identify $P=\oplus_{j=1}^n \NN e_j$ with the set of dimension vectors for the quiver $Q_\fm$.

\begin{lem}
The quiver $Q_\fm$ is acyclic, ie does not contain any oriented cycle, if and only if the $n$ vectors $m_1, \dots, m_n$ are 
all contained in a closed half-plane of $M_\R=\R^2$.
\end{lem}

\begin{proof}
The quiver $Q_\fm$ contains an arrow from the vertex $i$ to the vertex $j$ if and only if $(m_i,m_j)$ is an oriented basis of $\R^2$.
\end{proof}

Let us assume that the quiver $Q_\fm$ is acyclic. 
For every $p \in P$, we consider the hyperplane 
\[p^{\perp}:=\{ \theta \in P_\R^{\vee}\,|\, 
\theta(p)=0\}\] in the dual $P_\R^{\vee}$ of the $n$-dimensional $\R$-vector space $P_\R:=P \otimes \R=
\oplus_{j=1}^n \R e_j$.
Every point $\theta \in p^{\perp}$
defines a notion of stability for representations of $Q_\fm$ of dimension $p$, and we have a 
projective variety $M_p^{\theta-ss}$, moduli space of $\theta$-semistable representations of $Q_\fm$ of dimension $p$,
containing the open smooth locus 
$M_p^{\theta-st}$ of $\theta$-stable representations. 

Let $\{-,-\}$ be the skew-symmetric form on $P_\R$ defined by $\{e_j,e_k\}:=\langle m_j,m_k\rangle$ for every $1 \leq j,k \leq n$. Let $K \subset P_\R$ be the kernel of $\{-,-\}$, that is the set of $p\in P_\R$ such that $
\{p,p'\}=0$ for every $p' \in P_\R$. We denote by 
$\rho: P_\R^\vee \rightarrow K_\R^\vee$ the dual projection. 
Following \cite[\S 2.2]{meinhardt2017donaldson}, we say that a stability $\theta$ is $\infty$-generic\footnote{Here, $\infty$ is the slope of $-\theta(p)+i$ when $\theta(p)=0$. } if for every $p_1,p_2\in P$,  $\theta(p_1)=\theta(p_2)=0$ implies $\{p_1,p_2\}=0$.

Let
$\iota \colon M_p^{\theta-st} \hookrightarrow M_p^{\theta-ss}$
be the natural inclusion. 
The main result of \cite{meinhardt2017donaldson} is that, if $\theta \in p^\perp$ is $\infty$-generic, 
the Laurent polynomials
\[ \Omega_p^{Q_\fm, \theta}
(q^{\frac{1}{2}})
\coloneqq 
(-1)^{\dim M_p^{\theta-ss}} q^{-\frac{1}{2}
\dim M_p^{\theta-ss}}
\sum_{j=0}^{\dim M_p^{\theta-st}}
\left( \dim H^{2j} (M_p^{\theta-ss}, \iota_{!*}\Q)\right)
q^j \]
\[\in  (-1)^{\dim M_p^{\theta-ss}}q^{-\frac{1}{2}
\dim M_p^{\theta-ss}} \NN[q] \]
are the refined DT invariants of $Q_\fm$ for the stability $\theta$.
In the above formula,
$\iota_{!*}$ is the intermediate extension functor defined by $\iota$ and so
$\iota_{!*} \Q$ is a perverse sheaf on
$M_p^{\theta-ss}$.

If all $m_1,\dots,m_n$ are collinear, then the quiver $Q_{\mathfrak{m}}$ has no arrows, the form $\{-,-\}$ is zero, the DT invariants of $Q_{\mathfrak{m}}$ are independent of $\theta$, we write simply $\Omega_p^{Q_\fm}(q^{\frac{1}{2}})$ 
for $\Omega_p^{Q_\fm,\theta}(q^{\frac{1}{2}})$, and they are given explicitly by $\Omega_{e_j}^{Q_\fm}(q^{\frac{1}{2}})=1$,
for all $1 \leqslant j \leqslant n$, and 
$\Omega_p^{Q_\fm}(q^{\frac{1}{2}})=0$ for $p \in P-\{e_1,\dots,e_n\}$.

If not all $m_1,\dots,m_n$ are collinear, the skew-symmetric form $\{-,-\}$ has rank $2$, the kernel $K$ is of dimension $n-2$, and the subspace $\rho^{-1}(0) \subset P_\R^\vee$ is of dimension $2$. In fact, we have an isomorphism
\begin{align*}
M_\R =M \otimes \R &\xrightarrow{\sim} \rho^{-1}(0) \\
m_i &\mapsto \{-, e_i\} \,.
\end{align*}
Note that for every $p\in P$, we have $\{-,p\} \in \rho^{-1}(0)$, because for every $k \in K$, we have $\{k,p\}=0$. 

For every $p \in P$, we denote $\Omega_p^{Q_\fm}(q^{\frac{1}{2}})$ 
for $\Omega_p^{Q_\fm,\theta}(q^{\frac{1}{2}})$ where $\theta$ is the \emph{anti-attractor stability} given by
$\theta:= \{-,p\} \in \rho^{-1}(0) \subset P_\R^\vee$. As $\{-,-\}$ is of rank $2$, the anti-attractor stability $\theta:= \{-,p\}$ is $\infty$-generic: if $\{p_1,p\}=\{p_2,p\}=0$, then both $p_1$ and $p_2$ are of the form a multiple of $p$ plus an element of $K$, and so $\{p_1,p_2\}=0$.

\begin{thm} \label{thm_dt}
For every $\fm=(m_1, \dots, m_n)$ 
such that the quiver $Q_\fm$ is acyclic, and for every $p \in P=\NN^n$, we have the equality
\[ \Omega_p^{Q_\fm}(q^{\frac{1}{2}})
=\Omega_p^{Y_\fm}(q^{\frac{1}{2}})\]
between the refined DT invariant 
$\Omega_p^{Q_\fm}(q^{\frac{1}{2}})$ of the quiver $Q_\fm$ and the log BPS invariant 
$\Omega_p^{Y_\fm}(q^{\frac{1}{2}})$
of the log Calabi-Yau surface $Y_\fm$.
\end{thm}

\begin{proof}
According to Theorem \ref{precise_main_thm_ch2}
and Lemma \ref{lem_twist}, the log Gromov-Witten invariants of $Y_\fm$ are computed by the twisted quantum scattering diagram 
$S(\hat{\fD}_\fm^\tw)$.

If all $m_1,\dots,m_n$ are collinear, the result is clear because in this case $S(\hat{\fD}_\fm^\tw)$ consists only of initial rays. From now on, we assume that $m_1,\dots,m_n$ are not all collinear.
By \cite{bri17}, DT invariants of the quiver $Q_{\mathfrak{m}}$ can be organized in a consistent scattering diagram in $P_\R^\vee$ called the stability scattering diagram. According to \cite[Theorem 1.5]{bri17}, when $Q_{\mathfrak{m}}$ is acyclic, the stability scattering diagram agrees with the cluster scattering diagram described explicitly in
\cite[Theorem 11.2]{bri17} as the consistent completion of an explicit set of initial walls. 
From this explicit description, one checks that the restriction of the cluster scattering diagram to 
$\rho^{-1}(0)\simeq M_\R$ coincides with $S(\hat{\fD}_\fm^\tw)$. In particular, the log BPS invariants attached to the rays $\R_{\geq 0} r(p)$ of $S(\hat{\fD}_\fm^\tw)$ in $M_\R$
agree with the DT invariants attached to the rays $\R_{\geq 0} \{-,p\}$ in $\rho^{-1}(0)$ obtained by intersecting with $\rho^{-1}(0)$ the walls of the cluster scattering diagram contained in $p^{\perp} \subset P_\R^\vee$.
\end{proof}

In the limit $q^{\frac{1}{2}}
\rightarrow 1$, 
if $Q_\fm$ is 
complete bipartite, then 
\mbox{Theorem 
\ref{thm_dt}} 
reduces to the 
Gromov-Witten/Kronecker correspondence of 
\cite{MR2662867}, \cite{MR3004575},
\cite{MR3033514}.

Theorem \ref{thm_dt} can be viewed as a concrete example of equality between open BPS invariants and DT invariants of quivers. 
The expectation for this kind of relation goes back at least to 
\cite{cecotti2009bps}, as reviewed in Section \ref{section_cecotti_vafa}.
Related recent stories include \cite{kucharski2017bps}, 
\cite{kucharski2017knots}, where some knot invariants, which via some string theoretic duality should be examples of open BPS invariants, 
are identified with some quiver DT invariants, and \cite{zaslow2018wavefunctions}, where 
a precise correspondence between open BPS invariants of 
a certain class of Lagrangian
submanifolds in $\C^3$ and some DT invariants of quivers
is conjectured.

Theorem
\ref{thm_dt} gives a different proof of Theorem \ref{thm_integ} when $Q_\fm$ is acyclic. When $Q_\fm$ is not acyclic,
it is unclear a priori how to relate the log
BPS invariants $\Omega_p^{Y_\fm}(q^{\frac{1}{2}})$ to some DT quiver theory.
In the physics language, one should remove the contributions of non-trivial single-centered (pure Higgs) indices (see \cite{MR3080495}
and follow-ups). It is still an open question to define mathematically the corresponding operation in DT quiver theory.
The fact that the integrality given by 
Theorem \ref{thm_integ} holds even if 
$Q_\fm$ is not acyclic is probably  
additional evidence that it should be possible.

When $Q_\fm$ is acyclic, Theorem
\ref{thm_dt} gives a positivity result for the log BPS invariants $\Omega_p^{Y_\fm}
(q^{\frac{1}{2}})$. It is unclear how to prove a similar positivity result if $Q_\fm$ is not acyclic.

We finish this Section with a remark about signs. The definition of $\Omega_p^{Y_\fm}
(q^{\frac{1}{2}})$ given in
Section \ref{section_integ_result} includes a global sign 
$(-1)^{\ell_p-1}=(-1)^{\beta_p \cdot (\partial Y_\fm)-1}$, whereas the formula given above for $\Omega_p^{Q_\fm}(q^{\frac{1}{2}})$ 
includes a global sign $(-1)^{\dim M_p^{\theta-ss}}$.
Using that 
$\beta_p \cdot (\partial Y_\fm)$ and $\beta_p^2$ have the same parity by Riemann-Roch on 
$Y_\fm$, the following result gives a direct proof that these two signs are identical.

\begin{lem} \label{lem_beta_square}
For every $p \in P$, we have 
\[\dim M_p^{\theta-ss} = \beta_p^2+1 \,.\]
\end{lem}

\begin{proof}
We write $p=\sum_{j=1}^n p_j e_j \in P$.
By standard quiver theory, we have 
\[ \dim M_p^{\theta-ss} =
\sum_{j=1}^n \sum_{k=1}^n \langle m_j, m_k \rangle_+ p_j p_k - \sum_{j=1}^n p_j^2+1 \,.
 \]
By definition (Section 
\ref{section_curve_classes_log_cy}) we have 
\[ \beta_p = \nu^{*} \beta - \sum_{j=1}^n p_j E_j \,,\]
where $\nu \colon Y_\fm
\rightarrow \overline{Y}_\fm$ is the blow-up morphism 
and $\beta \in H_2(\overline{Y}_\fm,\Z)$ is defined by certain intersection numbers.
It follows that 
\[\beta_p^2 = \beta^2 - \sum_{j=1}^n p_j^2 \,.\]
From the intersection numbers defining
$\beta$, we see that the convex polygon dual to $\beta$ is obtained by successively adding the vectors $p_j m_j$ and $\ell_p m_p$,
in the order given by the counterclockwise
ordering of the $m_j$ and $m_p$ given by their argument. By standard toric geometry, 
$\beta^2$ is given by twice the area of the dual polygon and so we have 
\[\beta^2=\sum_{j=1}^n \sum_{k=1}^n 
\langle m_j,m_k\rangle_+ p_j p_k \,.\]
It follows that \[\beta_p^2 = \sum_{j=1}^n
\sum_{k=1}^n \langle m_j, m_k \rangle_+ p_j p_k - \sum_{j=1}^n p_j^2 = \dim M_p^{\theta-ss}-1 \,.\]
\end{proof}

\subsection{del Pezzo surfaces}
\label{section_del_pezzo}
In this Section, we study the conjectures of Section 
\ref{section_integ_conj} in the case where $Y$ is a del Pezzo surface $S$ and 
$\partial Y$ is a smooth anticanonical divisor $E$ of $Y$. 
In particular, $E$ is a smooth genus one curve. We formulate 
\mbox{Conjecture
\ref{conj_del_pezzo}}, precise form of 
\mbox{Conjecture \ref{conj_main}} of the Introduction.

\begin{lem} \label{lem_rationality_del_pezzo}
Let $S$ be a del Pezzo surface, and 
$E$ be a smooth anticanonical divisor of $S$.
Then, for every $\beta \in H_2(Y,\Z)$, 
the rationality Conjecture
\ref{conj_rationality} is true for the log Calabi-Yau pair
$(S,E)$ with respect to the curve class 
$\beta$.
\end{lem}

\begin{proof}
As in Section \ref{section_3d}, the invariants $N_{g,\beta}^{S/E}$
can be written as equivariant Gromov-Witten invariants of the 3-fold $S \times \C$
relative to the divisor 
$E \times \C$.
The rationality result then follows from
the Gromov-Witten/stable pairs correspondence for the relative 3-fold geometry $S \times \C/E \times \C$.

This case of the Gromov-Witten/stable pairs correspondence can be proved following 
\mbox{Section 5.3} of
\cite{MR2746343}.
This involves considering the degeneration
of $S \times \C$ 
to the normal cone of $E \times \C$. 
Let $N$ be the normal bundle to 
$E$ in $S$.
The degeneration formula expresses 
equivariant Gromov-Witten/stable pairs theories of $S \times \C$, without insertions, 
in terms of the relative equivariant Gromov-Witten/stable pairs theories, without 
insertions, of 
$S \times \C/E \times \C$ and
$\PP(N \oplus \cO_E) \times \C /
E \times \C$.

The 3-fold $S \times \C$ is deformation equivalent to a toric 3-fold.
Indeed, a del Pezzo surface is deformation equivalent to a (not necessarily del Pezzo) toric surface: 
if $S$ is a blow-up of $\PP^2$ in $n$ points, then $S$ is deformation equivalent to a surface obtained by $n$ successive toric blow-ups of $\PP^2$.
Therefore, the Gromov-Witten/stable pairs correspondence
for $S \times \C$, without insertions, 
follows from Section 5.1 of \cite{MR2746343}. 

The equivariant Gromov-Witten/stable pairs theory of 
$\PP(N \oplus \cO_E) \times \C /
E \times \C$ coincides with the non-equivariant Gromov-Witten theory of 
$\PP(N \oplus \cO_E) \times E/
E \times E$. The 3-fold \mbox{$\PP(N \oplus \cO_E) \times E$} is a $\PP^1$-bundle over $E \times E$ and we are considering curves of degree 0 over the second $E$ factor.
As $E \times E$ is holomorphic symplectic, the Gromov-Witten/stable pairs theories vanish unless the curve class has also degree 0 over the first $E$ factor. 
The Gromov-Witten/stable pairs correspondence for \[\PP(N \oplus \cO_E) \times E/E \times E\,,\]
without insertions,  thus follows from the Gromov-Witten/stable pairs correspondence, without insertions, for local curves. 

It follows from Proposition 6 of
\cite{MR3127827} 
that the degeneration formula can be inverted to imply the Gromov-Witten/stable pairs correspondence, without insertions, for $S \times \C/E \times \C$.
\end{proof}

By Lemma \ref{lem_rationality_del_pezzo}, 
we have rational functions 
\[\overline{\Omega}_\beta^{S/E}
(q^{\frac{1}{2}}) \in \Q(q^{\pm \frac{1}{2}}) \,,\]
such that 
 \[ \overline{\Omega}_\beta^{S/E}
(q^{\frac{1}{2}})
=(-1)^{\beta \cdot E+1} 
\left( 2 \sin \left( \frac{\hbar}{2}
\right) \right) \
\left( \sum_{g\geqslant 0} N_{g,\beta}^{S/E} 
\hbar^{2g-1} \right) \,,\]
as power series in $\hbar$, after the change of variables $q=e^{i\hbar}$.

We define 
\[ \Omega_\beta^{S/E}
(q^{\frac{1}{2}})
=\sum_{\beta=\ell \beta'}
\frac{\mu(\ell)}{\ell} \frac{q^{\frac{1}{2}}
-q^{-\frac{1}{2}}}{q^{\frac{\ell}{2}}
-q^{-\frac{\ell}{2}}}
\overline{\Omega}_{\beta'}
(q^{\frac{\ell}{2}}) \in \Q(q^{\frac{1}{2}}) \,.\]

According to Conjecture 
\ref{conj_integrality}, one should have 
$\Omega_\beta^{S/E}(q^{\frac{1}{2}})
\in \Z [q^{\pm \frac{1}{2}}]$.

Let $M_\beta$ be the moduli space of 
dimension one stable sheaves on $S$, 
of class $\beta \in H_2(S,\Z)$, and Euler
characteristic $1$. It is a smooth
projective variety of dimension 
$\beta^2+1$.
Let 
\[ \chi_q(M_\beta) \coloneqq 
q^{-\frac{1}{2}(\beta^2+1)}
\sum_{j,k=0}^{\beta^2+1}
(-1)^{j+k} h^{j,k}(M_\beta) q^j 
\in \Z[q^{\pm \frac{1}{2}}] \]
be the normalized Hirzebruch genus of $M_\beta$, where $h^{j,k}$ are the Hodge numbers. It follows from Theorem 2 of
\cite{MR2304330}, following \cite{MR1228610} and
\cite{MR1351502}, that $h^{j,k}(M_\beta)=0$ if $j \neq k$. 
In particular, $\chi_q(M_\beta)$ coincides 
with the normalized Poincar\'e polynomial of 
$M_\beta$.

\begin{conj} \label{conj_del_pezzo}
We have 
\[ \Omega_\beta^{S/E}
(q^{\frac{1}{2}})
=(-1)^{\beta^2+1}
(\beta \cdot E)
\chi_q(M_\beta) \,.\]
\end{conj}

Note that we have $\beta^2=\beta \cdot E \mod 2$
by the Riemann-Roch theorem.
In the limit $q^{\frac{1}{2}} \rightarrow 1$, 
Conjecture \ref{conj_del_pezzo} reduces to 
\begin{align*}
 N_{0,\beta}^{S/E} &= (-1)^{\beta \cdot E-1}\sum_{\beta=\ell \beta'} (-1)^{(\beta')^2+1} \frac{(\beta' \cdot E)}{\ell^2} e(M_{\beta'})\\
&=(-1)^{\beta \cdot E-1}(\beta \cdot E) \sum_{\beta=\ell \beta'}
\frac{1}{\ell^3} (-1)^{(\beta')^2+1} e(M_{\beta'}) \,,
\end{align*}
which is a known result. 
Indeed, by an application of the degeneration formula originally due to 
Graber-Hassett and generalized in
\cite{van2017local}, we have 
\[N_{0,\beta}^{S/E}=(-1)^{\beta \cdot E+1}
(\beta \cdot E)N_{0,\beta}^X\,\] 
where 
$X$ is the local Calabi-Yau 3-fold given by the total space of the canonical line bundle $K_S$ of $S$, 
and $N_{0,\beta}^X$ is the genus 0, class 
$\beta$, Gromov-Witten invariant of $X$.
So the previous formula is equivalent to
\[N_{0,\beta}^X=\sum_{\beta=\ell \beta'}
\frac{1}{\ell^3}
(-1)^{(\beta')^2+1} e(M_{\beta'}) \,,\]
which is exactly the Katz conjecture
(Conjecture 2.3 of \cite{MR2420017}) for $X$. As $X$ is deformation equivalent to a toric Calabi-Yau 3-fold, the Katz conjecture
for $X$ follows from the combination of the 
Gromov-Witten/stable pairs correspondence
(Section 5.1 of \cite{MR2746343}), the integrality result of
\cite{MR2250076} and \mbox{Theorem 6.4} of
\cite{MR2892766}.

The right-hand side $(-1)^{\beta^2+1} \chi_q(M_\beta)$ should be thought as a refined DT invariant of $X$, counting dimension one sheaves. 
From this point of view, Conjecture \ref{conj_del_pezzo} 
is an equality between a log BPS invariant on one side and a refined DT invariant on the other side, in a way completely parallel to Theorem \ref{thm_dt}.
Further conceptual evidence for Conjecture \ref{conj_del_pezzo} and a further refinement of Conjecture 
\ref{conj_del_pezzo} will be presented elsewhere.

\section{Relation with Cecotti-Vafa} \label{section_cecotti_vafa}

In this last Section, we make no claim of mathematical results or mathematical precision. 
We briefly explain how the main results of this paper are related to some previous expectations in the theoretical physics literature. 

In \cite{cecotti2009bps}, Cecotti-Vafa have given a physical derivation of the fact that the refined BPS indices of a $\cN=2$ 4d quantum field theory admitting a Seiberg-Witten curve satisfy the refined Kontsevich-Soibelman wall-crossing formula. To make a connection with Theorem \ref{precise_main_thm_ch2}, we focus on only one part of the argument, establishing the relation between open Gromov-Witten invariants and the wall-crossing formula via Chern-Simons theory.
In particular, we do not discuss the application to the BPS spectrum of 
$\cN=2$ 4d quantum field theories, which 
would be related to our 
\mbox{Section
\ref{section_dt}} on quiver DT invariants.

\subsection{Summary of the Cecotti-Vafa argument}

\label{summary_cecotti_vafa}

Let $U$ be a non-compact hyperk\"ahler manifold, $(I,J,K)$ be a quaternionic triple of compatible complex structures, 
$(\omega_I, \omega_J, \omega_K)$ be the corresponding triple of real symplectic forms and 
$(\Omega_I, \Omega_J, \Omega_K)$ be the 
corresponding triple of holomorphic symplectic forms. 
In \cite{cecotti2009bps}, Cecotti-Vafa consider $U=\C^2$ but the generalization to an arbitrary hyperk\"ahler surface is clear and is considered for example in \cite{cecotti2010r}
(in particular Appendix B).

Let $\Sigma \subset U$ be an $I$-holomorphic Lagrangian subvariety of $U$, ie a submanifold such that $\Omega_I|_\Sigma =0$. It is a complex subvariety for the complex structure $I$ and a real Lagrangian for any of the real symplectic forms
$(\cos \theta) \omega_J + (\sin \theta) \omega_K$ for $\theta \in \R$.
There is in fact a twistor sphere 
$J_\zeta$,  where $\zeta \in \PP^1$, of compatible complex structures, such that $I=J_0$, 
$J=J_1$ and $K=J_i$.

Let $X$ be the non-compact Calabi-Yau 3-fold,
of underlying real manifold $U \times \C^{*}$,
equipped with a complex structure twisted in a twistorial way, ie such that the fiber over 
$\zeta \in \C^{*}$ is the complex variety 
$(U,J_\zeta)$. Consider $S^1 \subset \C^{*}$ and 
$L \coloneqq \Sigma \times S^1 \subset X$.

We consider counts of holomorphic maps 
$(C, \partial C) \rightarrow (X,L)$ from an open Riemann surface $C$ to $X$ with boundary $\partial C$ mapping to $L$.
Usually, boundary conditions for counts of open holomorphic curves are taken be Lagrangian submanifolds. In fact, $L$ is not Lagrangian in $X$ but only totally real. Combined with specific aspects of the twistorial geometry, it is probably enough to have well-defined open Gromov-Witten invariants. As suggested in \cite{cecotti2009bps}, it would be interesting to clarify this point. 
We restrict ourselves to open Riemann surfaces with only one boundary component. Given a class $\beta \in H_2(X,L)$, let $N_{g,\beta} \in \Q$ be the ``count'' of holomorphic maps $\varphi \colon (C,\partial C) \rightarrow (X,L)$ with $C$ a genus $g$ Riemann surface with one boundary component and $[\varphi(C, \partial C)]=\beta$. We write
\[ \partial \beta =[\partial C] \in H_1(L)\,,\]
to denote the image of $\beta$ by the natural 
boundary map $H_2(X,L) \rightarrow H_1(L)$.
A 
holomorphic map $\varphi \colon (C,\partial C) \rightarrow (X,L)$ of class $\beta \in H_2(X,L)$ is a $J_{e^{i \theta}}$-holomorphic map to 
$U$, at a constant value $e^{i \theta} \in S^1$, where 
$\theta$ is the argument of $\int_\beta \Omega_I$.

According to Witten \cite{MR1362846}, 
one should encode these counts of holomorphic maps as deformations of
Chern-Simons theory of gauge group $U(1)$ on $L$ . The field of this
theory is a $U(1)$ gauge field $A$ and its action is 
\[I_{CS}(A) \coloneqq \frac{1}{2} \int_L A \wedge dA \,.\]
According to \mbox{Section 4.4} of \cite{MR1362846},
this Chern-Simons action is deformed by additional terms involving the counts of holomorphic maps:
\[ I(A)=I_{CS}(A)
+ 
\sum_{\beta} \sum_{g \geqslant 0} N_{g,\beta}
\hbar^{2g} e^{-\int_\beta \omega}
e^{\int_{\partial \beta} A} \,.\]
The partition function of the deformed theory can 
be written as a correlation function in Chern-Simons theory 
\[ Z = \int DA \, e^{i \frac{I(A)}{\hbar}} \] 
\[=\Bigg\langle \exp \left( i \sum_{\beta
\in H_2(X,L)} \sum_{g \geqslant 0} N_{g,\beta}
\hbar^{2g-1} e^{-\int_\beta \omega}
e^{\int_{\partial \beta} A} \right) \Bigg\rangle_{CS} \,.\]
As $L=\Sigma \times S^1$, we can adopt a Hamiltonian description where $S^1$ plays the role of the time direction. The classical phase space of $U(1)$ Chern-Simons theory
on $L=\Sigma \times S^1$ 
is the space 
of $U(1)$ flat connections on $\Sigma$. When $\Sigma$ is a torus, the classical phase space is the dual torus $T$. For every
$m \in H_1(L)$, the holonomy around $m$ defined a function 
$z^m$ on $T$, ie a classical observable,
\[ z^m(A) \coloneqq e^{\int_m A} \,.\] 
The algebra structure is given by 
$z^m z^{m'}=z^{m+m'}$ and the Poisson structure by 
$\{z^m, z^{m'}\}=\langle m, m' \rangle z^{m+m'}$.
The algebra of quantum observables is given by the non-commutative
torus, $\hat{z}^m \hat{z}^{m'}=q^{\frac{1}{2} \langle m,m' \rangle}
\hat{z}^{m+m'}$, where $q=e^{i \hbar}$.
Writing $t^\beta = e^{-\int_\beta \omega}$, we obtain
\[Z=  \Tr_{\cH}  \left( T \prod_{\beta \in H_2(X,L)}
\Ad_{\exp \left( -i \sum_{g \geqslant 0} N_{g,\beta}
\hbar^{2g-1}t^\beta \hat{z}^m \right)} \right)\,,\]
where $\cH$ is the Hilbert space of quantum Chern-Simons
theory and
where $T \prod_\beta$ is a time ordered product, with
ordering according to the phase of 
$\int_\beta \Omega_I$.

The key physical input used by Cecotti-Vafa
\cite{cecotti2009bps}
is the continuity of the partition function $Z$
as function of the position of $L$ in $X$. 
It follows that the jump of the invariants 
$N_{g,\beta}$ under variation of $L$ in $X$
is controlled by the refined Kontsevich-Soibelman
wall-crossing formula formulated in terms of products of 
automorphisms of the quantum torus.

\subsection{Comparison with Theorem \ref{precise_main_thm_ch2}}
\label{comparison}

Our main result, Theorem 
\ref{precise_main_thm_ch2}, expresses the log Gromov-Witten theory 
of a log Calabi-Yau surface $(Y_{\fm}, \partial Y_{\fm})$
in terms of a 2-dimensional Kontsevich-Soibelman scattering diagram. The complement 
$U_{\fm} \coloneqq Y_{\fm}-\partial Y_{\fm}$
is a non-compact holomorphic symplectic surface admitting a
SYZ
real Lagrangian torus fibration.
In some cases, $U_{\fm}$ admits a hyperk\"ahler metric,
such that the original complex structure of $U_{\fm}$
is the compatible complex structure $J$, and such that 
the SYZ fibration becomes $I$-holomorphic Lagrangian. 
Typical examples include 2-dimensional Hitchin moduli spaces, see
\cite{boalch2012hyperkahler}
for a nice review.
In such cases, we can apply the Cecotti-Vafa story 
summarized above to $U \coloneqq U_{\fm}$, with 
$\Sigma$ a torus fiber of the SYZ fibration. 

The log Gromov-Witten invariants with insertion of a top lambda class $N_{g,\beta}$, introduced in Section
\ref{section_log_calabi_yau}, should be viewed as a rigorous definition of the open Gromov-Witten invariants in the twistorial geometry $X$, with 
boundary on a torus fiber $\Sigma$ ``near infinity''.
An early reference for the interpretation of some open Gromov-Witten invariants in terms of relative stable maps is \cite{MR2402819}.
The intuitive picture to have in mind is that an open Riemann surface with a boundary on a torus fiber very close to the divisor at infinity can be capped off by a holomorphic disc meeting the divisor at infinity in one point.
This is in part justified by the 3-dimensional interpretation of the invariants $N_{g,\beta}^{Y_\fm}$
given in Section \ref{section_3d} and in particular by Lemma \ref{lem_3d_2}.

Automorphisms of the quantum torus 
appearing in 
Section \ref{summary_cecotti_vafa}
coincide with the automorphisms of the quantum 
torus appearing in 
Theorem \ref{precise_main_thm_ch2}.
It follows that Theorem \ref{precise_main_thm_ch2} 
can be viewed as a mathematically rigorous check 
of the physical argument given by 
Cecotti-Vafa
\cite{cecotti2009bps}, based on the continuity of Chern-Simons correlation functions and on the connection predicted by Witten
\cite{MR1362846}
between higher genus open Gromov-Witten invariants and quantum Chern-Simons theory.

Finally, Ooguri-Vafa \cite{MR1765411} have given a physical derivation of an integrality result for open Gromov-Witten invariants of Calabi-Yau 3-folds, 
parallel 
to the Gopakumar-Vafa 
\cite{gopakumar1998m1} 
\cite{gopakumar1998m2}
integrality for closed 
Gromov-Witten invariants of Calabi-Yau 3-folds. 
Given the heuristic interpretation of the log Gromov-Witten invariants $N_{g,\beta}$ as open Gromov-Witten invariants,
this integrality coincides with the integrality of 
\mbox{Conjecture \ref{conj_integrality}}
and Theorem \ref{thm_integ}.

\vspace{+8 pt}
\noindent
Department of Mathematics \\
Imperial College London \\
pierrick.bousseau12@imperial.ac.uk


\newpage

\begin{thebibliography}{MNOP06b}

\bibitem[AC14]{MR3257836}
D.~Abramovich and Q.~Chen.
\newblock Stable logarithmic maps to {D}eligne-{F}altings pairs {II}.
\newblock {\em Asian J. Math.}, 18(3):465--488, 2014.

\bibitem[ACGS17a]{abramovich2017decomposition}
D.~Abramovich, Q.~Chen, M.~Gross, and B.~Siebert.
\newblock Decomposition of degenerate {G}romov-{W}itten invariants.
\newblock {\em arXiv preprint arXiv:1709.09864 v1}, 2017.

\bibitem[ACGS17b]{abramovich2017punctured}
D.~Abramovich, Q.~Chen, M.~Gross, and B.~Siebert.
\newblock Punctured logarithmic curves.
\newblock {\em preprint, available on the webpage of Mark Gross}, 2017.

\bibitem[AGV08]{MR2450211}
D.~Abramovich, T.~Graber, and A.~Vistoli.
\newblock Gromov-{W}itten theory of {D}eligne-{M}umford stacks.
\newblock {\em Amer. J. Math.}, 130(5):1337--1398, 2008.

\bibitem[AMW14]{abramovich2012comparison}
D.~Abramovich, S.~Marcus, and J.~Wise.
\newblock Comparison theorems for {G}romov-{W}itten invariants of smooth pairs
  and of degenerations.
\newblock {\em Ann. Inst. Fourier (Grenoble)} 63, 1611--1667, 2014.

\bibitem[AW13]{MR3778185}
D.~Abramovich and J.~Wise.
\newblock Birational invariance in logarithmic {G}romov-{W}itten theory.
\newblock {\em Compos. Math.} 154(3):595--620, 2018.


\bibitem[Bea95]{MR1351502}
A.~Beauville.
\newblock Sur la cohomologie de certains espaces de modules de fibr\'es
  vectoriels.
\newblock In {\em Geometry and analysis ({B}ombay, 1992)}, pages 37--40. Tata
  Inst. Fund. Res., Bombay, 1995.

\bibitem[BG16]{MR3453390}
F.~Block and L.~G\"ottsche.
\newblock Refined curve counting with tropical geometry.
\newblock {\em Compos. Math.}, 152(1):115--151, 2016.


\bibitem[Boa12]{boalch2012hyperkahler}
P.~Boalch.
\newblock Hyperk{\"a}hler manifolds and nonabelian {H}odge theory of
  (irregular) curves.
\newblock {\em arXiv preprint arXiv:1203.6607}, 2012.

\bibitem[Bou19]{bousseau2017tropical}
P.~Bousseau.
\newblock Tropical refined curve counting from higher genera and lambda
  classes.
\newblock {\em Invent. Math.}, 215(1):1--79, 2019.

\bibitem[Bou20]{bousseau2018quantum}
P.~Bousseau.
\newblock Quantum mirrors of log {C}alabi-{Y}au surfaces and higher genus
  curves counting.
\newblock {\em Compositio Mathematica,
156(2), 360-411}, 2020.

\bibitem[Bri17]{bri17}
T.~Bridgeland.
\newblock Scattering diagrams, Hall algebras and stability conditions.
\newblock {\em Algebraic Geometry 4.5: 523-561}, 2017.


\bibitem[BGP08]{MR2357679}
J.~Bryan, T.~Graber, and R.~Pandharipande.
\newblock The orbifold quantum cohomology of {$\C^2/\Z_3$} and
  {H}urwitz-{H}odge integrals.
\newblock {\em J. Algebraic Geom.}, 17(1):1--28, 2008.


\bibitem[BP05]{MR2115262}
J.~Bryan and R.~Pandharipande.
\newblock Curves in {C}alabi-{Y}au threefolds and topological quantum field
  theory.
\newblock {\em Duke Math. J.}, 126(2):369--396, 2005.

\bibitem[Che14]{MR3224717}
Q.~Chen.
\newblock Stable logarithmic maps to {D}eligne-{F}altings pairs {I}.
\newblock {\em Ann. of Math. (2)}, 180(2):455--521, 2014.

\bibitem[CNV10]{cecotti2010r}
S.~Cecotti, A.~Neitzke, and C.~Vafa.
\newblock {R}-twisting and 4d/2d correspondences.
\newblock {\em arXiv preprint arXiv:1006.3435}, 2010.

\bibitem[CR02]{MR1950941}
W.~Chen and Y.~Ruan.
\newblock Orbifold {G}romov-{W}itten theory.
\newblock In {\em Orbifolds in mathematics and physics ({M}adison, {WI},
  2001)}, volume 310 of {\em Contemp. Math.}, pages 25--85. Amer. Math. Soc.,
  Providence, RI, 2002.

\bibitem[CV09]{cecotti2009bps}
S.~Cecotti and C.~Vafa.
\newblock {BPS} wall crossing and topological strings.
\newblock {\em arXiv preprint arXiv:0910.2615}, 2009.

\bibitem[ESm93]{MR1228610}
G.~Ellingsrud and S.~A.~Str\o~mme.
\newblock Towards the {C}how ring of the {H}ilbert scheme of {${\bf P}^2$}.
\newblock {\em J. Reine Angew. Math.}, 441:33--44, 1993.

\bibitem[FK94]{MR1264393}
L.~D. Faddeev and R.~M. Kashaev.
\newblock Quantum dilogarithm.
\newblock {\em Modern Phys. Lett. A}, 9(5):427--434, 1994.

\bibitem[FP00]{MR1728879}
C.~Faber and R.~Pandharipande.
\newblock Hodge integrals and {G}romov-{W}itten theory.
\newblock {\em Invent. Math.}, 139(1):173--199, 2000.

\bibitem[FS15]{MR3383167}
S.~A. Filippini and J.~Stoppa.
\newblock Block-{G}\"ottsche invariants from wall-crossing.
\newblock {\em Compos. Math.}, 151(8):1543--1567, 2015.

\bibitem[Fuk05]{MR2131017}
K.~Fukaya.
\newblock Multivalued {M}orse theory, asymptotic analysis and mirror symmetry.
\newblock In {\em Graphs and patterns in mathematics and theoretical physics},
  volume~73 of {\em Proc. Sympos. Pure Math.}, pages 205--278. Amer. Math.
  Soc., Providence, RI, 2005.

\bibitem[Ful98]{MR1644323}
W.~Fulton.
\newblock {\em Intersection theory}, volume~2 of {\em Ergebnisse der Mathematik
  und ihrer Grenzgebiete. 3. Folge. A Series of Modern Surveys in Mathematics
  [Results in Mathematics and Related Areas. 3rd Series. A Series of Modern
  Surveys in Mathematics]}.
\newblock Springer-Verlag, Berlin, second edition, 1998.


\bibitem[vGGR19]{van2017local}
M.~van Garrel, T.~Graber, and H.~Ruddat.
\newblock Local {G}romov-{W}itten {I}nvariants are {L}og {I}nvariants.
\newblock {Adv. Math.} 359, 860–-876, 2019.


\bibitem[GHK15]{MR3415066}
M.~Gross, P.~Hacking, and S.~Keel.
\newblock Mirror symmetry for log {C}alabi-{Y}au surfaces {I}.
\newblock {\em Publ. Math. Inst. Hautes \'Etudes Sci.}, 122:65--168, 2015.

\bibitem[GHKK18]{MR3758151}
M.~Gross, P.~Hacking, S.~Keel, and M.~Kontsevich.
\newblock Canonical bases for cluster algebras.
\newblock {\em J. Amer. Math. Soc.}, 31(2):497--608, 2018.

\bibitem[GP99]{MR1666787}
T.~Graber and R.~Pandharipande.
\newblock Localization of virtual classes.
\newblock {\em Invent. Math.}, 135(2):487--518, 1999.

\bibitem[GP10]{MR2662867}
M.~Gross and R.~Pandharipande.
\newblock Quivers, curves, and the tropical vertex.
\newblock {\em Port. Math.}, 67(2):211--259, 2010.

\bibitem[GPS10]{MR2667135}
M.~Gross, R.~Pandharipande, and B.~Siebert.
\newblock The tropical vertex.
\newblock {\em Duke Math. J.}, 153(2):297--362, 2010.

\bibitem[GS11]{MR2846484}
M.~Gross and B.~Siebert.
\newblock From real affine geometry to complex geometry.
\newblock {\em Ann. of Math. (2)}, 174(3):1301--1428, 2011.

\bibitem[GS13]{MR3011419}
M.~Gross and B.~Siebert.
\newblock Logarithmic {G}romov-{W}itten invariants.
\newblock {\em J. Amer. Math. Soc.}, 26(2):451--510, 2013.

\bibitem[GV98a]{gopakumar1998m1}
R.~Gopakumar and C.~Vafa.
\newblock M-{T}heory and {T}opological {S}trings--{I}.
\newblock {\em arXiv preprint hep-th/9809187}, 1998.

\bibitem[GV98b]{gopakumar1998m2}
R.~Gopakumar and C.~Vafa.
\newblock M-{T}heory and {T}opological {S}trings--{II}.
\newblock {\em arXiv preprint hep-th/9812127}, 1998.

\bibitem[Iac17]{iacovino2017ks}
V.~Iacovino.
\newblock Kontsevich-{S}oibelman {W}all {C}rossing {F}ormula and {H}olomorphic
  {D}isks.
\newblock {\em arXiv preprint arXiv:1711.05306}, 2017.

\bibitem[IM13]{MR3142257}
I.~Itenberg and G.~Mikhalkin.
\newblock On {B}lock-{G}\"ottsche multiplicities for planar tropical curves.
\newblock {\em Int. Math. Res. Not. IMRN}, (23):5289--5320, 2013.

\bibitem[JPT11]{MR2785870}
P.~Johnson, R.~Pandharipande, and H.-H. Tseng.
\newblock Abelian {H}urwitz-{H}odge integrals.
\newblock {\em Michigan Math. J.}, 60(1):171--198, 2011.

\bibitem[JS12]{MR2951762}
D.~Joyce and Y.~Song.
\newblock A theory of generalized {D}onaldson-{T}homas invariants.
\newblock {\em Mem. Amer. Math. Soc.}, 217(1020):iv+199, 2012.

\bibitem[Kat89]{MR1463703}
K.~Kato.
\newblock Logarithmic structures of {F}ontaine-{I}llusie.
\newblock In {\em Algebraic analysis, geometry, and number theory ({B}altimore,
  {MD}, 1988)}, pages 191--224. Johns Hopkins Univ. Press, Baltimore, MD, 1989.

\bibitem[Kat08]{MR2420017}
S.~Katz.
\newblock Genus zero {G}opakumar-{V}afa invariants of contractible curves.
\newblock {\em J. Differential Geom.}, 79(2):185--195, 2008.

\bibitem[KLR18]{kim2018degeneration}
B.~Kim, H.~Lho, and H.~Ruddat.
\newblock The degeneration formula for stable log maps.
\newblock {\em arXiv preprint arXiv:1803.04210}, 2018.

\bibitem[Kon06]{MR2250076}
Y.~Konishi.
\newblock Integrality of {G}opakumar-{V}afa invariants of toric {C}alabi-{Y}au
  threefolds.
\newblock {\em Publ. Res. Inst. Math. Sci.}, 42(2):605--648, 2006.

\bibitem[KRSS17a]{kucharski2017bps}
P.~Kucharski, M.~Reineke, M.~Sto{\v{s}}i{\'c}, and P.~Su{\l}kowski.
\newblock {BPS} states, knots, and quivers.
\newblock {\em Physical Review D}, 96(12):121902, 2017.

\bibitem[KRSS17b]{kucharski2017knots}
P.~Kucharski, M.~Reineke, M.~Stosic, and P.~Su{\l}kowski.
\newblock Knots-quivers correspondence.
\newblock {\em arXiv preprint arXiv:1707.04017}, 2017.

\bibitem[KS06]{MR2181810}
M.~Kontsevich and Y.~Soibelman.
\newblock Affine structures and non-{A}rchimedean analytic spaces.
\newblock In {\em The unity of mathematics}, volume 244 of {\em Progr. Math.},
  pages 321--385. Birkh\"auser Boston, Boston, MA, 2006.

\bibitem[KS08]{kontsevich2008stability}
M.~Kontsevich and Y.~Soibelman.
\newblock Stability structures, motivic {D}onaldson-{T}homas invariants and
  cluster transformations.
\newblock {\em arXiv preprint arXiv:0811.2435}, 2008.

\bibitem[KS11]{MR2851153}
M.~Kontsevich and Y.~Soibelman.
\newblock Cohomological {H}all algebra, exponential {H}odge structures and
  motivic {D}onaldson-{T}homas invariants.
\newblock {\em Commun. Number Theory Phys.}, 5(2):231--352, 2011.

\bibitem[KS13]{kontsevich2013wall}
M.~Kontsevich and Y.~Soibelman.
\newblock Wall-crossing structures in {D}onaldson-{T}homas invariants,
  integrable systems and {M}irror {S}ymmetry.
\newblock {\em arXiv preprint arXiv:1303.3253}, 2013.

\bibitem[Lin17]{lin2017correspondence}
Y.-S.~Lin.
\newblock Correspondence {T}heorem between {H}olomorphic {D}iscs and {T}ropical
  {D}iscs on {K3} {S}urfaces.
\newblock {\em arXiv preprint arXiv:1703.00411}, 2017.

\bibitem[LS06]{MR2402819}
J.~Li and Y.~Song.
\newblock Open string instantons and relative stable morphisms.
\newblock In {\em The interaction of finite-type and {G}romov-{W}itten
  invariants ({BIRS} 2003)}, volume~8 of {\em Geom. Topol. Monogr.}, pages
  49--72. Geom. Topol. Publ., Coventry, 2006.

\bibitem[Man15]{mandel2015refined}
T.~Mandel.
\newblock Scattering diagrams, theta functions, and refined tropical curve counts.
\newblock {\em arXiv preprint arXiv:1503.06183}, 2015.

\bibitem[MPS13]{MR3080495}
J.~Manschot, B.~Pioline, and A.~Sen.
\newblock On the {C}oulomb and {H}iggs branch formulae for multi-centered black
  holes and quiver invariants.
\newblock {\em J. High Energy Phys.}, (5):166, front matter+42, 2013.



\bibitem[Mar07]{MR2304330}
E.~Markman.
\newblock Integral generators for the cohomology ring of moduli spaces of
  sheaves over {P}oisson surfaces.
\newblock {\em Adv. Math.}, 208(2):622--646, 2007.


\bibitem[MNOP06a]{MR2264664}
D.~Maulik, N.~Nekrasov, A.~Okounkov, and R.~Pandharipande.
\newblock Gromov-{W}itten theory and {D}onaldson-{T}homas theory. {I}.
\newblock {\em Compos. Math.}, 142(5):1263--1285, 2006.

\bibitem[MNOP06b]{MR2264665}
D.~Maulik, N.~Nekrasov, A.~Okounkov, and R.~Pandharipande.
\newblock Gromov-{W}itten theory and {D}onaldson-{T}homas theory. {II}.
\newblock {\em Compos. Math.}, 142(5):1286--1304, 2006.



\bibitem[MPT10]{MR2746343}
D.~Maulik, R.~Pandharipande, and R.~P. Thomas.
\newblock Curves on {$K3$} surfaces and modular forms.
\newblock {\em J. Topol.}, 3(4):937--996, 2010.
\newblock With an appendix by A. Pixton.

\bibitem[MR17]{meinhardt2017donaldson}
S.~Meinhardt and M.~Reineke.
\newblock {D}onaldson--{T}homas invariants versus intersection cohomology of
  quiver moduli.
\newblock {J. Reine Angew. Math.}, 754, 143-–178, 2019.

\bibitem[Mik05]{MR2137980}
G.~Mikhalkin.
\newblock Enumerative tropical algebraic geometry in {$\R^2$}.
\newblock {\em J. Amer. Math. Soc.}, 18(2):313--377, 2005.



\bibitem[Mum83]{MR717614}
D.~Mumford.
\newblock Towards an enumerative geometry of the moduli space of curves.
\newblock In {\em Arithmetic and geometry, {V}ol. {II}}, volume~36 of {\em
  Progr. Math.}, pages 271--328. Birkh\"auser Boston, Boston, MA, 1983.

\bibitem[Nei14]{neitzke2014comparing}
A.~Neitzke.
\newblock Comparing signs in wall-crossing formulas. \\
\newblock {\em https://www.ma.utexas.edu/users/neitzke/expos/sign-expos.pdf},
  2014.

\bibitem[NS06]{MR2259922}
T.~Nishinou and B.~Siebert.
\newblock Toric degenerations of toric varieties and tropical curves.
\newblock {\em Duke Math. J.}, 135(1):1--51, 2006.

\bibitem[OV00]{MR1765411}
H.~Ooguri and C.~Vafa.
\newblock Knot invariants and topological strings.
\newblock {\em Nuclear Phys. B}, 577(3):419--438, 2000.

\bibitem[PP13]{MR3127827}
R.~Pandharipande and A.~Pixton.
\newblock Descendent theory for stable pairs on toric 3-folds.
\newblock {\em J. Math. Soc. Japan}, 65(4):1337--1372, 2013.

\bibitem[Rei10]{MR2650811}
M.~Reineke.
\newblock Poisson automorphisms and quiver moduli.
\newblock {\em J. Inst. Math. Jussieu}, 9(3):653--667, 2010.

\bibitem[Rei11]{MR2801406}
M.~Reineke.
\newblock Cohomology of quiver moduli, functional equations, and integrality of
  {D}onaldson-{T}homas type invariants.
\newblock {\em Compos. Math.}, 147(3):943--964, 2011.

\bibitem[RSW12]{MR3033514}
M.~Reineke, J.~Stoppa, and T.~Weist.
\newblock M{PS} degeneration formula for quiver moduli and refined
  {GW}/{K}ronecker correspondence.
\newblock {\em Geom. Topol.}, 16(4):2097--2134, 2012.

\bibitem[RW13]{MR3004575}
M.~Reineke and T.~Weist.
\newblock Refined {GW}/{K}ronecker correspondence.
\newblock {\em Math. Ann.}, 355(1):17--56, 2013.

\bibitem[Soi09]{MR2596639}
Y.~Soibelman.
\newblock On non-commutative analytic spaces over non-{A}rchimedean fields.
\newblock In {\em Homological mirror symmetry}, volume 757 of {\em Lecture
  Notes in Phys.}, pages 221--247. Springer, Berlin, 2009.

\bibitem[SYZ96]{MR1429831}
A.~Strominger, S.-T.~Yau, and E.~Zaslow.
\newblock Mirror symmetry is {$T$}-duality.
\newblock {\em Nuclear Phys. B}, 479(1-2):243--259, 1996.

\bibitem[Tod12]{MR2892766}
Y.~Toda.
\newblock Stability conditions and curve counting invariants on {C}alabi-{Y}au
  3-folds.
\newblock {\em Kyoto J. Math.}, 52(1):1--50, 2012.

\bibitem[Wit95]{MR1362846}
E.~Witten.
\newblock Chern-{S}imons gauge theory as a string theory.
\newblock In {\em The {F}loer memorial volume}, volume 133 of {\em Progr.
  Math.}, pages 637--678. Birkh\"auser, Basel, 1995.

\bibitem[Zag07]{MR2290758}
D.~Zagier.
\newblock The dilogarithm function.
\newblock In {\em Frontiers in number theory, physics, and geometry. {II}},
  pages 3--65. Springer, Berlin, 2007.

\bibitem[Zas18]{zaslow2018wavefunctions}
E.~Zaslow.
\newblock Wavefunctions for a {C}lass of {B}ranes in {T}hreespace.
\newblock {\em arXiv preprint arXiv:1803.02462}, 2018.

\end{thebibliography}
\end{document}